\documentclass[12pt]{amsart}

\usepackage[a4paper]{geometry}
\usepackage{amsmath, amsthm}
\usepackage{amssymb, amsfonts}
\usepackage{hyperref}
\usepackage{tikz}
\usepackage{bm}

\usepackage{microtype}
\usepackage{enumerate}

\usepackage{stmaryrd}
\DeclareSymbolFont{bbold}{U}{bbold}{m}{n}
\DeclareSymbolFontAlphabet{\mathbbm}{bbold}

\title{Category forcing and generic absoluteness}

\theoremstyle{plain}
	\newtheorem{theorem}{Theorem}[section]
	\newtheorem{proposition}[theorem]{Proposition}
	\newtheorem{lemma}[theorem]{Lemma}
	\newtheorem{corollary}[theorem]{Corollary}
	\newtheorem{fact}[theorem]{Fact}

	\newtheorem{claim}{Claim}
	\newtheorem{subclaim}{Subclaim}
\theoremstyle{definition}
	\newtheorem{definition}[theorem]{Definition}
	\newtheorem{notation}[theorem]{Notation}
	
	\newtheorem{notation*}{Notation}

\theoremstyle{remark}
	\newtheorem{remark}[theorem]{Remark}

\newcommand{\id}{\ensuremath{\mathrm{Id}}}
\newcommand{\Ord}{\ensuremath{\mathrm{Ord}}}

\newcommand{\ZFC}{\ensuremath{\mathsf{ZFC}}}
\newcommand{\ZF}{\ensuremath{\mathsf{ZF}}}
\newcommand{\MK}{\ensuremath{\mathsf{MK}}}

\DeclareMathOperator{\dom}{dom}

\DeclareMathOperator{\cof}{cof}

\DeclareMathOperator{\coker}{coker}
\DeclareMathOperator{\supp}{supp}
\DeclareMathOperator{\trcl}{trcl}

\DeclareMathOperator{\Coll}{Coll}

\DeclareMathOperator{\RO}{\mathsf{RO}}

\newcommand{\NS}{\ensuremath{\mathbf{NS}}} %non stationary ideal
\newcommand{\bool}[1]{\mathsf{#1}}

\newcommand{\FFF}{\mathcal{F}}

\newcommand{\pow}[1]{\mathcal{P}\left(#1\right)}

\newcommand{\Reg}[1]{\text{Reg}\left(#1\right)}
\newcommand{\rest}[1]{\upharpoonright_{#1}}

\newcommand{\qp}[1]{\left[ #1 \right]}
\newcommand{\Qp}[1]{\left\llbracket #1 \right\rrbracket}
\newcommand{\ap}[1]{\langle #1 \rangle}
\newcommand{\bp}[1]{\left\lbrace #1 \right\rbrace}

\newcommand{\AC}{\ensuremath{\text{{\sf AC}}}} 
 
\newcommand{\SP}{\ensuremath{\text{{\sf SP}}}}
\newcommand{\SSP}{\ensuremath{\text{{\sf SSP}}}}
\newcommand{\PR}{\ensuremath{\text{{\sf PR}}}}
\newcommand{\SPFA}{\ensuremath{\text{{\sf SPFA}}}}

\newcommand{\FA}{\ensuremath{\text{{\sf FA}}}}

\newcommand{\CFA}{{\sf CFA}}
\newcommand{\BCFA}{{\sf BCFA}}

\newcommand{\1}{\mathsf{1}}
\newcommand{\2}{\mathsf{2}}

\theoremstyle{definition}
	\newtheorem{question}[theorem]{Question}

\newcommand{\UnderTilde}[1]{{\setbox1=\hbox{$#1$}\baselineskip=0pt\vtop{\hbox{$#1$}\hbox to\wd1{\hfil$\sim$\hfil}}}{}}

\newcommand{\MA}{\ensuremath{\text{{\sf MA}}}}

\newcommand{\BFA}{\ensuremath{\text{{\sf BFA}}}}

\DeclareMathOperator{\Add}{Add}

\newcommand{\dirlim}{\varinjlim}
\newcommand{\invlim}{\varprojlim}
\newcommand{\rcslim}{\lim\limits_{\sf rcs}}

\newcommand{\SRP}{\ensuremath{\text{{\sf SRP}}}}

\let \restr = \upharpoonright

\let \into = \longrightarrow

\let \sub = \subseteq
\let \elsub = \preccurlyeq

\let \av = \arrowvert

\let \a = \alpha
\let \b = \beta
\let \g = \gamma
\let \d = \delta

\let \l = \lambda
\let \k = \kappa
\let \m = \mu
\let \n = \nu
\let \p = \pi

\let \t = \theta

\let \s = \sigma
\let \x = \xi

\let \o = \omega

\let \P = \Pi
\let \S = \Sigma

\let \al = \aleph
\let \la = \langle
\let \ra = \rangle
\let \mtcl = \mathcal
\let \mtbb = \mathbb
\let \it = \item

\DeclareMathOperator{\range}{range}
\DeclareMathOperator{\cf}{cf}
\DeclareMathOperator{\ot}{ot}
\DeclareMathOperator{\Lim}{Lim}
\DeclareMathOperator{\V}{\mathbf V}

\DeclareMathOperator{\MRP}{\textsf{MRP}}
\DeclareMathOperator{\CH}{\textsf{CH}}

\DeclareMathOperator{\BPFA}{\textsf{BPFA}}

\DeclareMathOperator{\STP}{\textsf{STP}}
\DeclareMathOperator{\TWCG}{\textsf{TWCG}}
\DeclareMathOperator{\TCG}{\textsf{TCG}}

\usepackage[mode=multiuser,status=draft,lang=english]{fixme}
\fxsetup{theme=colorsig}

\FXRegisterAuthor{M}{aM}{M}
\FXRegisterAuthor{D}{aD}{David}

\title{Incompatible bounded category forcing axioms}

\author[D.\ Asper\'o]{David Asper\'o}

\address{David Asper\'{o}, School of Mathematics, University of East Anglia, Norwich NR4 7TJ, UK}

\email{d.aspero@uea.ac.uk}

\author[M.\ Viale]{Matteo Viale}

\address{Matteo Viale, Department of Mathematics ``Giuseppe Peano'', Univ.\ Torino, via Carlo Alberto 10, 10125, Torino, Italy}

\email{matteo.viale@unito.it}

\date{}

\begin{document}

\thanks{The first author acknowledges support of EPSRC Grant EP/N032160/1.
The second author acknowledges support of GNSAGA. 
Parts of this research was partially done whilst the authors were visiting fellows at the Isaac Newton Institute for Mathematical Sciences in the programme `Mathematical, Foundational and Computational Aspects of the Higher Infinite' (HIF)}


\maketitle
\pagestyle{myheadings}\markright{Incompatible bounded category forcing axioms}

\begin{abstract} 
%We introduce category forcing axioms for well-behaved classes $\Gamma$. These are strong forcing axioms which completely decide 
%the theory of $H_{\lambda_\Gamma^+}$ (where $\lambda_\Gamma$ is a cardinal naturally associated to $\Gamma$) modulo generic extensions via forcing notions from $\Gamma$. 
%$\mathsf{MM}^{+++}$ was the first category forcing axiom to be isolated (by the second author). In this paper we present a general theory of category forcings, and prove the existence of $\al_1$-many pairwise incompatible category forcing axioms for $\o_1$-suitable classes. 
We introduce bounded category forcing axioms for well-behaved class\-es $\Gamma$. These are strong forms of bounded forcing axioms which completely decide the theory of some initial segment of the universe $H_{\lambda_\Gamma^+}$ modulo forcing in $\Gamma$, for some cardinal $\lambda_\Gamma$ naturally associated to $\Gamma$. These axioms naturally extend projective absoluteness for arbitrary set-forcing---in this situation $\lambda_\Gamma=\omega$---to classes $\Gamma$ with $\lambda_\Gamma>\omega$. Unlike projective absoluteness, these higher bounded category forcing axioms do not follow from large cardinal axioms, but can be forced under mild large cardinal assumptions on $V$. We also show the existence of many classes $\Gamma$ with $\lambda_\Gamma=\omega_1$, and giving rise to pairwise incompatible theories for $H_{\omega_2}$. 
\end{abstract}

%\listoffixmes
	
	\tableofcontents

	%\input{introduction4}
	
	% !TEX root = david-matteo.tex

\section{Introduction}
Forcing axioms are principles asserting the existence of sufficiently generic filters for all forcing notions in some reasonable class. In a general form they state, for a given class $\Gamma$ of forcing notions and a given cardinal $\kappa$, that for every $\mtcl P\in\Gamma$ and every collection $\mtcl D$ of ${\leq}\kappa$-many dense subsets of $\mtcl P$ there is a filter $G$ of $\mtcl P$ such that $G\cap D\neq\emptyset$ for each $D\in\mtcl D$.\footnote{This is equivalent to the axiom obtained by letting $\mtcl D$ consist of maximal antichains of $\mtcl P$.} We refer to this statement as $\FA(\Gamma)_\kappa$. %It is a matter of fact that 
These axioms are successful in settling a wide range of problems undecidable on the basis of the commonly accepted axioms for set theory. This is not surprising given that forcing axioms can be often characterized, for reasonable classes $\Gamma$, as \emph{maximality principles} with respect to generic extensions via forcing notions coming from $\Gamma$; specifically, the assertions that all statements in some reasonable class $\Sigma$ that can be forced over the universe by some forcing in $\Gamma$ are in fact true.  

For certain classes $\Sigma$ of statements, such a maximality principle amounts to asserting \emph{generic absoluteness} of the universe with respect to the relevant forcing extensions, i.e., the assertion that for every $\mtcl P\in\Gamma$, every $\mtcl P$-generic filter $G$ over $V$, and every $\sigma\in \Sigma$, $V\models\sigma$ if and only if $V[G]\models\sigma$.\footnote{There are of course interesting classes $\Sigma$ for which the relevant notion of maximality does not imply the corresponding notion of generic absoluteness. Take for example the class $\Sigma$ of all sentences of the form $H_{\omega_2}\models\sigma$, where $\sigma$ is a $\Pi_2$ sentence. Maximality for this class $\Sigma$ relative to all set-forcing extensions is consistent, as it follows from Woodin's $\mathbb P_{\mbox{max}}$ axiom $(\ast)$. On the other hand, generic absoluteness for $\Sigma$ is simply false as $\CH$ can be expressed by a sentence in $\Sigma$ and both $\CH$ and $\lnot\CH$ are forcible.}  This is the case, for example, if $\Sigma$ is the class of all $\Sigma_1$ sentences with parameters in $H_\lambda$, for a fixed cardinal $\lambda$, and all forcing notions in $\Gamma$ preserve the cardinality of $\lambda$.\footnote{When $\lambda$ is an infinite cardinal and $\Gamma$ is a class of forcing axioms preserving the cardinality of all cardinals $\mu\leq \lambda$, the \emph{bounded forcing axiom} $\BFA_\kappa(\Gamma)$ can often be in fact characterized as precisely this form of generic absoluteness for $\kappa=\lambda^+$. $\BFA_\kappa(\Gamma)$ is the axiom saying that for every $\mtcl P\in\Gamma$ and every collection $\mtcl A$ of maximal antichains of $\mtcl P$, if $|\mtcl A|\leq\kappa$ and $|A|\leq\kappa$ for every $A\in\mtcl A$, then there is a filter $G\sub\mtcl P$ such that $G\cap A\neq\emptyset$ for each $A\in\mtcl A$.} Another example is given by the class $\Sigma$ of all projective sentences, i.e., all sentences of the form $H_{\omega_1}\models\sigma$, where $\sigma$ is any sentence with parameters in $H_{\omega_1}$. It is a remarkable fact, due to Woodin, that the mere presence in $V$ of sufficiently strong large cardinals---for example a proper class of Woodin cardinals, or a supercompact cardinal---outright implies this principle of generic absoluteness, known as \emph{Projective Absoluteness}. Generic absoluteness is a naturally attractive possible feature of the universe, at least for the set-theorist with realist inclinations, in that it manages to neutralize, to some extent, the effects of forcing on the universe (forcing being our prime method for proving independence over our base set theory). 

It turns out that if we avail ourselves of the expressive power provided by second order set theory, we can make sense of hypotheses which in terms of consistency strength lie way below the existence of Woodin cardinals and which nevertheless suffice to produce models of Projective Absoluteness. To be specific, if we work in the extension of Morse-Kelley set theory ($\MK$) with the axiom saying that the class $\Ord$ of all ordinals is Mahlo (i.e., every class of ordinals which is closed and unbounded contains an inaccessible cardinal), then we may find a set-forcing extension of $V$ in which Projective Absoluteness holds. In fact, for every set-generic extension $V[G]$ of $V$ there is a set-generic extension of $V[G]$ satisfying Projective Absoluteness. Of course now we do not obtain Projective Absoluteness in $V$ but only prove that it is forcible. In order to even be able to state our starting hypothesis we need to transcend first order logic---which of course was enough to express the existence of a proper class of Woodin cardinals---and make use of the additional expressive power of second order logic. This is arguably a drawback of the present situation. Nevertheless, we will make free use of second order logic in this paper as doing so will allow us to make sense of situations for which this would not be possible otherwise. 

The construction referred to above originates in the classical argument for deriving Projective Absoluteness from large cardinals. The main observation is that if $\Ord$ is Mahlo, then there are unboundedly many inaccessible cardinals $\delta$ with the property that if $G$ is $\Coll(\omega,  {<}\delta)$-generic over $V$, then $H_{\omega_1}^{V[G]}$ ($=V_\delta[G]$) is an elementary substructure of $V[H]$ for some generic filter $H$ over $V$ for the class-forcing $\Coll(\omega, {<}\Ord)$ (which of course is the same as $\Coll(\omega, {<}\Ord)$ as computed in $V[G]$). Using the fact that every set-forcing extension can be absorbed in a forcing extension via $\Coll(\omega,\,\kappa)$, for some high enough cardinal $\kappa$, it follows that Projective Absoluteness holds in $V[G]$ in the above situation. This analysis shows in fact the following: If $\Ord$ is Mahlo, then Projective Absoluteness holds if and only if the identity on $H_{\omega_1}$ is an elementary embedding from this structure into $V^{\Coll(\omega, {<}\delta)}$. This justifies defining Projective Absoluteness to be precisely this axiom (which we do). We also write $\bool{PA}$ for Projective Absoluteness.

This paper can be naturally split in two parts. The main goal of the first part is to extend the above observation concerning $\bool{PA}$ and how to obtain it to fragments $H_\kappa$ of the universe beyond $H_{\omega_1}$. Let us assume that $\kappa$ is a successor cardinal, $\kappa=\lambda^+$.\footnote{As we will see, our methods naturally pertain to the theory of structures $H_\kappa$ for $\kappa$ being a successor cardinal.} Our guiding idea for obtaining analogues of $\bool{PA}$ applying to $H_\kappa$ is to focus on some class $\Gamma$ of forcing notions preserving all cardinals $\mu\leq\lambda$ and try to extend the methods in the argument for the forcibility of Projective Absoluteness in a suitable way to apply to forcings in $\Gamma$. 
We shall be dealing with fairly big classes $\Gamma$ for which there is an iterability theorem, for example the class of all proper forcings, or the class of all semiproper forcings. For these classes one cannot possibly hope to obtain generic absoluteness at the level of $H_{\omega_2}$ relative to \emph{all} extensions by members of $\Gamma$. For example there is always a proper poset forcing $\CH$ and there is always one  forcing $\lnot\CH$ and, as already mentioned, $\CH$ is expressible over $H_{\omega_2}$. We will instead aim at obtaining generic absoluteness relative to all forcing notions in a suitable class $\Gamma$ which, moreover, force the second order axiom corresponding to the axiom, in the $\bool{PA}$ situation, asserting that $H_{\omega_1}$ is an elementary substructure of the $\Coll(\omega, {<}\Ord)$-extension of $V$. 

For technical reasons it will be convenient to deal with classes consisting of complete Boolean algebras.\footnote{See for example ~\cite{VIAUSAX} for more details. There will be no real loss of generality in restricting the discussion to classes of complete Boolean algebras thanks to the fact that all classes we will be naturally interested in will be closed under taking regular open completions.} The right analogue, in this context, of the collapse $\Coll(\omega, {<}\Ord)$ in the $\bool{PA}$ argument turns out to be the class-forcing  whose conditions are all algebras in $\Gamma$, and where $\bool{C}\in\Gamma$ is stronger than $\bool{B}\in\Gamma$, which we will denote by $\bool{C}\leq_\Gamma\bool{B}$, if and only if there is a complete Boolean algebra homomorphism $i:\bool{B}\longrightarrow \bool{C}$ such that $\bool{B}$ forces the quotient algebra $\bool{C}/i[\dot G_{\bool{B}}]$ to be in the class $\Gamma$ as interpreted in $V^{\bool{B}}$.\footnote{We will in fact be working with definable classes $\Gamma$. In a statement like the one above we are of course really referring to some official definition of $\Gamma$.} Thus we are naturally seeing the category whose objects are all algebras in $\Gamma$, and whose arrows are homomorphisms $i:\bool{B}\longrightarrow\bool{C}$ of the above form, as a class-sized forcing notion.\footnote{As we will soon mention, we will call our axioms corresponding to these categories $\Gamma$ bounded \emph{category} forcing axioms.} 

Given a class $\Gamma$ of complete Boolean algebras, we will associate to $\Gamma$ a certain cardinal $\lambda_\Gamma$. This will be the supremum of  the class of all cardinals preserved by all members of $\Gamma$.\footnote{$\lambda_\Gamma$ could sometimes be all of $\Ord$ (for example if $\Gamma$ is the class of forcings with the countable chain condition), but in all classes we will consider $\lambda_\Gamma$ will be an actual cardinal (in this case it is of course the maximum cardinal preserved by all members of $\Gamma$).} In all classes $\Gamma$ we will consider in the second part of the paper, $\lambda_\Gamma=\omega_1$ (and so $H_{\lambda_\Gamma^+}$ will be $H_{\omega_2}$).  If $\Gamma$ has suitable nice properties in all extensions by members of $\Gamma$, then forcing with $(\Gamma, \leq_\Gamma)$ preserves all cardinals $\mu\leq\lambda_\Gamma$ and makes $\Ord$ equal to $\lambda_\Gamma^+$ (i.e., it forces $V=H_{\lambda_\Gamma^+}$), and every set in the extension $V[H]$ is already in $V[H\cap V_\delta]$ for some ordinal $\delta$ such that (the regular open completion of) $\Gamma\cap V_\delta$ is in $\Gamma$ and $H\cap V_\delta$ is $V$-generic for $\Gamma\cap V_\delta$. Furthermore, under suitable mild large cardinal assumptions---typically the same hypothesis we had for $\bool{PA}$, namely that $\Ord$ is Mahlo, suffices---we have that for every $\bool{B}\in\Gamma$ there is a $\bool{B}$-name $\dot{\bool{Q}}$ for an algebra in $V^{\bool{B}}$'s version of $\Gamma$ such that $\bool{C}=\bool{B}\ast\dot{\bool{Q}}$ is in $\Gamma$ and is such that if $G$ is $\bool{C}$-generic over $V$, then $H_{\lambda_\Gamma^+}^{V[G]}$ is an elementary substructure of $V[H]$ for every generic filter $H$ over $V$ for the category forcing $(\Gamma, \leq_\Gamma)$ as computed in $V[G]$. We will call classes $\Gamma$ satisfying all the relevant nice properties in all generic extensions by members of $\Gamma$ \emph{absolutely well-behaved}. By extending the $\bool{PA}$ argument we will prove, in addition, that if $\Ord$ is Mahlo, $\Gamma$ is absolutely well-behaved, and $H_{\lambda_\Gamma^+}$ is an elementary substructure of  $V[H]$ for every generic filter $H$ over $V$ for $(\Gamma, \leq_\Gamma)$, then the following is the case.

\begin{enumerate}
\item A strong form of the bounded forcing axiom $\BFA_{\lambda_\Gamma}(\Gamma)$ holds.\footnote{This part is easy.} 
\item If $G$ is a $V$-generic filter for some algebra in $\Gamma$ and $H_{\lambda_\Gamma^+}^{V[G]}$ is an elementary submodel of $V[H]$ for some generic filter $H$ over $V[G]$ for the category forcing $(\Gamma, \leq_\Gamma)$ as computed in $V[G]$, then $H_{\lambda_\Gamma^+}^V$ and $H_{\lambda_\Gamma^+}^{V[G]}$ have the same theory.\footnote{This is considerably more involved.} 
\end{enumerate}

It follows, from (1) and (2) above, together with the discussion before (1), that the second order axiom saying that $H_{\lambda_\Gamma^+}$ is an elementary substructure of $V[H]$ for some $V$-generic filter $H$ over $(\Gamma, \leq_\Gamma)$ is a strong form of the bounded forcing axiom $\BFA_{\lambda_\Gamma}(\Gamma)$ which achieves our goal (and which can actually be forced). We therefore call this second order axiom the \emph{bounded category forcing axiom for $\Gamma$}, and denote it by $\BCFA(\Gamma)$.\footnote{There is no need to specify $\lambda_\Gamma$ as this cardinal can be read off from $\Gamma$.} 

Our results in the first part of the present paper turn forcing from
a tool useful to prove undecidability results into a tool useful to prove theorems: In order to show that $\BCFA(\lambda_\Gamma)$, together with the ambient set theory, implies $H_{\lambda_\Gamma^+}\models\sigma$ for a given sentence $\sigma$, it suffices to show that $\BCFA(\Gamma)$, together with the ambient set theory, implies the existence of a forcing notion $\bool{B}$ in $\Gamma$ which forces both $\BCFA(\Gamma)$ and $H_{\lambda_\Gamma^+}\models\sigma$.

In the second part of the paper we isolate $\aleph_1$-many absolutely well-behaved classes $\Gamma$ of complete Boolean algebras, all of them with $\lambda_\Gamma=\omega_1$. Among these we have for example the class of all complete Boolean algebras which are proper forcing notions, the class of all complete Boolean algebras which are semiproper forcing notions, the class of all complete Boolean algebras which are proper forcing notions preserving Suslin trees, etc. The main point in this second part is to prove that all the corresponding bounded category forcing axioms are pairwise provably incompatible, sometimes in the presence of  extra mild large cardinal assumptions, regardless of the fact that we have $\Gamma_0\sub\Gamma_1$ for many choices of $\Gamma_0$ and $\Gamma_1$. For example, if $\Gamma_0$ and $\Gamma_1$ are, respectively,  the class of semiproper forcing notions and the class of proper forcing notions, we have that $\BCFA(\Gamma_0)$ and $\BCFA(\Gamma_1)$ are incompatible---assuming there is, for example, a measurable cardinal and an inaccessible cardinal $\delta$ such that $V_\delta\prec V$---despite the fact that $\Gamma_1\sub\Gamma_0$ and therefore $\BFA_{\aleph_1}(\Gamma_0)$ (which is implied by $\BCFA(\Gamma_0)$) implies $\BFA_{\aleph_1}(\Gamma_1)$ (which is implied by $\BCFA(\Gamma_1)$). Indeed, for this choice of $\Gamma_0$ and $\Gamma_1$ we have that if $\BCFA(\Gamma_0)$ holds and there is a measurable cardinal, then Club Bounding holds,\footnote{Club Bounding is the statement that for every function $f:\omega_1\longrightarrow\omega_1$ there is some $\alpha<\omega_2$ such that every canonical function for $\alpha$ bounds $f$ on a club.} whereas if $\BCFA(\Gamma_1)$ holds and there is an inaccessible cardinal $\delta$ such that $V_\delta\prec V$, then Club Bounding fails. 

Bounded category forcing axioms are to be seen as strong forms of bounded forcing axioms providing a picture of the universe as being saturated by \emph{only} forcing coming from the relevant class $\Gamma$. For classes $\Gamma_0\sub \Gamma_1$, even if it is of course true that $\BFA_{\lambda_\Gamma}(\Gamma_1)$ implies $\BFA_{\lambda_\Gamma}(\Gamma_0)$, there will typically\footnote{Our results in the second part of the paper provide some evidence that this will be the case for all choices of $\Gamma_0$ and $\Gamma_1$.} be statements $\sigma$ about $H_{\lambda_\Gamma^+}$ such that $\BFA_{\lambda_\Gamma}(\Gamma_1)$ implies $\sigma$, whereas $\lnot\sigma$ can be forced by a forcing in $\Gamma_0$ and, once forced, will be preserved by subsequent forcing in $\Gamma_0$. This explains why $\sigma$ will follow from $\BCFA(\Gamma_1)$ whereas it will fail in the $\BCFA(\Gamma_0)$ model.

One lesson to be learned from these incompatibility results is that natural forms of (bounded) forcing axioms do not, by themselves, favour a universist conception of set theory. There is unavoidable branching at the level of these axioms; in particular, the set-theorist with a universist mindset will need additional criteria---beyond the `naive' view on maximality provided by looking at the containment relation between the classes $\Gamma$ under consideration---to favour one of these axioms over the others. This should be compared with the fact that, as an empirical fact, consistent large cardinal axioms seem to be orderable under implication.\footnote{At least upwards directed, in the sense that for any two large cardinal axioms $A_1$ and $A_2$ there always seems to be a large cardinal axiom $A_3$ subsuming both $A_1$ and $A_2$. This seems to apply both to large cardinal axioms over $\ZFC$ and to the family of large cardinal axioms in the wider $\ZF$ context.} A reasonable additional criterion available to the universist when assessing bounded category forcing axiom---consistent with the view that these axioms are indeed strong forms of bounded forcing axioms---could be to focus on maximizing the class of $\Pi_2$ sentences over $H_{\lambda_\Gamma^+}$ implied by the axiom.\footnote{And, after all, the mathematical applicability of bounded forcing axioms is correlated to their $\Pi_2$ consequences for the theory of $H_{\lambda_\Gamma^+}$.} In this respect, when looking at classes $\Gamma$ with $\lambda_\Gamma=\omega_1$, the class that fares the best is of course the class of all forcing notions preserving stationary subsets of $\omega_1$ (we will denote this class by $\SSP$). This move of course amounts to ignoring the completeness granted by the axioms and instead focusing on the forcing axiom side of the axiom. When making it, we are probably taking the $\Pi_2$ maximality for $H_{\lambda_\Gamma^+}$ secured by the axiom as the main object of interest of strong (bounded) forcing axioms, and regard any possible completeness for the theory of $H_{\lambda_\Gamma^+}$ modulo forcing as a welcome extra feature, of foundational interest, that these axioms may have if they are strong enough.

\subsection{Model companionship versus bounded category forcing axioms}
We now point out the relevance of our results to the notions of model completeness and model companionship introduced by Robinson. There is a growing body of evidence relating forcing axioms, generic absoluteness results, and the generic 
multiverse to these model-theoretic notions. 
These notions describe, 
in a model-theoretic terminology 
%which is meaningful for 
applicable to an arbitrary first order theory $T$, 
%the first order 
the closure properties 
%with respect to first order logic 
of algebraically closed fields and the way the 
elementary class given by such fields 
sits inside the elementary class given by arbitrary fields. 
%This paper contributes significantly to unveil further this relationship.
%This paper contributes to the study of these ideas in the context of set theory and (fragments of) the set-theoretic universe. 
%significantly to unveil further this relationship.
Applied to the set-theoretic realm, the basic idea is that $H_{\omega_1}$ plays with respect to the generic multiverse the role $\mathbb{C}$ does for arbitrary fields, while (assuming forcing axioms)
$H_{\omega_2}$ plays this same role with respect to the generic multiverse given by forcings in an appropriate class
(e.g. proper, $\SSP$, etc).
 
 A first key notion in this context is that of existentially closed structures in a signature $\tau$.
A $\tau$-structure $\mathcal{M}$ is existentially closed in a superstructure $\mathcal{N}$ if 
$\mathcal{M}\prec_1 \mathcal{N}$. 
For a $\tau$-structure $\mathcal{M}$ with domain $M$, 
let us write $(M,\tau^M)$ as a shorthand for $(M,R^{\mathcal{M}}:R\in\tau)$.
% whenever
%$\mathcal{N}$ is a superstructure of $\mathcal{M}$ which models $T$.

Consider a signature $\tau_{\bool{ST}}$ in which one adds a predicate symbol $R_\phi$ of arity $n$ 
for any $\Delta_0$ formula $\phi(x_1,\dots,x_n)$ and interprets in the models of set theory these predicate symbols $R_\phi$ as the extension of the formula $\phi$.
In this set-up, one of the key consequences of Shoenfield's absoluteness is that 
$(H_{\omega_1}^V,\tau_{\bool{ST}}^V)\prec_1 (H_{\kappa}^{V[G]},\tau_{\bool{ST}}^{V[G]})$ whenever $G$ is a forcing extension of $V$
and $\kappa$ is an uncountable cardinal in $V[G]$; actually Shoenfield's absoluteness states that
$H_{\omega_1}^V$ is existentially closed in all of its well-founded 
$\tau_{\bool{ST}}$-superstructures which model $\ZFC^-$ (i.e. $\ZFC$ minus the power-set axiom).

Consider now bounded forcing axioms; these axioms state that 
$(H_{\omega_2}^V,\tau_{\bool{ST}}^V)\prec _1 (H_\kappa^{V[G]},\tau_{\bool{ST}}^{V[G]})$ whenever $G$ is $V$-generic for a forcing notion $P$ in a given class of forcings $\Gamma$ (stationary set preserving, proper, semiproper, etc) and 
$\kappa\geq\omega_2^{V[G]}$; once again, these axioms assert that $H_{\omega_2}^V$ is existentially closed in  its 
$\tau_{\bool{ST}}$-superstructures which model $\ZFC^-$ 
%(i.e. $\ZFC$ minus the power-set axiom)
 \emph{and} are obtained by 
\emph{certain types of} forcings.

%Towards this aim 
Let us explore briefly the notion of $T$-existentially closed structure and show how such structures are produced in model theory. 
Given a first order theory $T$ for a signature $\tau$, a $\tau$-structure 
$\mathcal{M}$ is $T$-e.c.\ if
$\mathcal{M}\prec_1\mathcal{N}$ whenever $\mathcal{N}$ is a superstructure of $\mathcal{M}$
which realizes $T_\forall$ (the universal fragment of $T$). 

Note that neither $\mathcal{M}$ nor $\mathcal{N}$ may be models of $T$, the only sure thing is that they model\footnote{For example $T$ is the theory of fields elementarily equivalent to 
$\mathbb{Q}$ in signature $\bp{+,\cdot,0,1}$, $K$ is algebraically closed, $L\supseteq K$ is a ring which is not a field.
Then $K$ is $T$-e.c., $K\prec_1 L$, both are models of $T_\forall$, but neither of them models $T$.} $T_\forall$.
A key (and not so trivial) fact which will play an essential role in our arguments is that whenever $\mathcal{M}$ is
$T$-e.c., so is $\mathcal{N}$ if $\mathcal{N}\prec_1\mathcal{M}$.
The standard example of a $T$-e.c.\ structure is an algebraically closed field, where $T$ is the theory of fields in signature 
$\bp{+,\cdot,0,1}$: if $K$ is algebraically closed, for any $L\supseteq K$
which models\footnote{Note that $L$ may not even be a ring, it is just a structure with no-zero divisors, where $+,\cdot$ satisfy the commutativity, associatitvity, and distributivity laws, and $0$ and $1$ are the neutral elements of $+,\cdot$. $L$ is a ring if it satisfies the $\Pi_2$-sentence stating that $(L,+,0)$ is a group.} $T_\forall$, any $\Sigma_1$ formula with parameters in $K$ (i.e. statements of the form
$\exists \vec{y} \qp{\bigwedge_{i=1}^n p_i(\vec{y})=0\wedge\bigwedge_{j=1}^m q_j(\vec{y})\neq 0}$ with $p_i,q_j$ polynomials with coefficients in $K$) realized in $L$ is already true in $K$.

How does one construct a $T$-e.c.\ structure? The simplest way is to start with a model $\mathcal{M}_0$ 
of $T_\forall$ 
%(the universal fragment of $T$) 
and construct using some book-keeping device a chain
$(\mathcal{M}_\alpha:\alpha<\kappa)$ of models of $T_\forall$ in which at each stage $\alpha$ one tries to make true in  
$\mathcal{M}_{\alpha+1}$ some existential formula with parameters 
in\footnote{Much in the same way one builds the algebraic closure of a field $L$ by passing
from $L=L_0$ to $L_1$ the field of fractions of $L_0[x]/p_0(x)$ (with $p_0$ an irreducible polynomial with coefficients in $L$), 
and then inductively from $L_n$ to the field of fractions of $L_{n}[x]/p_{n}(x)$ with the $p_n$s given by some book-keeping device.}
$\mathcal{M}_\alpha$. Note that there is tension between the constraint given by $\mathcal{M}_\alpha\models T_\forall$ and the requirement that $\mathcal{M}_{\alpha+1}$ realizes some $\Sigma_1$ formula with parameters in 
$\mathcal{M}_\alpha$ (for example this formula cannot be the negation of some universal axiom of $T$).

A key point is that an increasing chain of models of $T_\forall$ is still a model of $T_\forall$, hence the construction does not stop at limit levels; now if $\kappa$ is large enough, at stage $\kappa$ all existential 
formulae with parameters in some $\mathcal{M}_\alpha$ for $\alpha<\kappa$ have been realized (if possible) in some 
$\mathcal{M}_\beta$ with $\beta<\kappa$, hence $\mathcal{M}_\kappa$ is $T$-e.c.

Compare this procedure with the standard proof of the consistency of bounded forcing axioms: for example 
to establish the consistency of $\BPFA$ 
one does exactly the same, but now one starts from $\mathcal{M}_0=(V,\in)$ and in passing from 
$\mathcal{M}_\alpha=(V[G_\alpha],\in)$ to
$\mathcal{M}_{\alpha+1}=(V[G_{\alpha+1}],\in)$, one can use only forcings which are proper in $V[G_\alpha]$; moreover,
to catch one's tail one must continue the iteration for a $\kappa$ which is a reflecting cardinal. 
Also, in this case 
there is tension 
between realizing some new $\Sigma_1$ formula with parameters in $H_{\omega_2}^{V[G_\alpha]}$, and doing so by 
using a proper forcing (which preserves the $\Pi_1$ formula, for 
$\tau_{\bool{ST}}$ in parameter $S\subseteq\omega_1$,
\emph{$S$ is stationary}).

Now one of the guiding ideas of this paper is: what happens if we continue this iterated construction along all the ordinals?
If the ordinals are long enough, one should end up with a structure where all existential formulae which can be 
consistently made true by a proper forcing have been made true.
%In particular the final class-forcing extension $V[G]$ of $V$
%should be as ``existentially closed as possible with respect to proper forcings''. This is somewhat vague
%% fishy 
%since there are no more set-sized forcings over $V[G]$ for which one could test this property.
Note, in addition, 
%also 
that $V[G]$ should also be a structure 
which resembles $H_{\omega_2}$: on the one hand, every ordinal above $\omega_1^V$ will have been collapsed to $\omega_1^V$ at some point in the iteration;  
%the 
%%$\Sigma_1$-statement in parameters $\alpha,\omega_1$ 
%%\emph{$\alpha$ has size $\aleph_1$} 
%$\Sigma_1$ sentence in parameters $\alpha$ and $\omega_1$ saying that
%\emph{$\alpha$ has size $\aleph_1$} 
%will be forced by a proper forcing appearing in the $\Ord$-length iteration for all ordinals $\alpha$; 
on the other hand,
we should also expect that 
the regularity of $\Ord$ is preserved as all iterands are set-sized and, moreover, $\omega_1$ is preserved as the iteration is proper. Hence, the $\Ord$-length iteration is ${<}\Ord$-CC, $\omega_1$-preserving, and collapsing all uncountable ordinals to size $\aleph_1$, if things are properly organized.

In the above situation, we can look at whether there are stages $\alpha$ for which $H_{\omega_2}^{V[G_\alpha]}$ is an actual elementary substructure of $V[G]$. And if the ordinals are long enough (e.g., $\Ord$ is Mahlo), this will in fact be the case. 
%However, we can now resort to the fact that existentially closed structures are closed under $\Sigma_1$-substructures, 
%and formulate this maximality property as 
%\begin{quote}
%\emph{$(H_{\omega_2}^V,\in)$ is a fully elementary substructure of $(V[G],\in)$}. 
%\end{quote}
%If this occurs we should expect that 
%$H_{\omega_2}^V$ is as existentially closed as possible with respect to forcing extensions of $V$ by 
%proper forcings. 
We can pursue this approach not only for proper forcings but for a variety of 
classes $\Gamma$ including the class 
%$\Omega$ of all forcing, or the class
 of proper forcing notions, semiproper forcing notions, $\SSP$, etc. 
%for which $\lambda_\Gamma$ is 
%either 
%$\omega$ or $\omega_1$, we will see that $\lambda_\Gamma$ is $\omega$ if $\Gamma$ is the class of all 
%forcings, while $\lambda_\Gamma$ is $\omega_1$ if $\Gamma$ is the class of semiproper (proper, SSP,\dots) forcings.

It turns out that if $\Gamma$ is a sufficiently well-behaved class,\footnote{All of the above class are.} then the corresponding $\Ord$-length iteration is equivalent to the category forcing $(\Gamma, \leq_\Gamma)$. We can then define the bounded category forcing axiom for $\Gamma$, $\BCFA(\Gamma)$, as the assertion that $H_{\omega_2}^V$ is an elementary substructure of $V[H]$, for some generic extension $H$ of $V$ by $(\Gamma, \leq_\Gamma)$.  Working in G\"odel-Bernays or in Morse-Kelley set theory, this is a perfectly meaningful statement.
 Using again the fact that $\Gamma$ is well-behaved, we can then prove, in the presence of our (mild) large cardinal assumption, that for every forcing extension $V[G]$ via a member from $\Gamma$ there is a further forcing extension $V[G][g]$ via a member from $\Gamma^{V[G]}$ satisfying $\BCFA(\Gamma)$;\footnote{We are of course identifying $\Gamma$ with some reasonable definition of it.} and furthermore, if $V\models\BCFA(\Gamma)$ and $G$ and $g$ are as above, then $H_{\omega_2}^V$ and $H_{\omega_2}^{V[G][g]}$ have the same theory.    
 
 These bounded category for axiom for $\Gamma$ not only imply the corresponding
%reinforce 
bounded forcing axioms but also strengthen the 
resurrection axioms introduced by Johnstone and Hamkins \cite{HAMJOH} and their iterated 
version introduced by Audrito and the second author \cite{AUDVIA}. We are now in a position to briefly outline how these results compare to the notion of
model completeness and model companionship.

A first order theory $T$ is \emph{model complete} if and only if $\mathcal{M}\prec \mathcal{N}$ whenever $\mathcal{N}$ is a superstructure of $\mathcal{M}$ and they are both models of $T$; equivalently if and 
only if any model of $T$ is $T_\forall$-e.c.. 
$T$ is \emph{the model companion of $S$} if
$T$ is model complete and if every model of $T$ embeds into a model of $S$ and vice versa; equivalently,
if $T_\forall=S_\forall$ and any model of $T$ is $S$-e.c.

A standard example is obtained by taking $S$ to be the theory of rings with no zero divisors in signature $\bp{+,\cdot,0,1}$ and $T$ the theory of algebraically closed fields in the same signature.

Now let us look at the $\Gamma$-generic multiverse (for a sufficiently well-behaved  class $\Gamma$):
\[
\bp{H_\kappa^{V[G]}: \, \kappa\geq\omega_2^{V[G]},\, G\, \text{ is $V$-generic for a forcing in $\Gamma$}}.
\]

The above considerations show that the class
\[
\bp{H_{\omega_2}^{V[G]}: \, V[G]\models\BCFA(\Gamma),\, G\, \text{ is $V$-generic for a forcing in $\Gamma$}}
\]
sits inside the $\Gamma$-generic multiverse much in the same way the elementary class of models of $T$ 
sits inside the elementary class of models of $S$, whenever $T$ is the model companion of $S$.
%
% the generic multiverse given by the 
%$H_{\omega_2}$ of
%models of $\MK+\BCFA(\SSP)$ which are obtained as forcing extensions of $V$ behave almost as the elementary class given by a model complete theory. Moreover since $\Sigma_1$-truths are preserved across forcing extensions
%by $\SSP$-forcings, we get that these models sits inside the generic multiverse 
%\[
%\bp{H_\kappa^{V[G]}: \, \kappa\geq\omega_2^{V[G]},\, G\, \text{$V$-generic for an SSP forcing}}
%\]
%much in the same way the elementary class given by a model complete theory $T$ sits inside the models of a theory 
%$S$ of which is themodel companion.

These considerations do \emph{not} establish that the theory of $H_{\omega_2}$ is model complete if $V\models\BCFA(\Gamma)$. However, much stronger connections between generic absoluteness, forcing axioms, and model companionship can be obtained if one works with richer signatures. For instance, 
\cite{VIAVENMODCOMP} shows that the theory of $H_{\omega_1}^V$ is the model companion of the theory of 
$V$ in the signature $\tau_\bool{UB}=\tau_{\bool{ST}}\cup\bool{UB}^V$ where 
$\bool{UB}^V$
is the class of universally Baire sets, and each $B\in\bool{UB}^V$ defines a predicate symbol for $\tau_\bool{UB}$.
%Leveraging 
%on 
Using the 
%breakthrough 
result of the first author and Schindler establishing that Woodin's axiom (*) follows from 
$\bool{MM}^{++}$ \cite{ASPSCH(*)}, 
the second author \cite{VIATAMSTII} has shown that in models of $\bool{MM}^{++}$ where $\bool{UB}^\sharp$ is invariant across forcing extensions, the theory of $H_{\omega_2}$ is the model companion of the theory of $V$
in signature $\tau_{\NS_{\omega_1},\bool{UB}}=\tau_\bool{UB}\cup\bp{\NS_{\omega_1}}$ (where $\NS_{\omega_1}$ is the non-stationary ideal on $\omega_1$, and is the canonical interpretation of its associated unary predicate symbol in 
$\tau_{\NS_{\omega_1},\bool{UB}}$). Note that assuming large cardinals, the $\tau_\bool{UB}$-theory of $H_{\omega_1}$ is generically invariant, while in \cite{VIATAMSTII} it is also shown that the universal fragment of the 
$\tau_{\NS_{\omega_1},\bool{UB}}$-theory of $V$ is invariant across the generic multiverse. Also,
the model completeness results of \cite{VIATAMSTII}  entail that the $\tau_{\NS_{\omega_1},\bool{UB}}$-theory of 
$H_{\omega_2}$ is invariant
across forcing extension of $V$ which model $\bool{MM}^{++}$. 
In particular, the results of the present paper are weaker than those obtained in \cite{VIATAMSTII} when predicated about the class $\Gamma$ of SSP-forcings, but as we will see below, they can also be asserted for a variety of classes $\Gamma$ other than $\SSP$ (for example proper, semiproper, etc);
finally the present paper generalizes and improves results appearing in 
\cite{HAMJOH,VIAMM+++,VIAMMREV,VIAAUD}, where appropriate strengthenings of forcing axioms are introduced with the aim of achieving similar goals.

\subsection{Unbounded category forcing axioms}

It is worth pointing out that bounded category forcing axioms are natural bounded forms of much stronger axioms, which we call \emph{category forcing axioms}. Given an absolutely well-behaved class $\Gamma$ of complete Boolean algebras, the category forcing axiom for $\Gamma$, $\CFA(\Gamma)$, asserts that the class of certain pre-saturated towers of ideals whose regular open completion is in $\Gamma$ and which have certain `rigidity' properties is dense in the category forcing $(\Gamma, \leq_\Gamma)$.\footnote{By `tower of ideals' we mean the forcing notion resulting from naturally relativizing the definition of the stationary tower forcing to some given family of ideals of sets with suitable properties. Forcing with such a forcing gives rise to an elementary embedding $j:V\longrightarrow M$, for a certain subclass $M$ of the generic extension. In the current situation, $j$ has critical point $\lambda_\Gamma^+$.} The theory of category forcing axioms, developed by the authors of the present paper, generalizes the theory of the strong form of Martin's Maximum know as $\bool{MM}^{+++}$. The axiom $\bool{MM}^{+++}$, isolated by the second author, is in fact $\CFA(\SSP)$. 
% where $\SSP$ denotes the class of forcing notions preserving stationary subsets of $\omega_1$.  

In the presence of sufficiently strong large cardinals,\footnote{The relevant large cardinals---typically in the region of the existence of a proper class of supercompact cardinals together and one almost super-huge cardinal---are now much stronger than the ones we need for our analysis of \emph{bounded} category forcing axioms.} we obtain the following.
\begin{enumerate}
\item $\CFA(\Gamma)$ can always be forced by a forcing in $\Gamma$, and in fact by the intersection of $(\Gamma, \leq_\Gamma)$ with $V_\delta$ for some supercompact cardinal $\delta$.
\item $\CFA(\Gamma)$ implies a strong form of the forcing axiom $\FA_{\lambda_\Gamma}(\Gamma)$.
\item If $\bool{B}\in\Gamma$, $G$ is $V$-generic for $\bool{B}$, and $V[G]\models\CFA(\Gamma)$, then $\mtcl C_{\lambda_\Gamma^+}^V$, the $\lambda_\Gamma^+$-Chang model as computed in $V$, is an elementary substructure of $\mtcl C_{\lambda_\Gamma^+}^{V[G]}$.
\end{enumerate}

%This paper 
%provides such an explanation in line with a series of results appearing 
%in~\cite{VIAAUDSTEBOOK,AUDVIA,VACVIA,VIAMM+++,VIAMMREV,VIAUSAX}.
%One of the key observations we make is that forcing 
%axioms come in pairs with generic absoluteness properties for their models: briefly, we show that
% there are natural strengthenings $\CFA(\Gamma)$ of well-known forcing axioms such as $\mathsf{PFA}$, 
%$\mathsf{MM}$, and so on---where $\Gamma$ is a suitably chosen class of forcings, e.g.\ the class of forcings which are proper, semiproper, stationary set preserving, etc.---with the property that forcings in $\Gamma$ which preserve $\CFA(\Gamma)$
%do not change the theory of a large fragment of the universe of sets. We refer to these statements $\CFA(\Gamma)$ as \emph{category forcing axioms}.

%The particular fragments of the universe we are referring to are the Chang models. 

In (3) above, given an infinite cardinal $\lambda$, \emph{the $\lambda$-Chang model}, denoted by $\mathcal C_\lambda$, is the $\subseteq$-minimal transitive model of $\ZF$ containing all ordinals and closed under $\lambda$-sequences. It can be construed as $\mathcal C_\lambda=L(\Ord^\lambda)$.\footnote{Where $L(\Ord^\lambda)=\bigcup_{\alpha\in\Ord}L(\Ord^\lambda\cap V_\a)=\bigcup_{\a\in \Ord}L[\Ord^{\lambda}\cap V_\alpha]$.}  In particular, $L(\mathcal P(\lambda))$ is a definable inner model of $\mathcal C_\lambda$, and $H_{\lambda^+}\subseteq \mathcal C_\lambda$ is definable in $\mathcal C_\lambda$ from $\lambda$. As is well-known, $\mathcal C_\lambda$ need not satisfy $\AC$. For instance, by a result of Kunen \cite[Thm. 1.1.6 and Rmk. 1.1.28]{LARSONBOOK}, if there are $\lambda^+$-many measurable cardinals, then $\mathcal C_\lambda\models\lnot\AC$. 

It is not difficult to see that if $\Gamma$ is an absolutely well-behaved class, then $\CFA(\Gamma)$ implies $\BCFA(\Gamma)$ (again, assuming sufficiently strong large cardinals). In particular, all axioms $\CFA(\Gamma)$ obtained from considering the classes $\Gamma$ from the second part of the present paper are provably pairwise incompatible.

We should point out that, unlike the bounded category forcing axioms we will be studying in the present paper, the stronger category forcing axioms we are referring to here admit a first order definition. In fact, given that the relevant classes $\Gamma$ admit a $\Delta_2$ definition, possibly with some parameter, it is easy to see that $\CFA(\Gamma)$ can be expressed as a $\P_3$ sentence (in the same parameter).   

We should clarify that, despite the more attractive aspects of category forcing axioms, compared to their bounded forms, we have opted to present in this paper only the theory of the latter. The reason for this is that the theory of their unbounded counterparts is much more involved and would require a longer article. This work will appear elsewhere.

\section{Forcing with forcings} \label{sec:forcingwithforcings}

The aim of this section is to develop a general theory of category forcings i.e.\ class forcings $(\Gamma,\leq_\Gamma)$ with $\Gamma$ a 
definable class of forcings and $P\leq_\Gamma Q$ if whenever $G$ is $V$-generic for $P$, in $V[G]$ there is $H$ $V$-generic for $Q$ such that
$V[G]$ is a generic extension of $V[H]$ by a forcing in $\Gamma^{V[H]}$. It will become apparent
% be transparent 
that such an analysis can be systematically carried out
if one focuses on the subclass of $\Gamma$ given by the complete boolean algebras in $\Gamma$. 
First of all this is no loss of information since our analysis will identify properties of elements of $\Gamma$ which are invariant with respect to 
boolean completions.
For example ``whenever $G$ is $V$-generic for $P$, in $V[G]$ there is $H$ $V$-generic for $Q$'' can be formulated in algebraic terms as
``there is a (possibly non-injective) complete homomorphism $i:\RO(Q)\to\RO(P)$''.
%(where $\Gamma$-correct amounts to say that the quotient forcing $\RO(P)/_H$ is in $\Gamma^{V[H]}$ whenever $H$ is $V$-generic for $Q$).
We hope that this example makes transparent that, by focusing on cbas rather than posets, 
we will be able to leverage the algebraic structure of complete boolean algebra to 
greatly simplify the formulation of certain concepts, as well as many proofs. 
%(for example the notion of $\Gamma$-rigidity can be formulated trivially in terms of cbas -- see Def. \ref{def:gammarig} -- but only in a rather %
%convoluted form in terms of partial orders -- see Lemma \ref{lem:eqtrGamma}(\ref{lem:eqtrGamma2})). 
However most of the forcings in $\Gamma$ we will consider are not naturally presented as cbas; for example one of our main result will be to show that $(\Gamma\cap V_\delta,\leq_\Gamma\cap V_\delta)$, which is not even separative, is in $\Gamma$ for most inaccessible cardinals $\delta$.
% note that 
%$(\Gamma\cap V_\delta,\leq_\Gamma\cap V_\delta)$. 
So we will feel free to decide depending on the issue under examination whether
% it is convenient 
to focus on $\Gamma$ as a class of cbas or rather as a class of partial orders, keeping in mind that our results applie equally well in both set-ups. 
%We will make sure that no confusion will arise on the reader.

\begin{notation}
Given a cba $\bool{B}\in V$ and a family $A$ of elements of $V^{\bool{B}}$,
\[
A^\circ=\bp{\ap{\tau,1_\bool{B}}: \tau\in A}.\footnote{This very convenient notational trick is due to Asaf Karagila.}
\]

For a partial order $P$ and $p\in P$, $P\restriction p=\bp{q\in P:q\leq p}$.
\end{notation}

\begin{definition}
Let $\bool{B}$ be an infinite complete boolean algebra and $\dot{\kappa}\in V^{\bool{B}}$ be 
such that 
\[
\Qp{\dot{\kappa} \text{ is a regular cardinal}}_{\bool{B}}=1_\bool{B}.
\]
We define
\[
H_{\dot{\kappa}}^{\bool{B}}=\bp{\tau: \,\tau\in V^{\bool{B}},\, 
\Qp{\trcl(\tau)\text{ is hereditarily of size less than }\dot{\kappa}}_{\bool{B}}=1_{\bool{B}}}
\]
\end{definition}

$(H_{\dot{\kappa}}^{\bool{B}})^\circ\in \mathcal{V}^{\bool{B}}$ is a canonical name for 
$H_{\dot{\kappa}_G}^{V[G]}$ whenever $G$ is $V$-generic for $\bool{B}$.
More precisely:

\begin{quote}{}
Assume  $G$ is $V$-generic for $\bool{B}$. Then
\[
(H_{\dot{\kappa}}^{\bool{B}})^\circ_G=\bp{\tau_G: \, 
\tau\in H_{\dot{\kappa}}^{\bool{B}}}=H_{\dot{\kappa}_G}^{V[G]}.
\]
\end{quote}

%Considering just $\bool{B}$-names of hereditary size at most $|\bool{B}|^+$ suffices, because all new elements
%of $H_{\dot{\omega}_1}^{V[G]}$ are named by some $\tau\in H_{|\bool{B}|^+}\cap V^{\bool{B}}$
% (the optimal bound on the size of $\bool{B}$-names necessary to compute $H_{\omega_1}^{V[G]}$ could be smaller than $|\bool{B}|^+$, but for the moment we stick to this definition).
% %given by the maximum of the density and the chain condition of $\bool{B}$).
On the face of its definition $H_{\dot{\kappa}}^{\bool{B}}$ is a proper class. 
To avoid unnecessary 
complications we can use Scott's trick and consider  just its elements whose rank is bounded 
by some sufficiently large fixed  $\alpha$. As we will see, in many cases of interest it suffices to take
$\alpha=|\bool{B}|$.

\begin{definition} Given a complete boolean algebra $\bool{B}$, we denote by $\dot{\omega_1}$ a canonically chosen member of $V^{\bool{B}}$ such that  \[
\Qp{\dot{\omega}_1 \text{ is the first uncountable cardinal}}_{\bool{B}}=1_\bool{B}.
\]
\end{definition}

There is of course a dependence on $\bool{B}$ in the above definition of $\dot{\omega_1}$. However, we will not need to make this dependence explicit in the notation. 

\subsection{Projective absoluteness} \label{subsec:genabsforax}

We want to 
%outline
sketch the reason why 
%the 
Solovay's model for set theory obtained by forcing with $\Coll(\omega,{<}\delta)$
gives the ``correct'' theory of projective sets. While doing so 
we outline how these results 
could be generalized to larger fragment
of the universe sets.
%
%
%Woodin has proved the following remarkable result:
The following is a weakening of Woodin's generic absoluteness results for the $\omega$-Chang model:

\begin{theorem}\cite[Cor. 3.1.7]{LARSONBOOK} \textbf{(Woodin)}.
Assume $V$ is a $\ZFC$ model in which there are class-many Woodin cardinals. 
%which are 
%a limit of Woodin cardinals. 
Let
$\phi(x_1,\dots,x_n)$ be  \emph{any} formula for the signature $\bp{\in}$ and $a_1,\dots,a_n\in H_{\omega_1}$.
TFAE:
\begin{enumerate}
\item $H_{\omega_1}^V\models\phi(a_1,\dots,a_n)$.
\item $V\models\Qp{H_{\dot{\omega}_1}\models\phi(\check{a}_1,\dots,\check{a}_n)}_{\bool{B}}=1_{\bool{B}}$ 
for some cba $\bool{B}$. 
\item $V\models\Qp{H_{\dot{\omega}_1}\models\phi(\check{a}_1,\dots,\check{a}_n)}_{\bool{B}}=1_{\bool{B}}$ 
for all cbas $\bool{B}$.
\end{enumerate} 
\end{theorem}

In particular, the second order theory of the natural numbers is provably invariant under set-sized forcing.

We will prove a weak form of the above theorem 
(considerably weakening the large cardinal assumptions);
this will 
%outline of
be a first approximation to
the type of axiom we will introduce to get 
generic absoluteness for the theory of $H_{\lambda^+}$
with respect to arbitrary infinite cardinals $\lambda$.

\begin{notation}
For $\delta$ an ordinal
\begin{itemize}
\item
$p\in \Coll(\omega,{<}\delta)$ if and only if
$p:\omega\times\delta\to\delta$ is such that
$\dom(p)$ is finite and $p(n,\alpha)\in \alpha$ for all $\alpha\in\dom(p)$.

\item
$\bool{B}_\delta$ denotes the boolean completion of the forcing notion
$\Coll(\omega,{<}\delta)$.

\item
In models $(V,\mathcal{V})$ of $\MK$, 
$\Coll(\omega,{<}\Ord),\bool{B}_\Ord\in\mathcal{V}$ are the class forcings 
obtained replacing $\delta$ with $\Ord$.
\end{itemize}
\end{notation}

$\bool{B}_\Ord$ is a well defined element of $\mathcal{V}$ because
$\Coll(\omega,{<}\Ord)$ is ${<}\Ord$-CC, hence the suprema needed to define $\bool{B}_\Ord$
are suprema of set-sized antichains, and therefore $\bool{B}_\Ord$ is a proper class.

Generically $\Coll(\omega,{<}\delta)$ adds bijections 
%of
between $\omega$ 
%with all sets of size less than $\delta$,
and all ordinals less than $\delta$,
while making
$\delta$ the first uncountable cardinal; $\Coll(\omega,{<}\Ord)$ therefore makes all sets existing in the generic extension hereditarily countable.

Solovay realized that these forcings produce natural models of the 
theory of projective sets (equivalently of the first order theory of $(H_{\omega_1},\in)$).

The following are well-known properties of $\bool{B}_\delta$ (which hold also for $\delta=\Ord$ 
by a straightforward generalization of the proofs).
\begin{theorem}
Let $(V,\mathcal{V})$ be a model of $\MK$. 
Assume $\delta$ is an inaccessible cardinal or $\delta=\Ord$.
Then: 
\begin{description}
\item[$\bool{B}_\delta$ is ${<}\delta$-CC]\cite[Theorem 15.17(iii)]{JECH}
 Hence $H_{\dot{\omega}_1}^{\bool{B}_\delta}=V_\delta\cap V^{\bool{B}_\delta}$.

\item[Universality]  \cite[Lemma 26.9]{JECH}
%For all cbas $\bool{B}\in V_\delta$ there is a complete embedding 
%$i:\bool{B}\to\Coll(\omega,<\delta)$. 
%
%Hence 
$\bool{B}_\delta$ contains as a complete subalgebra an isomorphic copy of any
cba of size smaller than $\delta$.

\item[Factor Lemma]\cite[Cor. 26.11]{JECH} For all cbas $\bool{B}\in V_\delta$,
\[
\bool{B}\ast\dot{\Coll}(\check{\omega},{<}\check{\delta})\cong 
%\bool{B}\times \Coll(\omega,<\delta)\cong 
\Coll(\omega,{<}\delta).
\]

%\noindent .

\item[Logical Completeness]  $\bool{B}_\delta$ is homogeneous \cite[Thm 26.12]{JECH}. 
% More precisely:
%
%Assume $i:\bool{C}\to\bool{D}$ is a regular embedding with 
%$\bool{C},\bool{D}\sqsubseteq\bool{B}_\delta$
%and of size smaller than $\delta$.
%
%Then there is an automorphism $k:\bool{B}_\delta\to \bool{B}_\delta$ 
%such that $k\restriction\bool{C}=i$.

%In particular 
%\(
%\Qp{\phi(\check{a}_1,\dots,\check{a}_n)
%}_{\bool{B}_\delta}\in \mathsf{2}
%\)
%for all $a_1,\dots,a_n \in V$
%%$\bool{B}\sqsubseteq \bool{B}_\delta$ of size less than $\delta$ and 
%%$a_1,\dots,a_n \in H_{\omega_1}$ and formula $\phi(x_1,\dots,x_n)$
%and formulae $\phi(x_1,\dots,x_n)$.
%
Hence 
\begin{align*}
T_\delta=&\{\phi(a_1,\dots,a_n):\, a_1,\dots,a_n\in H_{\omega_1},\\ 
&\Qp{H_{\dot{\omega}_1}\models\phi(\check{a}_1,\dots,\check{a}_n)}_{\bool{B}_\delta}=
1_{\bool{B}_\delta}\}
\end{align*}
is a complete first order theory.
%\item[Strong canonicity] \cite[Lemma 3.1.5]{LARSONBOOK}
%Let 
%denote the first order $\bp{\in}$-theory described in the previous item.
%(i.e. the theory of the structure $H_{\dot{\omega}_1}^{\bool{B}_\delta}/_G$ for some (any)
%$G\in\St(\bool{B}_\delta)$).
%Assume $\bool{C}$ is \emph{any} cba such that
%$\Qp{\check{\delta}=\dot{\omega}_1}_{\bool{C}}=1_{\bool{C}}$.
%Then the first order theory of $H_{\dot{\omega}_1}^{\bool{C}}/_H$ is $T_\delta$
%independently of the choice of $H\in\St(\bool{C})$.
%
%
%Then in $V^{\bool{C}}$ there is $\dot{Q}$ $\bool{C}$-name for a $<\delta$-distributive forcing
%such that whenever $H\ast K$ is $V$-generic for $\bool{C}\ast\dot{Q}$ in $V[H\ast K]$ there is
%$\bar{K}$ $V$-generic for $\Coll(\omega,<\delta)$ such that
%\[
%H_{\omega_1}^{V[\bar{K}]}=H_{\omega_1}^{V[H]}=H_{\omega_1}^{V[H\ast K]}.
%\]
%
%$T_\delta$ is the theory %(with parameters in $H_{\omega_1}$)
%$H_{\dot{\omega_1}}^{\bool{C}}\models T_\delta$ for \textbf{any} $\bool{C}$ such that:
%\begin{itemize}
%\item $\Qp{\check{\delta}=\dot{\omega}_1}_{\bool{C}}=1_{\bool{C}}$,
%\item $\bool{C}_\tau$ has size less than $\delta$ for all $\tau\in H_{\dot{\omega}_1}^{\bool{C}}$,
%($\bool{C}_\tau$ is the smallest regular subalgebra $\bool{D}$ of $\bool{C}$ such that $\tau\in V^{\bool{D}}$).
%\end{itemize}
\end{description}
\end{theorem}

%Observe that
%whenever $G$ is $V$-generic for $\Coll(\omega,<\Ord)$
%\[
%V[G]\models
%\ZFC\setminus\mathrm{Powerset}+\textrm{ all sets are hereditarily countable}.
%\]
%I.e. $V[G]$ is a model of the theory of $H_{\omega_1}$.

%
%In particular we obtain the following:
%\begin{theorem}
%Let $(V,\mathcal{V})$ be a model of $\MK$. 
%
%\begin{itemize}
%\item The first order theory
%\begin{align*}
%T_{\mathcal{V}}=&
%\{\phi(a_1,\dots,a_n):\, a_1,\dots,a_n\in H_{\omega_1}, \phi\text{ without class quantifiers},\\
%&\Qp{\phi(\check{a}_1,\dots,\check{a}_n)}_{\bool{B}_\Ord}
%=1_{\bool{B}_\Ord}\}. 
%\end{align*}
%is complete.
%
%
%\item
%Let $P\in\mathcal{V}$ be any class forcing notion preserving the regularity of $\Ord$,
%making all ordinals of countable size, and such that no new set-sized element added by $P$
%is $V$-generic only for a class-sized complete sub-forcing of $P$ 
%(for example $P=\Coll(\omega,<\Ord)$). 
%
%
%Then $T_{\mathcal{V}}$ 
%is the first order theory of $V[G]$ for any $G$ $\mathcal{V}$-generic for $P$.
%\end{itemize}
%\end{theorem}

%
%Essentially the proof of all the results of the previous slide can be carried in $V_{\delta+1}$ and 
%$(V_\delta,V_{\delta+1})$ is the canonical example of a model of $\MK$, hence the results
%of the previous slides can be carried over to $(V,\mathcal{V})$.

%In a model of $\MK$, 
%$T_{\Ord}$ is a canonical candidate for the ``true'' theory of 
%projective sets. 

\begin{definition}
$\bool{PA}$ (\textrm{Projective Absoluteness}) 
\[
H_{\omega_1}\models T_{\Ord}.
\]
\end{definition}

$\bool{PA}$ states that 
for all formulae $\phi(x_1,\dots,x_n)$ and $a_1,\dots,a_n\in H_{\omega_1}$ we have that
\[
H_{\omega_1}\models\phi(a_1,\dots,a_n)\text{ if and only if }
\Qp{\phi(\check{a}_1,\dots,\check{a}_n)}_{\bool{B}_\Ord}=1_{\bool{B}_\Ord}.
\]

  A complete boolean algebra $\bool{C}\in V$ satisfies $\bool{PA}$ if and only if 
(viewing it as a complete subalgebra of $\bool{B}_\Ord$) for all 
$\tau_1,\dots,\tau_n\in H_{\dot{\omega}_1}^{\bool{C}}$ and formulae $\phi(x_1,\dots,x_n)$
without class quantifiers, we have that
\[
\Qp{\phi(\tau_1,\dots,\tau_n)}_{\bool{B}_\Ord}=\Qp{(H_{\dot{\omega}_1}^{\bool{C}})^\circ\models
\phi(\tau_1,\dots,\tau_n)}_{\bool{C}}.
\]

\begin{theorem}
Assume 
\[
(V,\mathcal{V})\models\MK+\Ord\textrm{ is Mahlo}.
\]
 \emph{(i.e. every club subclass of the ordinals in $\mathcal{V}$ contains a regular cardinal)}.

Let $(\Omega,\to^\Omega)$ be the category given by the class of cbas as objects and 
the class of complete homomorphisms between them as arrows.\footnote{Note that the homomorphism may not be injective. The key point is that the preimage $i^{-1}[G]$ by a complete (possibly non-injective) homomorphism $i:\bool{B}\to \bool{C}$ of a $V$-generic filter $G$ for $\bool{C}$ is a $V$-generic filter for $\bool{B}$.}

Set $\bool{C}\leq_\Omega\bool{B}$ if there is $i:\bool{B}\to\bool{C}$ in $\to^\Omega$.

Then 
\[
D_{\bool{PA}}=\bp{\bool{B}\in\Omega: \bool{B} \text{ satisfies }\bool{PA} }
\]
is dense in the class partial order $(\Omega,\leq_\Omega)$.
\end{theorem}

In particular $\bool{PA}$ is a consistent axiom.

\begin{proof}
Since $\bool{B}\geq_\Omega\bool{B}_\delta$ for any inaccessible $\delta>|\bool{B}|$,
it suffices to show that there are stationarily many $\delta$ such that
$\bool{B}_\delta \text{ satisfies }\bool{PA}$.

We need to put together three separate observations on the properties of $\bool{B}_\gamma$:
\begin{enumerate}
\item For any ordinal $\gamma$ 
\[
\Coll(\omega,{<}\gamma)\sqsubseteq\Coll(\omega,{<}\Ord),
\] 
(i.e. the inclusion map is a complete embedding of partial orders).

\item Assume $\delta$ is an inaccessible cardinal .
Since $\bool{B}_\delta$ is ${<}\delta$-CC and has size $\delta$, 
it is not hard to check that 
%\[
%H_{\omega_1}^{\delta}=\bp{\tau: \,\tau\in V^{\bool{B}_\delta}\cap V_\delta,\, 
%\Qp{\trcl(\tau)\text{ is countable}}_{\bool{B}_\delta}=1_{\bool{B}_\delta}}
%\]
%is such that
\[
(H_{\dot{\omega}_1}^{\delta})^\circ_G=\bp{\tau_G: \, \tau\in V_\delta\cap V_{\bool{B}},\,
\Qp{\trcl(\tau)\text{ is countable}}_{\bool{B}_\delta}=1_{\bool{B}}}=H_{\omega_1}^{V[G]}=V_\delta[G]
\]
whenever $G$ is $V$-generic for $\bool{B}_\delta$.

Observe also that $\Coll(\omega,{<}\delta)$ is a definable class forcing in the $\ZFC$-model 
$(V_\delta,\in)$, which the $\MK$-model 
$(V_\delta,V_{\delta+1},\in)$ proves to be ${<}\delta$-CC.

By the forcing theorem\footnote{The forcing theorem applies in $V_\delta$ to the class forcing (relative to $V_\delta$) $\Coll(\omega,{<}\delta)$ because $\Coll(\omega,{<}\delta)$ is definably ${<}\delta$-CC, in the sense that any definable antichain over $(V_\delta,\in)$ belongs to $V_\delta$.} applied in $V_\delta$ to the definable class forcing $\Coll(\omega,{<}\delta)$, 
we get that
for all $V$-generic filters $G$ for $\bool{B}_\delta$,
all formulae $\phi(x_1,\dots,x_n)$, and all $\tau_1,\dots,\tau_n\in V_\delta\cap V_{\bool{B}}$,
\[
(H_{\dot{\omega}_1}^{\delta})^\circ_G=V_\delta[G]\models\phi(\tau_1,\dots,\tau_n)
\]
if and only if for some $p\in\ G\cap \Coll(\omega,{<}\delta)$,
\[
V_\delta\models p\Vdash\phi(\tau_1,\dots,\tau_n).
\]
%The above considerations hold for any inaccessible cardinal $\delta\in V$.

\item
Since $\ap{V,\mathcal{V}}$ is a model of $\MK$, the class 
\[
C=\bp{\gamma:\,(V_\gamma,\in)\prec (V,\in)}
\]
is a club subset of $\Ord$. 

Now
for each $\gamma\in C$, each formula $\phi(x_1,\dots,x_n)$ without class quantifiers, and each 
$\tau_1,\dots,\tau_n\in V_\gamma\cap V^{\bool{B}}$,
\[
V_\gamma \models p\Vdash_{\Coll(\omega,{<}\gamma)}\phi(\tau_1,\dots,\tau_n)
\]
if and only if
\[
V\models p\Vdash_{\Coll(\omega,{<}\Ord)}\phi(\tau_1,\dots,\tau_n).
\]
This is the case since the classes involved in the definition of 
$$p\Vdash_{\Coll(\omega,{<}\Ord)}\phi(\tau_1,\dots,\tau_n)$$ are all definable in 
$V$ using parameters in $V_{\gamma}$.
\end{enumerate}

Since $\Ord$ is Mahlo in  $\ap{V,\mathcal{V}}$, there are stationarily many 
inaccessible cardinals $\delta$ in $C$, and for 
any inaccessible $\delta\in C$ the three properties outlined above for $\bool{B}_\delta$ hold simultaneously.

The following claim suffices to complete the proof of the theorem.
\begin{claim}
Assume $\delta\in C$ is inaccessible.
Then for all formulae $\phi(x_1,\dots,x_n)$ without class quantifiers and
$\tau_1,\dots,\tau_n\in H_{\dot{\omega_1}}^{\bool{B}_\delta}=V_\delta\cap V^{\bool{B}_\delta}$,
\begin{align*}
\Qp{H_{\dot{\omega}_1}\models\phi(\tau_1,\dots,\tau_n)}_{\bool{B}_\delta}=
\Qp{\phi(\tau_1,\dots,\tau_n)}_{\bool{B}_\Ord}.
\end{align*}
\end{claim}
\begin{proof}
Since $\delta\in C$ is inaccessible, $\phi$ does not have class quantifiers, and 
$\Coll(\omega,{<}\delta)\sqsubseteq\Coll(\omega,{<}\Ord)$, we have that:
\begin{align*}
\Qp{H_{\dot{\omega}_1}\models\phi(\tau_1,\dots,\tau_n)}_{\bool{B}_\delta}=\\
=\bigvee_{\Coll(\omega,{<}\delta)}\bp{p\in\Coll(\omega,{<}\delta): \,\ap{V_\delta,\in} \models p\Vdash_{\Coll(\omega,{<}\delta)}\phi(\tau_1,\dots,\tau_n)}=\\
=\bigvee_{\Coll(\omega,{<}\delta)}\bp{p\in\Coll(\omega,{<}\Ord): \,
\ap{V,\in}\models p\Vdash_{\Coll(\omega,{<}\Ord)}\phi(\tau_1,\dots,\tau_n)}=\\
=\bigvee_{\Coll(\omega,{<}\Ord)}\bp{p\in\Coll(\omega,{<}\Ord): \,
\ap{V,\in}\models p\Vdash_{\Coll(\omega,{<}\Ord)}\phi(\tau_1,\dots,\tau_n)}=\\
=\Qp{\phi(\tau_1,\dots,\tau_n)}_{\bool{B}_\Ord}.
\end{align*}
The first and last equalities hold by definition (for the first observe that 
$\dot{H}_{\omega_1}^\delta=V_\delta \cap V^{\bool{B}_\delta}$ is a canonical $\bool{B}_\delta$-name for 
$H_{\omega_1}^{V[G]}$ for $G$ $V$-generic for $\bool{B}_\delta$); the second equality holds because 
$(V_\delta,\in)\prec (V,\in)$; the third equality holds because
$\Coll(\omega,{<}\delta)\sqsubseteq\Coll(\omega,{<}\Ord)$.
\end{proof}
The theorem is proved.
\end{proof}

\begin{remark}
If $\gamma\in C$ is not inaccessible, then $\Coll(\omega,{<}\gamma)$ is not ${<}\gamma$-cc and collapses
$\gamma$ to become countable.
 Hence 
$H_{\dot{\omega}_1}^{\bool{B}_\gamma}\neq V_\gamma\cap V^{\bool{B}_\gamma}$. 
 This gives that
$\bool{PA}$ may not be satisfied by $\bool{B}_\gamma$ in this case, since the boolean value
of $\Qp{\phi(\tau_1,\dots,\tau_n)}_{\bool{B}_\Ord}$ may require some 
$p\not\in\Coll(\omega,{<}\gamma)$ to be computed if there is some  
$\tau_i$ which is not in $V_\gamma$.
\end{remark}

\begin{corollary}
Assume $(V,\mathcal{V})\models\MK$ and $i:\bool{B}\to\bool{C}$ is a complete homomorphism.
Assume $\bool{B}$ and $\bool{C}$ both satisfy $\bool{PA}$.
Then for any $G$ $V$-generic for $\bool{C}$, letting $H=i^{-1}[G]$, we get that
\[
H_{\dot{\omega}_1}^{\bool{B}}[H]=H_{\omega_1}^{V[H]}\prec H_{\omega_1}^{V[G]}=H_{\dot{\omega}_1}^{\bool{C}}[G].
\]
\end{corollary}

\begin{proof}
The Corollary is an immediate consequence of the following basic model-theoretic observation:

\begin{fact}
Assume $\mathcal{M}_0,\mathcal{M}_1$ and $\mathcal{N}$ are $L$-structures such that the following
diagram is realized:
\[
			\begin{tikzpicture}[xscale=1.5,yscale=-1.2]
				\node (A0_0) at (0, 0) {$\mathcal{M}_0$};
				\node (A0_2) at (2, 0) {$\mathcal{N}$};
				\node (A1_1) at (1, 1) {$\mathcal{M}_1$};
				\path (A0_0) edge [->]node [auto] {$\scriptstyle{\Sigma_\omega}$} (A0_2);
				\path (A1_1) edge [->]node [auto,swap] {$\scriptstyle{\Sigma_{\omega}}$} (A0_2);
				\path (A0_0) edge [->]node [auto,swap] {$\sqsubseteq$} (A1_1);
			\end{tikzpicture}
		\]

Then
\[
\mathcal{M}_0\prec \mathcal{M}_1.
\]
\end{fact}
%\begin{proof}
%Proceed by induction on $n$ observing for the induction step that
%$\Sigma_n$-formulae with parameters in $\mathcal{M}_0$ true in $\mathcal{M}_1$
%are true in $\mathcal{N}$, and therefore also in $\mathcal{M}_0$, since $\mathcal{M}_0\prec\mathcal{N}$.
%\end{proof}

Notice that if $K$ is $V$-generic for $\bool{B}_\Ord$, $G\in V[K]$ is $V$-generic for $\bool{C}$, and
$H\in V[G]$ is $V$-generic for $\bool{B}$, we obtain the above configuration:
\[
			\begin{tikzpicture}[xscale=1.5,yscale=-1.2]
				\node (A0_0) at (0, 0) {$H_{\omega_1}^{V[H]}$};
				\node (A0_2) at (2, 0) {$H_{\omega_1}^{V[K]}$};
				\node (A1_1) at (1, 1) {$H_{\omega_1}^{V[G]}$};
				\path (A0_0) edge [->]node [auto] {$\scriptstyle{\Sigma_\omega}$} (A0_2);
				\path (A1_1) edge [->]node [auto,swap] {$\scriptstyle{\Sigma_{\omega}}$} (A0_2);
				\path (A0_0) edge [->]node [auto,swap] {$\sqsubseteq$} (A1_1);
			\end{tikzpicture}
		\]
\end{proof}

\begin{corollary}
Assume 
\[
(V,\mathcal{V})\models\MK+\bool{PA}+\textrm{$\Ord$ is Mahlo}.
\]

TFAE:
\begin{enumerate}
\item $H_{\omega_1}\models\phi(a_1,\dots,a_n)$.

\item $\Qp{H_{\omega_1}^{\bool{B}}
\models\phi(\check{a}_1,\dots,\check{a}_n)}_{\bool{B}}=1_{\bool{B}}$ 
for some cba $\bool{B}$ which satisfies $\bool{PA}$.

\item $\Qp{H_{\omega_1}^{\bool{B}}\models
\phi(\check{a}_1,\dots,\check{a}_n)}_{\bool{B}}=1_{\bool{B}}$ for all cbas $\bool{B}$ which satisfy
$\bool{PA}$.
\end{enumerate} 
\end{corollary}

Woodin's generic absoluteness results for projective sets provide significant strengthenings of the conclusion of the theorem above in the presence of stronger hypotheses.
%strengthen the above theorem to a great extent.
%are greatly strengthening the above theorem.
The following is a weakening of \cite[Theorem 3.1.2]{LARSONBOOK} which has the same flavor of what we have been showing so far:

\begin{theorem}[Woodin]
Assume $\delta$ is a Woodin cardinal which is a limit of Woodin cardinals
and $\bool{B}\in V_\delta$ is a cba.

Then %$\bool{B}$ is a regular subalgebra of $\bool{B}_\delta$, and
\[
H_{\dot{\omega}_1}^{\bool{B}}\prec H_{\dot{\omega}_1}^{\bool{B}_\delta}.
\]
\end{theorem}

Now observe that the fact whether or not $\delta$ is Woodin is detected at stage $V_{\delta+1}$:
\[
V\models\delta\text{ is a Woodin cardinal }\quad \Longleftrightarrow\quad
V_{\delta+1}\models\delta\text{ is a Woodin cardinal}.
\]
In particular `\emph{$\Ord$ is Woodin}' is 
%first order 
expressible in any model of $\MK$.

\begin{corollary}
Assume 
\[(V,\mathcal{V})\models\MK\,+\,\text{$\Ord$ is Woodin $+$ there are class many Woodin cardinals}.\]

Then \textbf{every} $\bool{B}$ in $\Omega$ satisfies $\bool{PA}$. 

Hence
TFAE:
\begin{enumerate}
\item $H_{\omega_1}\models\phi(a_1,\dots,a_n)$.
\item $\Qp{H_{\omega_1}^{\bool{B}}
\models\phi(\check{a}_1,\dots,\check{a}_n)}_{\bool{B}}=1_{\bool{B}}$ 
for some cba $\bool{B}$. 
\item $\Qp{H_{\omega_1}^{\bool{B}}\models
\phi(\check{a}_1,\dots,\check{a}_n)}_{\bool{B}}=1_{\bool{B}}$ for all cbas $\bool{B}$.
\end{enumerate} 
\end{corollary}

We observe the following:

\begin{remark}
The key steps
%ingredients 
in the proof of the density of 
\[
D_\bool{PA}=\bp{\bool{B}: H_{\dot{\omega_1}}^{\bool{B}}\prec 
V^{\Coll(\omega,{<}\Ord)}}
\]
for the preorder given by the class category $(\Omega,\to^\Omega)$ in a model of 
\[
\MK+\Ord\text{ is Mahlo} 
\]
%reposes 
rely on the following crucial
properties:
\begin{itemize}

\item
$\Coll(\omega,{<}\delta)$ preserves the regularity of $\delta$ for any regular cardinal $\delta$ and makes $\delta$ the first uncountable cardinal;

\item $\Coll(\omega,{<}\delta)\sqsubseteq\Coll(\omega,{<}\Ord)$ for all regular cardinals $\delta$;

\item any cba $\bool{B}\in V_\delta$ embeds as a complete suborder in
$\Coll(\omega,{<}\delta)$ for any inaccessible $\delta$;

%\item $\bool{B}\ast\dot{\Coll}(\check{\omega},{<}\Ord)\cong\Coll(\omega,{<}\Ord)$ for any cba $\bool{B}\in V$.  %%% NOTE THAT THIS, EVEN IF TRUE, IS NOT USED IN THE ABOVE PROOF.
\end{itemize}
\end{remark}

We want to replicate the above proof pattern for arbitrary classes of forcings $\Gamma$ closed under two-step iterations.
The guiding idea will be that projective absoluteness 
%covers the case of 
%$\Gamma$-absoluteness
is just $\BCFA(\Gamma)$
for $\Gamma=\Omega$, the class of all forcing notions. We will show that for a variety of definable classes of forcings
$\Gamma$ there is a (uniformly) definable cardinal $\lambda_\Gamma$ associated
%attached
 to $\Gamma$ such that
\[
D_\Gamma=\bp{\bool{B}\in\Gamma: H_{\dot{\lambda}_{\Gamma}^+}^{\bool{B}}\prec 
V^{\Gamma\restriction\bool{B}}}
\]
is dense in the class forcing induced by the category $(\Gamma,\to^\Gamma)$.

%The case 
Our analysis of projective absoluteness 
%illustrated 
in the previous section 
%outlines 
shows
that
$\lambda_\Omega=\omega$.

We will focus in section \ref{suitable classes} on the analysis of classes $\Gamma$ for which 
$\lambda_\Gamma=\omega_1$. However, the machinery we will present below is modular and
shows that for any class of forcings $\Gamma$ satisfying certain reasonable properties,
$\lambda_\Gamma$ is well defined and $D_\Gamma$ is dense in the class forcing induced by the category $(\Gamma,\to^\Gamma)$.

First of all we need to introduce the terminology and notation required to formulate precisely what is meant
by the class forcing induced by the category $(\Gamma,\to^\Gamma)$, by 
$\BCFA(\Gamma)$,
%$\Gamma$-absoluteness, 
etc.

\begin{notation}
Let $i:\bool{B}\to\bool{C}$ be a complete homomorphism of cbas.
\[
\ker(i)=\bigvee\bp{b\in\bool{B}: i(b)=0_{\bool{C}}},
\]
\[
\coker(i)=\neg\ker(i),
\]

$i$ is a regular embedding if it is injective and complete.

For any $b\in\bool{B}$ 
\begin{align*}
k_{b}:&\bool{B}\to\bool{B}\restriction b\\
&c\mapsto c\wedge b.
\end{align*}

\[
\dot{G}_{\bool{B}}=\bp{\ap{\check{b},b}:\, b\in\bool{B}}
\] 
is the canonical $\bool{B}$-name for
the $V$-generic filter.
\end{notation}

\begin{remark}
All complete homomorphisms $i:\bool{B}\to\bool{C}$ 
are of the form $i_0\circ k_b$ with $i_0:\bool{B}\restriction b\to\bool{C}$ a regular embedding and 
$b=\coker(i)$.
\end{remark}

\begin{notation}
Let $i:\bool{B}\to\bool{C}$ a complete homomorphism and $\dot{\kappa}\in \mathcal{V}^{\bool{B}}$,
$\dot{\delta}\in \mathcal{V}^{\bool{C}}$ be such that
\[
\Qp{\dot{\kappa}\text{ is a regular uncountable cardinal}}_\bool{B}=1_{\bool{B}}
\]
and
\[
\Qp{\dot{\delta}\text{ is a regular uncountable cardinal}}_\bool{C}=1_{\bool{C}},
\]
where 
\[
\Qp{\hat{i}(\dot{\kappa})\leq\dot{\delta}}=1_{\bool{C}}.
\]
We say that
\[
H_{\dot{\kappa}}^{\bool{B}}\prec H_{\dot{\delta}}^{\bool{C}}
\]
if and only if for all 
$\tau_1,\dots,\tau_n\in H_{\dot{\kappa}}^{\bool{B}}$ and formulae $\phi(x_1,\dots,x_n)$
without class quantifiers, we have that
\[
i(\Qp{H_{\dot{\kappa}}\models\phi(\tau_1,\dots,\tau_n)}_{\bool{B}})=
\Qp{H_{\dot{\delta}}\models
\phi(\hat{i}(\tau_1),\dots,\hat{i}(\tau_n))}_{\bool{C}}.
\]
Here,
\begin{align*}
\hat{i}:&\mathcal{V}^{\bool{B}}\to \mathcal{V}^{\bool{C}}\\
&\tau\mapsto\bp{\ap{\hat{i}(\sigma),i(b)}:\,\ap{\sigma,b}\in\tau}
\end{align*}
\end{notation}
%The case studied in 
Our analysis
of projective absoluteness covers the situation in which
$$i:\bool{B}=\RO(\Coll(\omega,{<}\delta))\to\bool{C}=\RO(\Coll(\omega,{<}\gamma))$$ 
is the inclusion, $\delta<\gamma$ are both in $C$ (hence they both force $\bool{PA}$), 
and $\dot{\kappa}=\dot{\omega}_1^{\bool{B}}$ and $\dot{\delta}=\dot{\omega}_1^{\bool{C}}$
are the canonical names for the first uncountable cardinal, for $\bool{B}$ and $\bool{C}$, respectively.

\begin{notation}
Given the standard model $(V,\mathcal{V})$ of $\MK$,
a \emph{definable} class $\Gamma$ denotes  
an element of $\mathcal{V}$ defined by a first order formula
$\phi_\Gamma(x,a_\Gamma)$ with $a_\Gamma\in V$ and 
$\phi_\Gamma(x,y)$ a formula without class quantifiers.
\end{notation}

Note that we allow
the free variable $x$ to also take values in $\mathcal{V}$, since one of our objective will be to infer that
$\phi_\Gamma(\Gamma,a)$ holds in $(V,\mathcal{V},\in)$ for a wide range of classes $\Gamma$.

\begin{definition} Let $\Gamma$ be a definable class of pre-orders.
\begin{itemize}
\item
$(\Gamma,\to^\Gamma)$ is the category whose objects are cbas in $\Gamma$, and whose arrows are the
$\Gamma$-correct homomorphisms (not necessarily injective).

\item
For cbas $\bool{B},\bool{C}$, we write
\begin{itemize}
\item $\bool{B}\geq_\Gamma\bool{C}$ if there is a $\Gamma$-correct $i:\bool{B}\to\bool{C}$, and
\item $\bool{B}\geq^*_\Gamma\bool{C}$ if there is a $\Gamma$-correct $i:\bool{B}\to\bool{C}$ which is also injective (i.e. $i$ is regular).
\end{itemize}

For any ordinal $\delta$,
$\Gamma_\delta=\Gamma\cap V_\delta$.
\end{itemize}
\end{definition}

\begin{definition}
Let $\Gamma$ be a definable class of pre-orders.
\begin{itemize}
\item
$\Gamma$ is \emph{stable under forcing} if for any $P\in \Gamma$, any pre-order $Q$ such that
$\RO(P)\cong\RO(Q)$ is also in $\Gamma$.
\item
%A complete (not necessarily injective) homomorphism of cbas 
%$i:\bool{B}\to \bool{C}$ is $\Gamma$-correct if whenever
%\[
%\Qp{\phi_\Gamma(\bool{C}/_{i[\dot{G}_\bool{B}]},\check{a}_\Gamma)}_{\bool{B}}=\Qp{\bool{C}/_{i[\dot{G}_\bool{B}]}\in\dot{\Gamma}}_{\bool{B}}\geq \coker(i),
%\]
%then $\bool{C}\in\Gamma$ as well.
A complete (not necessarily injective) homomorphism of cbas 
$i:\bool{B}\to \bool{C}$ is \emph{$\Gamma$-correct} if
\[
\Qp{\phi_\Gamma(\bool{C}/_{i[\dot{G}_\bool{B}]},\check{a}_\Gamma)}_{\bool{B}}=\Qp{\bool{C}/_{i[\dot{G}_\bool{B}]}\in\dot{\Gamma}}_{\bool{B}}\geq \coker(i).\footnote{This set-up 
%is cooked exactly in order to express 
%algebraically 
provides an algebraic expression of
the condition that whenever $G$ is $V$-generic for $\bool{C}$ with $i(\coker(i))\in G$ and $H=i^{-1}[G]$, $\bool{C}/_{i[H]}$ is a cba in $V[H]$ which is in $\Gamma^{V[H]}$.}
\]

\item
$\Gamma$ is \emph{closed under two-step iterations} if $\bool{C}\in\Gamma$ whenever 
$\bool{B}\in\Gamma$ and 
$i:\bool{B}\to \bool{C}$ is $\Gamma$-correct.

\item
$(\Gamma,\to^\Gamma)$ is the category whose objects are cbas in $\Gamma$, and whose arrows are the
$\Gamma$-correct homomorphisms (not necessarily injective).

\item
For cbas $\bool{B},\bool{C}$, we write
\begin{itemize}
\item $\bool{B}\geq_\Gamma\bool{C}$ if there is a $\Gamma$-correct $i:\bool{B}\to\bool{C}$, and
\item $\bool{B}\geq^*_\Gamma\bool{C}$ if there is a $\Gamma$-correct $i:\bool{B}\to\bool{C}$ which is also injective (i.e. $i$ is regular).
\end{itemize}

For any ordinal $\delta$,
$\Gamma_\delta=\Gamma\cap V_\delta$.
\end{itemize}
\end{definition}

\begin{remark}
Stability under forcing in $\Gamma$ asserts that $\Gamma$ is closed with respect to the equivalence relation on preorders given by $P\sim Q$ iff $\RO(P)\cong \RO(Q)$. All the properties of 
$\Gamma$ we are interested in are invariant under this equivalence relation, and each equivalence class according to $\sim$ has a canonical representative given by the unique (modulo isomorphism) 
cba $\bool{B}$
belonging to it. In particular, we can restrict our analysis of $\Gamma$ concentrating just 
on its subclass given by the cbas in it.
It will be however sometimes 
convenient to allow as elements of $\Gamma$ also partial orders $P\in V$ which are not cbas
but are such that $\phi_\Gamma(P,a_\gamma)$ holds: specifically, we will need to infer that
for well-behaved classes $\Gamma$,
$\Gamma_\delta\in\Gamma$ for many inaccessible cardinals $\delta$; $(\Gamma_\delta,\leq_\Gamma)$ 
is a pre-order but not a cba, and it is much simpler to 
%unfold and
 describe the combinatorial
properties of $\Gamma_\delta$ 
%rather
 than those of its boolean completion.
%All the classes of pre-orders $\Gamma$ we will consider are such that
%$\phi_\Gamma(P,a_\gamma)$ holds if and only if so does $\phi_\Gamma(\RO(P),a_\gamma)$,
%hence it is harmless to work in this more general set-up.

%
%\begin{remark}
%\emph{}

\begin{itemize}
\item
If $(\Gamma,\to^\Gamma)$ has lower bounds for its finite subsets, then $\leq_\Gamma$ is a trivial forcing notion 
since all conditions are compatible.
\item
For any cba $\bool{B}$ there is a regular embedding $i:\bool{B}\to \RO(\Coll(\omega,\delta))$
for any large enough $\delta$;
consequently the class $(\Omega,\to^\Omega)$ 
of all cbas and all complete homomorphisms between them has lower bounds for its finite subsets, 
hence $\leq^*_\Omega$ and
$\leq_\Omega$ are trivial forcing notions.
\end{itemize}
\end{remark}

%
%
%
%
%\begin{definition}
%$\bool{B}$ is stationary set preserving $\SSP$ if 
%\[
%\llbracket S\text{ is stationary }\rrbracket_{\bool{B}}=1_{\bool{B}}
%\]
%for all $S$ stationary subset of $\omega_1$ in $V$.
%\end{definition}
%
%\begin{definition}
%A complete homomorphism (now it is important that $i$ may not be injective)
%$i:\bool{B}\to\bool{Q}$ is $\SSP$-correct iff
%\[
%\llbracket \bool{Q}/_{i[\dot{G}_{\bool{B}}]}\in\SSP\rrbracket_{\bool{B}}=1_{\bool{B}}.
%\]
%\end{definition}
%
%
%
%
%$\UU^{\SSP,\SSP}$ is the category of complete $\SSP$ boolean algebras with
%$\SSP$-correct homomorphisms.
%
%

\small{
\begin{definition}
A partial order $P$ is in $\SSP$ if for any stationary $S\subseteq\omega_1$,
\[
P\Vdash\check{S}\text{ is stationary}.
\]
\end{definition}

\begin{fact}
 $(\SSP,\leq_\Omega)$ (hence also   $(\SSP,\leq_\SSP)$) is an atomless partial order.
 \end{fact}

\begin{proof}
Assume
 \begin{itemize}
 \item
 $P$ is Namba forcing on $\aleph_2$,
 \item
 $Q$ is
 $\Coll(\omega_1,\omega_2)$.
 \end{itemize}
 Then $\RO(P)$ and $\RO(Q)$ are incompatible conditions in 
 $(\SSP,\leq_\Omega)$ (and therefore also in $(\SSP,\leq_\SSP)$):
 Assume $\bool{R}\leq_\SSP\RO(P),\RO(Q)$, and $H$ is $V$-generic for 
 $\bool{R}$. Then
 \begin{itemize}
  \item
  $\omega_1^{V[H]}=\omega_1$,
  \item 
 there are $G,K\in V[H]$ $V$-generic filters for 
 $P$ and $Q$ respectively (since $\bool{R}\leq_\Omega\RO(P),\RO(Q)$).
 \end{itemize}
 
 $G$ gives in $V[H]$ a sequence cofinal in $\omega_2^V$ of type $\omega$.

 $K$ gives in $V[H]$ a sequence cofinal in $\omega_2^V$ of type $\omega_1^V$.

 Contradiction with the preservation of $\omega_1$ in $V[H]$ (which holds since $\bool{R}\in\SSP$).

 This argument can be repeated in $V^{\bool{B}}$ for any $\bool{B}\in \SSP$.
 \end{proof}

 \begin{remark}
Similar arguments show that $\leq_\Gamma$ defines an atomless partial order for a variety of 
 $\Gamma\subseteq\SSP$ (for example for $\Gamma$ being the class of proper posets, or the class of semiproper posets). 
 \end{remark}

\subsection{$\BCFA(\Gamma)$
%$\Gamma$-absoluteness 
 and how to get it}

First of all we define the cardinal $\lambda_\Gamma$ we attach to a given class $\Gamma$ of forcing notions. $\lambda_\Gamma$ could actually be $\Ord$. However, in all cases of interest $\lambda_\Gamma$ will in fact exist as a cardinal.
It will be needed to formulate 
%$\Gamma$-absoluteness
$\BCFA(\Gamma)$ properly
(it will also be a key parameter
% to formulate 
%properly
in the proper formulation of the iteration theorem we will later 
%on 
require to hold for
forcings in $\Gamma$).

\begin{definition}
Given a definable class of forcings $\Gamma$, $\lambda_\Gamma$ is the supremum of all cardinals
$\eta\in V$ such that all forcings in $\Gamma$ preserve $\eta$.
\end{definition}

\begin{remark}
$\lambda_\Omega=\omega$, $\lambda_\SSP=\omega_1$, and the same holds for all 
$\Gamma\subseteq\SSP$ such that every countably closed forcing is in $\Gamma$ (e.g.\ the class of
proper forcings and the class of semiproper forcings). It is easy to see that 
%With a bit of work the reader can check that 
$\lambda_\Gamma$ is either $\Ord$ or  
the maximum of the set of cardinals of which it is a supremum. We will be interested just in the case in
which $\lambda_\Gamma$ is an uncountable regular cardinal 
(and the reader can safely assume throughout the paper that
$\lambda_\Gamma=\omega_1$).
\end{remark}

\begin{definition}
Let $(V,\mathcal{V})$ be the standard model of $\MK$ and $\Gamma\in\mathcal{V}$ be a definable class of forcings
closed under two-step iterations
and such that $\lambda_\Gamma\in \Ord$. 
We say that
%\emph{$\Gamma$-absoluteness} 
the \emph{Bounded Category Forcing Axiom for $\Gamma$}, $\BCFA(\Gamma)$, 
holds if
\[
H_{\lambda_\Gamma^+}\prec V^{\Gamma},
\]
and
\[
D_\Gamma=\bp{\bool{B}\in\Gamma: H_{\dot{\lambda}_\Gamma^+}^{\bool{B}}\prec V^{\Gamma\restriction\bool{B}}}
\]
is dense in $(\Gamma,\leq_\Gamma)$.\footnote{$H_{\lambda_\Gamma^+}\prec V^{\Gamma}$ means that $H_{\lambda_\Gamma^+}^V\prec V[G]$ for every $\Gamma$-generic filter $G$ over $V$, whereas $H_{\dot{\lambda}_\Gamma^+}^{\bool{B}}\prec V^{\Gamma\restriction\bool{B}}$ means that whenever
$G$ is $\mathcal{V}$-generic for $\Gamma\restriction\bool{B}$, in $V[G]$ there is $H$, $V$-generic for $\bool{B}$, such that
\[
H_{\dot{\lambda}_\Gamma^+}^{V[H]}\prec V[G].
\]}
%whenever $G$ is $\mathcal{V}$-generic for $\Gamma$.
\end{definition}

As we will see, the density of $D_\Gamma$ can be proved right away in $\MK$ for many classes of forcings $\Gamma$,
 hence 
 %$\Gamma$-absoluteness
 $\BCFA(\Gamma)$
holds once we force with some element of $D_\Gamma$.

Assume we have a definable class of forcings $(\Gamma,\to_\Gamma)$ with the following properties.\footnote{While parsing through these items,
the reader should keep in mind the case 
$\Gamma$ being $\SSP$ or the class of (semi)proper forcings with $\lambda_\Gamma=\omega_1$.}
\begin{itemize}
\item \emph{Pretameness:} $V^\Gamma\models\Ord=\lambda_\Gamma^+$. 
%\text{ is the second uncountable cardinal}$ and $\Gamma\subseteq\SSP$;

\item 
 \emph{Factor Lemma:} $\Gamma$ is closed under two-step iterations; therefore\footnote{As we will see below, it will be rather delicate to define $\dot{\Gamma}$ and to infer that $\bool{B}\ast\dot{\Gamma}$
 and $\Gamma\restriction\bool{B}$ define equivalent class forcings in $V$; this despite the 
 %evident 
 intuition that closure of $\Gamma$ under two-step iterations
 should grant it.}
 $\bool{B}\ast\dot{\Gamma}\cong \Gamma\restriction\bool{B}$ holds
for all cbas $\bool{B}\in\Gamma$;
 
\item \emph{Self-similarity:}
For stationarily many inaccessible cardinals $\alpha$,
\begin{itemize}
\item
%$\Gamma\cap V_\alpha=\Gamma_\alpha\in\Gamma$,
$\Gamma\cap V_\alpha=\Gamma_\alpha$ is such that $(\Gamma_\a, \leq_\Gamma)\in\Gamma$,
 and (as in the case $\alpha=\Ord$) 
$(\Gamma_\alpha,\leq_\Gamma)$ preserves the regularity of $\alpha$ making it the successor of $\lambda_\Gamma$, and
\item
the map 
$\bool{B}\mapsto \Gamma_\alpha\restriction \bool{B}$ 
regularly embeds $\Gamma_\alpha$  
into $\Gamma\restriction\Gamma_\alpha$.
%\item
%$(\Gamma_\alpha,\leq_\Gamma)$ is in $\Gamma$. %NOTE THIS FOLLOWS FROM 1ST ITEM ABOVE.
\end{itemize}

\item \emph{Universality:} 
For stationarily many inaccessible cardinals $\alpha$,
every $\bool{B}\in\Gamma_\alpha$ regularly embeds into $\Gamma_\alpha\restriction \bool{C}$
for some $\bool{C}\leq_\Gamma\bool{B}$ in $\Gamma_\alpha$. 
\end{itemize}
Then we would be able to replicate the same proof pattern we used to establish the consistency of projective absoluteness replacing $\Coll(\omega,{<}\Ord)$ by $\Gamma$ and $\Coll(\omega,{<}\alpha)$ by 
$\Gamma_\alpha$ 
to infer that
\[
D_\Gamma=\bp{\bool{B}\in\Gamma: H_{\dot{\lambda}_\Gamma^+}^{\bool{B}}\prec V^\Gamma}
\]
is dense in $(\Gamma,\leq_\Gamma)$, 
as witnessed by the forcings $\Gamma_\alpha\restriction\bool{C}$ for $\alpha$ inaccessible with
$\Gamma_\alpha\in\Gamma$, $V_\alpha\prec V$, and $\bool{C}\in\Gamma_\alpha$.

We also want to observe some other basic facts concerning the above conditions.

\begin{fact}
Assume $\Gamma$ is {\bf pretame}, satisfies the {\bf Factor Lemma}, the {\bf Self-similarity} condition, and the
{\bf Universality} condition. 
Then whenever $K$ is $\mathcal{V}$-generic for $\Gamma$ and 
$\bool{C}\in K$, in $V[K]$ there is a $V$-generic filter for $\bool{C}$.
\end{fact}
\begin{proof}
This is easily granted if $\bool{C}=\Gamma_\alpha\in K$ ($K\cap V_\alpha$ is $V$-generic for 
$\Gamma_\alpha$
by {\bf self-similarity}). Now assume 
$\bool{B}\geq_\Gamma \Gamma_\alpha\restriction\bool{C}$ for some $\Gamma_\alpha\restriction\bool{C}\in K$. Let
$i:\bool{B}\to \Gamma_\alpha\restriction\bool{C}$ witness 
$\bool{B}\geq_\Gamma \Gamma_\alpha\restriction\bool{C}$ ($i$ exists by {\bf universality}). Then
$i^{-1}[K]\in V[K]$ is $V$-generic for $\bool{B}$.
\end{proof}

The following observation is a first indication that these properties are closely related 
to generic absoluteness results and to forcing axioms.

\begin{fact} 
Assume $\Gamma$ is {\bf pretame}, satisfies the {\bf Factor Lemma}, the {\bf Self-similarity} condition, and the
{\bf Universality} condition. 
Let
$\bool{B}\in D_\Gamma$, and let $H$ be $V$-generic for $\bool{B}$. Then
\[
H_{\lambda_\Gamma^+}^{V[H]}\prec_{\Sigma_1} V[H]^{\bool{C}}
\]
for all $\bool{C}\in \Gamma^{V[H]}$.
\end{fact}
\begin{proof}
Notice that if $\bool{C}\leq_\Gamma \bool{B}$,
$K$ is $\mathcal{V}$-generic for $\Gamma$, with $H\in V[K]$ $V$-generic for $\bool{C}$, and
$G\in V[H]$ is $V$-generic for $\bool{B}$, we obtain the configuration:
\[
			\begin{tikzpicture}[xscale=1.5,yscale=-1.2]
				\node (A0_0) at (0, 0) {$H_{\lambda_\Gamma}^{V[G]}$};
				\node (A0_2) at (2, 0) {$V[K]$};
				\node (A1_1) at (1, 1) {$V[H]$};
				\path (A0_0) edge [->]node [auto] {$\scriptstyle{\Sigma_\omega}$} (A0_2);
				\path (A1_1) edge [->]node [auto,swap] {$\sqsubseteq$} (A0_2);
				\path (A0_0) edge [->]node [auto,swap] {$\sqsubseteq$} (A1_1);
			\end{tikzpicture}
		\]
This gives that any $\Sigma_1$-property with parameters in $H_{\lambda_\Gamma}^{V[G]}$ true in $V[H]$ remains true in $V[K]$ and thus reflects to $H_{\lambda_\Gamma}^{V[G]}$.
\end{proof}

%\begin{fact} 
%Assume $\Gamma$ is {\bf pretame}, satisfies the {\bf Factor Lemma}, the {\bf Self-similarity} condition, and the
%{\bf Universality} condition. 
%Suppose $\lambda_\Gamma=\omega_1$, let
%$\bool{B}\in D_\Gamma$, and let $H$ be $V$-generic for $\bool{B}$. Then
%\[
%H_{\omega_2}^{V[H]}\prec_{\Sigma_1} V[H]^{\bool{C}}
%\]
%for all $\bool{C}\in \Gamma^{V[H]}$.
%\end{fact}
%\begin{proof}
%Notice that if $\bool{C}\leq_\Gamma \bool{B}$,
%$K$ is $\mathcal{V}$-generic for $\Gamma$, with $H\in V[K]$ $V$-generic for $\bool{C}$, and
%$G\in V[H]$ is $V$-generic for $\bool{B}$, we obtain the configuration:
%\[
%			\begin{tikzpicture}[xscale=1.5,yscale=-1.2]
%				\node (A0_0) at (0, 0) {$H_{\omega_2}^{V[G]}$};
%				\node (A0_2) at (2, 0) {$V[K]$};
%				\node (A1_1) at (1, 1) {$V[H]$};
%				\path (A0_0) edge [->]node [auto] {$\scriptstyle{\Sigma_\omega}$} (A0_2);
%				\path (A1_1) edge [->]node [auto,swap] {$\sqsubseteq$} (A0_2);
%				\path (A0_0) edge [->]node [auto,swap] {$\sqsubseteq$} (A1_1);
%			\end{tikzpicture}
%		\]
%This gives that any $\Sigma_1$-property with parameters in $H_{\omega_2}^{V[G]}$ true in $V[H]$ remains true in $V[K]$ and thus reflects to $H_{\omega_2}^{V[G]}$.
%\end{proof}
In particular, the elements of $D_\Gamma$ are forcing strong versions of the corresponding bounded forcing axiom for 
$\Gamma$.

We now give rigorous definitions and outline how to infer the above properties of
$(\Gamma,\to^\Gamma)$ for a wide family of classes of forcings which includes the classes of proper, semiproper, and
$\SSP$ forcings.

\section{Forcing with forcings: definitions} \label{sec:wellbehavedclassdef}

\subsection{The Factor Lemma for $\Gamma$.}

%%%{The factor lemma holds for two-steps iterations}

In Subsection \ref{well-behaved classes} we will define the notion of absolutely well-behaved class. One of the things we will prove about such classes is the following.

\begin{lemma}[The Factor Lemma for $\Gamma$]
Let $\ap{V,\mathcal{V}}\models\MK$, and let
$\Gamma\in\mathcal{V}$ be a definable absolutely well-behaved class of forcings.
% closed under two-step iterations.
Let $\bool{B}\in\Gamma$ and
\[
\Gamma_\bool{B}=\bp{\dot{\bool{C}}\in\V^{\bool{B}}:\, 
\Qp{\phi_\Gamma(\dot{\bool{C}},\check{a}_\Gamma)}_{\bool{B}}=1_{\bool{B}}},
\]
\[
\to^\Gamma_{\bool{B}}=\bp{\dot{k}\in\V^{\bool{B}}:\, \Qp{\dot{k}:\dot{\bool{C}}\to\dot{\bool{D}}\text{ is $\Gamma$-correct}}_{\bool{B}}=1_{\bool{B}}}.
\]
Then whenever $G$ is $V$-generic for $\bool{B}$, we have that
$(\Gamma_\bool{B})^\circ_G=\Gamma^{V[G]}$ and
$(\to^\Gamma_\bool{B})^\circ_G=(\to^\Gamma)^{V[G]}$. 

Moreover, for any $\bool{C}\leq_\Gamma\bool{B}$ fix in $V$ $i_{\bool{C}}:\bool{B}\to\bool{C}$ witnessing this.
Then the map:
\begin{align*}
\Theta_\bool{B}:& \Gamma\restriction \bool{B}\to \bool{B}\ast \Gamma_{\bool{B}}^\circ\\
& \bool{C}\mapsto \bool{B}\ast(\bool{C}/_{i_{\bool{C}}[\dot{G}_\bool{B}]})
\end{align*}
defines a dense embedding of $(\Gamma\restriction\bool{B},\leq_\Gamma)$ into the class partial order
$\bool{B}\ast \Gamma_{\bool{B}}^\circ$ 
with order given by $(d,\dot{\bool{D}})\leq (c,\dot{\bool{C}})$ if and only if
$d\leq_\bool{B} c$ and
\[
\Qp{\exists \dot{k}:\dot{\bool{C}}\to\dot{\bool{D}}\text{ $\Gamma$-correct}}_\bool{B}\geq d.
\]
%where 
%are the canonical $\bool{B}$-names in $\mathcal{V}^{\bool{B}}$ for the objects and arrows of the category
%$(\Gamma^{V[G]},\to^{\Gamma^{V[G]}})$. 
%\[
%\dot{\to^\Gamma}=\bp{\dot{k}\in\V^{\bool{B}}:\, \Qp{k:\dot{\bool{C}}\to\bool{D}\text{ is $\Gamma$-correct}}_{\bool{B}}=1_{\bool{B}}}
%\]
\end{lemma}

We defer the proof to Section~\ref{subsec:prffactlem}.
%A key observation is that we 
We don't need the Factor Lemma for any of the proofs occurring in Sections 
\ref{sec:gammait}, \ref{sec:univGamma}, \ref{sec:proofGammadeltainGamma}, \ref{sec:gammafreezgammarig}.
On the other hand, this lemma is needed
% only
 in the proof of Theorem \ref{thm:mainth1}, which will be an easy corollary of 
all the results obtained in the above sections combined with the Factor Lemma, and it is also needed in the derivation of Corollary \ref{completeness}.

Despite the apparent self-evidence of the Factor Lemma, its proof requires a careful
formulation and unfolding of the relation existing between elements of $\Gamma^{V[G]}$ and
$\to^{\Gamma^{V[G]}}$ and their corresponding $\bool{B}$-names
${\Gamma}_\bool{B}$, 
$\to^{\Gamma}_\bool{B}$.
To appreciate the difficulties one may encounter when proving this lemma, observe that the following set of equalities holds true\footnote{Given 
 a $\bool{B}$-name $\dot{C}\in V^{\bool{B}}$ for a forcing notion, 
$i_{\dot{C}}:\bool{B}\to \RO(\bool{B}\ast\dot{C})$, $b\mapsto \ap{b,1_{\dot{\bool{C}}}}$ denotes the canonical embedding of $\bool{B}$ in the boolean completion of the two-step iteration $\bool{B}\ast\dot{C}$.}
for all $V$-generic filters $G$ for $\bool{B}$: 
\begin{align*}
(\Gamma_\bool{B})^\circ_G&=\bp{\dot{\bool{C}}_G:\dot{\bool{C}}\in\Gamma_{\bool{B}}}=\\
&=\bp{\dot{\bool{C}}_G:\Qp{\dot{\bool{C}}\in(\Gamma_{\bool{B}})^\circ}_{\bool{B}}=1_{\bool{B}}}=\\
&=\bp{\dot{\bool{C}}_G:\Qp{\dot{\bool{C}}\in(\Gamma_{\bool{B}})^\circ}_{\bool{B}}\in G}=\\
&=\bp{\bool{C}/_{i_{\dot{\bool{C}}}[G]}: \,
\bool{C}=\RO(\bool{B}\ast\dot{\bool{C}})\text{ and }\Qp{\phi_{\Gamma}(\dot{\bool{C}},\check{a}_\Gamma)}_{\bool{B}}=1_{\bool{B}}}=\\
&=\bp{\bool{C}/_{i[G]}: \,\bool{C}\in\Gamma\restriction\bool{B},\, i:\bool{B}\to\bool{C} \, 
\text{ is $\Gamma$-correct}}.
\end{align*}
Nonetheless, the equality between the terms in the second and third lines is a bit delicate to prove, and requires a 
careful reformulation of Cohen's forcing theorem.

It will be even more delicate to prove the corresponding set of equalities 
for the canonical $\bool{B}$-name $(\to^{\Gamma}_\bool{B})^\circ$ for $\to^{\Gamma^{V[G]}}$,
and to infer the other desired properties of the $\bool{B}$-names $({\Gamma}_\bool{B})^\circ$, 
$(\to^{\Gamma}_\bool{B})^\circ$.

\subsection{$\Gamma$-iterability} \label{sec:gammait}

There are two further key properties of $\Gamma$ we need to outline in order to infer the nice properties
for $\Gamma$ needed to establish the consistency of $\BCFA(\Gamma)$. 
%$\Gamma$-absoluteness.
Formulated in categorial terms, we need to have that $\Gamma$ is closed under set sized products and that
many $\Gamma$-valued diagrams have a colimit.
Formulated in the forcing terminology, we need the closure of $\Gamma$ under lottery sums, 
and an iteration theorem for $\Gamma$.

Let's first focus on the definition of the iterability property for forcings in $\Gamma$:

\begin{definition}
Assume $\Gamma$ is a definable class of pre-orders closed under two-step iterations, and 
$\lambda_\Gamma$ is a regular uncountable cardinal.\footnote{The reader may keep in mind
$\lambda_\Gamma=\omega_1$ and $\Gamma$ being the class of semiproper forcings 
while parsing through the definition.}

\begin{itemize}
\item $\Gamma$ has the \emph{Baumgartner property} if whenever $\delta$ is inaccessible and
\[
\mathcal{F}=\bp{i_{\alpha\beta}:\bool{B}_\alpha\to\bool{B}_\beta:\,\alpha\leq\beta<\delta}\subseteq \to^\Gamma\cap V_\delta
\]
is an iteration system with $\bool{B}_\alpha=\dirlim\mathcal{F}\restriction\alpha\in \Gamma$
for stationarily many $\alpha<\delta$, then $\varinjlim\mathcal{F}\in\Gamma$.
\item
$\Gamma$ is \emph{iterable} if $\Gamma$ has the Baumgartner property and
Player $II$ has a winning strategy $\Sigma(\Gamma)$ in the game
$\mathcal{G}(\Gamma)$ of length $\Ord$ between players $I$ and $II$ defined as follows: 
\begin{itemize}
\item
players $I$ and $II$  alternate playing $\Gamma$-correct \emph{injective}
homomorphism
$i_{\alpha,\alpha+1}:\bool{B}_\alpha\to\bool{B}_{\alpha+1}$;
\item
player $I$ plays at odd stages, player $II$ at even stages ($0$ and all limit ordinals are even);
\item
at stage $0$, $II$ plays the identity on the trivial cba $\2$.
%a $\Gamma$-correct injective embedding 
%$i_{0,1}:\2\to\bool{B}_{1}$ (i.e. a $\bool{B}_1\in\Gamma$);
\item
at limit stages $\eta$, $II$ must play\footnote{$\varinjlim\{\bool{B}_\alpha:\alpha<\eta\}$ is the direct 
limit 
of the iteration system given by the maps $i_{\gamma\beta}:\bool{B}_\gamma\to\bool{B}_\beta$ 
which are built along the
play of $\mathcal{G}(\Gamma)$ 
(see appendix~\ref{appendix} for the definition of direct and inverse limit of an 
iteration system).} 
a $\bool{B}_\eta\in\Gamma$ which admits for each $\alpha<\eta$ a $\Gamma$-correct
$i_{\alpha\eta}:\bool{B}_\alpha\to\bool{B}_\eta$ such that 
$i_{\beta,\eta}\circ i_{\alpha,\beta}=i_{\beta,\eta}$ for all $\alpha\leq\beta<\eta$;
moreover $II$ must play 
$\varinjlim\{\bool{B}_\alpha:\alpha<\eta\}$ at stage $\eta$
if:
\begin{itemize}
\item either $\cof(\eta)=\lambda_\Gamma$, 
\item or $\eta$ is inaccessible and 
$\{i_{\alpha\beta}:\bool{B}_\alpha\to\bool{B}_\beta:\,\alpha\leq\beta<\eta\}\subseteq \Gamma\cap V_\eta$;
\end{itemize}
\item 
$II$ wins $\mathcal{G}(\Gamma)$ if she can play at all stages.
\end{itemize}
\end{itemize}

\end{definition}

\begin{remark}
\emph{}

\begin{itemize}
\item
If $\Gamma$ has the Baumgartner property, $\delta>\lambda_\Gamma$ is inaccessible, and
\[
\mathcal{F}=\bp{i_{\alpha\beta}:\bool{B}_\alpha\to\bool{B}_\beta:\,\alpha\leq\beta<\delta}\subseteq \to^\Gamma\cap V_\delta
\]
is a play of $\mathcal{G}(\Gamma)$, $\varinjlim\mathcal{F}\in \Gamma$ is ${<}\delta$-CC, since:
\begin{itemize}
\item 
All $\bool{B}_\alpha$ are ${<}\delta$-CC having size less than $\delta$, and for all
$\alpha<\delta$ of cofinality $\lambda_\Gamma$, $\bool{B}_\alpha$ is the direct limit of
$\mathcal{F}\restriction\alpha$, and therefore Baumgartner's theorem \ref{iBaumgartner} applies.
\item
$\varinjlim\mathcal{F}\in\Gamma$ by the Baumgartner property.
\end{itemize}
In particular, if $\Gamma$ is closed under two-step iterations but is not iterable, and $\Sigma$ is a strategy for player $II$, a play 
\[
\mathcal{F}=\bp{i_{\alpha\beta}:\bool{B}_\alpha\to\bool{B}_\beta:\,\alpha\leq\beta<\delta}
\]
of the game $\mathcal{G}(\Gamma)$
which $II$ cannot win using $\Sigma$ is such that either 
\begin{itemize}
\item $\cof(\delta)<\delta$, or 
\item $\cf(\delta)=\lambda_\Gamma$
and $\varinjlim(\mathcal{F})\notin\Gamma$, or 
\item $\cf(\delta)\neq\lambda_\Gamma$ and there is no lower bound
in $\Gamma$ to the cbas played in the game before stage $\delta$,
or 
\item some $\bool{B}_\alpha$ has size bigger than $\delta$. % (and collapses $\delta$ to ahve cofinality smaller than .
\end{itemize}
Moreover $II$ can always play at successor stages of $\mathcal{G}(\Gamma)$.
\item The class $\Gamma$ of proper forcings is iterable: $II$ plays
the identity at all non-limit stages, the full limit at limit stages of countable cofinality, and the direct limit at limit
stages of uncountable cofinality.
\item With slightly more refined strategies one can prove that also the class of semiproper forcings is 
iterable, and that so is the class of stationary set preserving forcings assuming the existence of 
class many supercompact cardinals. We will address this issue with more care in
Section \ref{suitable classes}.
\end{itemize}
\end{remark}

\subsection{Universality of $(\Gamma,\leq_\Gamma)$ and $\Gamma$-rigidity.} \label{sec:univGamma}

What conditions grant that $(\Gamma,\leq_\Gamma)$ absorbs as a complete suborder any set sized 
$P\in\Gamma$?

The optimal case is that 
there is a complete embedding 

\[
i_{\bool{B}}:\bool{B}\to \Gamma\restriction \bool{B}
\]

for a dense set of $\bool{B}\in \Gamma$.

If this is the case, take $\bool{Q}\in \Gamma$, find $\bool{B}\leq_\Gamma\bool{Q}$ in the above dense set
and $i:\bool{Q}\to\bool{B}$ in $\to^\Gamma$ witnessing this.

Then $i_\bool{B}\circ i:\bool{Q}\to \Gamma\restriction \bool{B}$ will witness that
$\Gamma\restriction \bool{B}$ absorbs $\bool{Q}$ as well.

\small{
Now, given $\bool{B}\in\Gamma$, we have a natural candidate for a complete embedding $i_{\bool{B}}:\bool{B}\to \Gamma\restriction \bool{B}$: %(provided $\Gamma$ is closed under set sized products):
\begin{align*}
i_{\bool{B}}:&\bool{B}\to  \Gamma\restriction \bool{B}\\
&b\mapsto \bool{B}\restriction b.
\end{align*}

\begin{itemize}
\item
\textbf{$i_{\bool{B}}$ is order preserving:}  

If $b_1\leq b_0$, the map 
\begin{align*}
i_{b_1}:&\bool{B}\restriction b_0\to\bool{B}\restriction b_1\\
& c\mapsto c\wedge b_1
\end{align*}
is $\Gamma$-correct and witnesses that $\bool{B}\restriction b_0\geq_\Gamma\bool{B}\restriction b_1$.
\item
\textbf{$i_{\bool{B}}$ preserves sups:}

If $\{a_i:i\in I\}\subset \bool{B}^+$ is a maximal antichain in $\bool{B}$,
the product algebra
\[
\prod_{i\in I} (\bool{B}\restriction a_i)
\]
(the lottery sum of $\{\bool{B}\restriction a_i:i\in I\}$) is isomorphic to $\bool{B}$,
the top element of $\Gamma\restriction \bool{B}$ in $(\Gamma\restriction\bool{B},\leq_\Gamma)$.

\item
\textbf{PROBLEM:}
Does this map preserve incompatibility? In general NO!

Assume $\bool{B}$ is homogeneous (for example $\bool{B}$ is the boolean completion
of Cohen's forcing $2^{{<}\omega}$). 
Assume
$s,t$ are incompatible conditions in $\bool{B}$; 
by homogeneity, $\bool{B}\restriction s$ is isomorphic to $\bool{B}\restriction t$;
therefore the incompatible $s,t\in\bool{B}$ are mapped to the compatible conditions 
$\bool{B}\restriction s$, $\bool{B}\restriction t$ in
$(\Gamma,\leq_\Gamma)$.
\end{itemize}
To overcome this problem we need to find densely many highly inhomogeneous $\bool{B}\in\Gamma$.

Suppose for the moment that the map
\begin{align*}
i_{\bool{B}}:&\bool{B}\to  \Gamma\restriction \bool{B}\\
&b\mapsto \bool{B}\restriction b.
\end{align*}
defines a complete embedding (i.e. preserves the incompatibility relation), 
and pick a $\mathcal{V}$-generic $H$ for $\Gamma$
such that $\bool{B}\in H$.
Then $G=i_{\bool{B}}^{-1}[H]$ is $V$-generic for $\bool{B}$.

Suppose now that  
\[
\bp{\bool{B}: i_\bool{B}\text{ defines a complete embedding}}
\]
is dense in $(\Gamma,\leq\Gamma)$.
Assume $H$ is $\mathcal{V}$-generic for $\Gamma$.
Pick $\bool{Q}\in H$.
By density there is $\bool{B}\in H$ refining $\bool{Q}$ and such that
$i_{\bool{B}}$ defines a complete embedding of $\bool{B}$ into $\Gamma\restriction\bool{B}$.
Let $i:\bool{Q}\to\bool{B}$ witness $\bool{B}\leq_\Gamma\bool{Q}$.
Then $K=(i_\bool{B}\circ i)^{-1}[H]$ is $V$-generic for $\bool{Q}$.

In particular we get the following:
\begin{fact}
Assume
\[
E_\Gamma=\bp{\bool{B}\in\Gamma: \, i_\bool{B}:b\mapsto \bool{B}\restriction b\text{ defines a complete embedding of 
$\bool{B}$ into $\Gamma\restriction\bool{B}$}}
\]
is dense in $(\Gamma,\leq_\Gamma)$.

Then any $\mathcal{V}$-generic filter $H$ for $\Gamma$ adds a $V$-generic filter for any 
$\bool{Q}\in H$. 
\end{fact}

%\end{document}

%
%\begin{definition}
%$(\Gamma,\to^\Gamma)$ is well-behaved if:
%\begin{itemize}
%\item
%$\Gamma$ is closed under two-steps iterations,
%\item
%$\Gamma$ is closed under set sized products (if $\bp{\bool{B}_i:i\,\in I}\subseteq\Gamma$, $\prod_{i\in I}\bool{B}_i\in \Gamma$),
%\item
%$\Gamma$ is closed under complete subalgebras,
%\item
%$\Gamma$ is closed under images of complete homomorphisms (such homomomorphisms are all of the form
%\begin{align*}
%i_{b}:&\bool{B}\restriction\to\bool{B}\restriction b\\
%& c\mapsto c\wedge b
%\end{align*}
%---any complete homomorphisms $i:\bool{B}\to\bool{C}$ splits as $i\restriction (\bool{B}\restriction\coker(i))\circ i_{\coker(i)}$
%where $\neg\coker(i)=\ker(i)=\bigvee\bp{b\in\bool{B}: i(b)=0_{\bool{C}}}$.x
%\end{itemize}
%\end{definition}
%

%\begin{remark}
%$\SSP$, proper, semiproper, and in general any class of forcings defined by the preservation of a $\Pi_1$-property are well-behaved.
%
%
%CCC is not well behaved (it is not closed under lottery sums).
%\end{remark}
%

How do we get to the density of $E_\Gamma$? The key step is to reformulate properly the condition that 
$i_\bool{B}$ is 
a complete embedding.

%\subsection{$\Gamma$-rigidity}

\begin{definition} \label{def:gammarig}
Let $\Gamma$ be a definable class of cbas closed under two-step iterations, and let $\bool{B}\in \Gamma$.
$\bool{B}\in \Gamma$ is $\Gamma$-rigid if for 
$i_0,i_1:\bool{B}\to\bool{Q}$ in $\to^\Gamma$ we have that $i_0=i_1$. 
\end{definition}

%Note that (assuming $\Gamma_\delta\in \Gamma$ for some $\delta>|\bool{B}|$), 
%the above yields that the map $b\mapsto\bool{B}\restriction b$ is incompatibility preserving for $\leq_\Gamma$.
\begin{remark} \label{rmk:gammarigembed}
$\Gamma$-rigid cbas $\bool{B}$ are absorbed by 
$(\Gamma\restriction\bool{B},\leq_\Gamma)$
using the map $i_{\bool{B}}:b\mapsto\bool{B}\restriction b$. See the Lemma below.
\end{remark}

%This is immediate by the following Lemma:

\begin{lemma}\label{lem:eqtrGamma}
The following are equivalent characterizations of 
$\Gamma$-rigidity for an algebra $\bool{B}\in\Gamma$:
\begin{enumerate}
\item\label{lem:eqtrGamma1}
for all $b_0$, $b_1\in \bool{B}$ such that
$b_0\wedge_{\bool{B}}b_1=0_{\bool{B}}$,
$\bool{B}\restriction b_0$ is incompatible with $\bool{B}\restriction b_1$ in $(\Gamma,\leq_\Gamma)$. 
\item\label{lem:eqtrGamma2}
For every $\bool{C}\leq_\Gamma\bool{B}$ and every $V$-generic filter $H$ for 
$\bool{C}$, there is just one $\Gamma$-correct $V$-generic filter $G\in V[H]$ 
 for $\bool{B}$.
\item\label{lem:eqtrGamma3}
For all $\bool{C}\leq_\Gamma\bool{B}$ in $\Gamma$ there is only one $\Gamma$-correct
homomorphism
$i:\bool{B}\to\bool{C}$.
\end{enumerate}
\end{lemma}

%\begin{remark}
%%$\Gamma$--rigidity entails rigidity by its very definition. 
%%Nonetheless, 
%It is conceivable that even if $\bool{B}$ is $\Gamma$--rigid, there could be a complete (and non-surjective) homomorphism 
%$k:\bool{B}\to \bool{B}\restriction b$  
%which is not $\Gamma$--correct. If $H$ is $V$--generic for $\bool{B}$ with $b\in H$,
%$k^{-1}[H]=G\in V[H]$ is also $V$--generic for $\bool{B}$.
%Hence in $V[H]$ there could be distinct $V$--generic filters for $\bool{B}$ even if $\bool{B}$ is 
%$\Gamma$--rigid. This is not in conflict with~\ref{lem:eqtrGamma}(\ref{lem:eqtrGamma2}),
%since $G\in V[H]$ would not be $\Gamma$--correct for $\bool{B}$ in $V[H]$. 
%\end{remark}  

\begin{proof} 
We prove these equivalences by contraposition as follows:
\begin{description}
\item[~\ref{lem:eqtrGamma2} implies~\ref{lem:eqtrGamma1}] 
Assume~\ref{lem:eqtrGamma1} 
fails as witnessed by $i_j:\bool{B}\restriction b_j\to\bool{Q}$ for $j=0,1$ with
$b_0$ incompatible with $b_1$ in $\bool{B}$.
Pick $H$, a $V$-generic filter for $\bool{Q}$. Then $G_j=i_j^{-1}[H]\in V[H]$ (for $j=0$, $1$) are distinct  and $\Gamma$-correct
$V$-generic filters for $\bool{B}$ in $V[H]$,
since $b_j\in G_j\setminus G_{1-j}$.  

\item[~\ref{lem:eqtrGamma1} implies~\ref{lem:eqtrGamma3}]
Assume~\ref{lem:eqtrGamma3} fails for $\bool{B}$ as witnessed by $i_0\neq i_1:\bool{B}\to\bool{C}$.
Let $b$ be such that $i_0(b)\neq i_1(b)$. 
W.l.o.g.\ we can suppose that $r=i_0(b)\wedge i_1(\neg b)>0_{\bool{C}}$.
Then $j_0:\bool{B}\restriction b\to \bool{C}\restriction r$ and
$j_1:\bool{B}\restriction \neg b\to \bool{C}\restriction r$ given by $j_k(a)=i_k(a)\wedge r$ for $k=0,1$ and
$a$ in the appropriate domain witness that $\bool{B}\restriction \neg b$ and $\bool{B}\restriction b$
are compatible in $(\Gamma,\leq_\Gamma)$, i.e.\ that~\ref{lem:eqtrGamma1} fails.

\item[~\ref{lem:eqtrGamma3} implies~\ref{lem:eqtrGamma2}]
Assume~\ref{lem:eqtrGamma2} fails for $\bool{B}$ as witnessed by some 
$\bool{C}\leq_{\Gamma}\bool{B}$, a
$V$-generic filter $H$ for $\bool{C}$, and  
$\Gamma$-correct $V$-generic filters $G_1\neq G_2\in V[H]$ for $\bool{B}$.
Let $\dot{G}_1,\dot{G}_2\in V^{\bool{C}}$ 
be such that $(\dot{G}_1)_H=G_1\neq (\dot{G}_2)_H=G_2$
are $\Gamma$-correct $V$-generic filters for $\bool{B}$ in
$V[H]$ for both $j=1,2$.
%\begin{equation}\label{eqn:Gamma}
%\[
%\text{\emph{$(\dot{G}_1)_H=G_1\neq (\dot{G}_2)_H=G_2$
%are $\Gamma$-correct $V$-generic filters for $\bool{B}$ in
%$V[H]$ for both $j=1,2$.}}
%\]
%\end{equation}
Find $q\in G$ forcing that $b\in \dot{G}_1\setminus \dot{G}_2$ for some fixed $b\in\bool{B}$.
Then %, %by Lemma~\ref{lem:twostepchar}($2 \implies 3$),
for some $r\in H$ refining $q$, we have that both homomorphisms
$i_j=i_{\dot{G}_j,r}:\bool{B}\to\bool{C}\restriction r$ defined by 
$a\mapsto\Qp{\check{a}\in\dot{G}_j}_{\bool{C}}\wedge r$
are $\Gamma$-correct. However $i_1(b)=r=i_2(\neg b)$, and hence $i_1\neq i_2$
witness that~\ref{lem:eqtrGamma3} fails for $\bool{B}$ and $\bool{C}\restriction r$.
\end{description}
\end{proof}

\begin{fact}
The class of $\Gamma$-rigid cbas is closed under set-sized products (i.e. lottery sums),  
and the restriction operation $\bool{B}\mapsto\bool{B}\restriction b$ for $b\in\bool{B}^+$.
\end{fact}
\begin{proof}
We leave to the reader to check that for all $b\in\bool{B}^+$, $\bool{B}\restriction b$ is $\Gamma$-rigid if so is $\bool{B}$.

Assume now that $\bp{\bool{B}_i:i\in I}$ is a family of $\Gamma$-rigid cbas.

Towards a contradiction, assume $k_j:(\prod_{i\in I} \bool{B}_i)\to \bool{C}$ are distinct and $\Gamma$-correct for $j=0,1$.
if $k_0\restriction \bool{B}_i=k_1\restriction\bool{B}_i$ for all $i\in I$, we get that $k_0=k_1$, which is a contradiction.
Hence for some $i\in I$, $k_0\restriction \bool{B}_i\neq k_1\restriction\bool{B}_i$.

Then for some $a\in\bool{B}_i$, $k_0(a)\neq k_1(a)$, therefore $k_0(a)\Delta k_1(a)>0_{\bool{C}}$.
So either $k_0(a)\wedge k_1(\neg a)=c>0_{\bool{C}}$ or
$k_1(a)\wedge k_0(\neg a)=c>0_{\bool{C}}$. In either cases
we get that $\bool{B}_i$ is not $\Gamma$-rigid as witnessed by the distinct $\Gamma$-correct maps (for $i=0,1$)
$k^*_i:b\mapsto k_i(b)\wedge c$ with domain $\bool{B}_i$ and range $\bool{C}\restriction c$.
\end{proof}

\begin{lemma}\label{lem:copyingdenseintoant}
Let $\Gamma$ be closed under set-sized products (i.e. set-sized lottery sums) 
and such that the $\Gamma$-rigid forcings are dense in 
$(\Gamma,\leq_\Gamma)$.
Assume $D$ is a dense open class of $(\Gamma,\leq_\Gamma)$.

Then for all $\bool{B}\in \Gamma$, there is a $\Gamma$-rigid $\bool{C}\leq^*_\Gamma\bool{B}$
and a maximal antichain $A$ of $\bool{C}$ such that $k_{\bool{C}}[A]\subseteq D$.
\end{lemma}

\begin{proof}
Given $\bool{B}\in\Gamma$, find a $\Gamma$-rigid
$\bool{C}_0\leq_\Gamma\bool{B}$ with $\bool{C}_0\in D$, with
$k_0:\bool{B}\to\bool{C}_0$ a witness that $\bool{C}_0\leq_\Gamma\bool{B}$ and
\[
b_0=\coker(k_0).%=\neg\bigvee\bp{a\in\bool{B}: k_{\bool{B}\bool{C_0}}(a)=0_{\bool{C}_0}}.
\]

Then 
\begin{align*}
k_0\restriction b_0:&\bool{B}\restriction b_0\to\bool{C_0}\\
&b\mapsto k_{0}(b)
\end{align*}
is injective and $\Gamma$-correct and witnesses $\bool{B}\restriction b_0\geq^*_\Gamma\bool{C}_0$.

Now find $\bool{C}_1\leq_\Gamma\bool{B}\restriction \neg b_0$ $\Gamma$-rigid and in $D$
with
$k_1:\bool{B}\restriction\neg b_0\to\bool{C}_1$ a witness that $\bool{C}_0\leq_\Gamma\bool{B}$ and
\[
b_1=\coker(k_1)\leq\neg_{\bool{B}} b_0.%=\neg\bigvee\bp{a\in\bool{B}: k_{\bool{B}\bool{C_0}}(a)=0_{\bool{C}_0}}.
\]
Then $k_{1}\restriction b_1:\bool{B}\restriction b_1\to\bool{C}_1$ 
is injective.

Continuing this way we construct by induction a maximal antichain 
$E=\bp{b_\alpha:\alpha<\gamma}$ of $\bool{B}$
such that for all $\alpha<\gamma$ there is 
$b_{\alpha}\leq\neg \bigvee_{\beta<\alpha}b_\beta$, $\bool{C}_\alpha\in D$,
and $k_\alpha: \bool{B}\restriction b_\alpha\to\bool{C}_\alpha$ witnessing
$\bool{C}_\alpha\leq^*_\Gamma \bool{B}\restriction b_\alpha$.

Let 
\begin{align*}
k:&\bool{B}\to\prod_{\alpha<\gamma}\bool{C}_\alpha\\
&b\mapsto \ap{k_\alpha(b\wedge b_\alpha):\alpha<\gamma}.
\end{align*}

Then $\bool{C}\leq^*_\Gamma\bool{B}$ as witnessed by $k$,  and
$A=\bp{c_\alpha=k(b_\alpha):\alpha<\gamma}=k[E]$ is a maximal antichain of $\bool{C}$
 such that for all $\alpha<\gamma$:
\begin{itemize}
\item
$k(b_\alpha)=
\ap{0_{\bool{C}_0},\dots,0_{\bool{C}_\eta},\dots,1_{\bool{C}_\alpha},0_{\bool{C}_{\alpha+1}},\dots,0_{\bool{C}_\xi},\dots\dots}\in A$;
\item 
$\bool{C}_\alpha\cong\bool{C}\restriction k(b_\alpha)=k_{\bool{C}}(c_\alpha)\in D$.
\end{itemize}  
Finally notice that $\bool{C}$ is $\Gamma$-rigid, being the lottery sum of $\Gamma$-rigid forcings.
\end{proof}

It can also be shown that $\Gamma$-rigidity is preserved by passing to generic quotients; i.e. 

\begin{quote}
\emph{Assume $\bool{C}\in V$ is $\Gamma$-rigid, $G$ is $V$-generic for $\bool{B}$, and
$i:\bool{B}\to\bool{C}$ is a $\Gamma$-correct homomorphism.
Then $\bool{C}/_{i[G]}$ is $\Gamma^{V[G]}$-rigid in $V[G]$.} 
\end{quote}

But this fact (which we don't need) has a very convoluted proof  (it essentially amounts to a different proof of the Factor Lemma), so we omit it.

%The proof of the following fact is left to the reader.

\subsection{Well-behaved classes $\Gamma$}\label{well-behaved classes}

We can now give the key definitions and state the main theorem which will be repeatedly used in Section 
\ref{suitable classes}.

\begin{definition}
Let $(V,\mathcal{V})$ be the standard model of $\MK$.
Given a definable class of cbas $\Gamma\in\mathcal{V}$ closed under two-step iterations:
\begin{itemize}

\item
$(\Gamma,\leq_\Gamma)$ is \emph{strategically ${<}\Ord$-closed} if $\lambda_\Gamma$
is a regular uncountable cardinal, and $\Gamma$ is
iterable with an iteration strategy $\Sigma(\Gamma)$ definable in the $\ZFC$-model $(V,\in)$.

\item
$\Gamma$ is \emph{closed under lottery sums} 
if any set-sized product of cbas in $\Gamma$ is in $\Gamma$.

\item $\Gamma$ is \emph{closed under isomorphisms} if $\bool{C}\in\Gamma$ whenever $\bool{B}\in\Gamma$ and $\bool{C}\cong\bool{B}$.

\item
$\Gamma$ is \emph{closed under restrictions and complete subalgebras}
%$\Gamma$ is \emph{closed under restrictions and complete homomorphisms}
%\emph{closed under complete homomorphisms and regular subalgebras}
 if for every
%all 
$\bool{B}\in\Gamma$:
\begin{itemize}
\item for every $b\in\bool{B}$, the map 
\begin{align*}
k_b:&\bool{B}\to\bool{B}\restriction b\\
&c\mapsto c\wedge b
\end{align*}
is $\Gamma$-correct, and
\item
any complete subalgebra of $\bool{B}$ is in $\Gamma$.
\end{itemize}
\end{itemize}
\end{definition}

The following is the key definition of the paper.

\begin{definition}\label{def:wellbeh}
Assume $(V,\mathcal{V})$ is a model of $\MK+$\emph{ $\Ord$ is Mahlo}.

A definable class $\Gamma$ is \emph{well-behaved in $V$} if:
\begin{enumerate}
%\item
%The statement $\bool{B}\in \Gamma$ provably reflects to all inaccessible cardinals $\delta$ with $V_\delta$ a model of $T$.
%More precisely the following is provable in $\ZFC$
%$\MK+$\emph{ $\Ord$ is Mahlo} prova
%for all inaccessible cardinals $\delta$ and $\bool{B}\in V_\delta$ 
%\[
%
%\]
\item\label{def:wellbeh1}
$\lambda_\Gamma$ is a regular uncountable cardinal.
\item \label{def:wellbeh2}
$\Gamma$ is closed under isomorphisms, two-step iterations, lottery sums, restrictions, and complete subalgebras.
%two-step iterations, lottery-sums, complete subalgebras, and
%complete homomorphisms.
\item \label{def:wellbeh3}
For all $\bool{B}\in\Gamma$ and $G$ $V$-generic for $\bool{B}$,
$V[G]$ models that $\Gamma^{V[G]}$ is strategically ${<}\Ord$-closed.
\item \label{def:wellbeh5}
$\Gamma$ contains as elements all $<\lambda_\Gamma$-closed forcings.
%the ${\leq}\lambda_\Gamma$-closed forcings.
\item \label{def:wellbeh6}
For all inaccessible cardinals $\delta>|a_\Gamma|$
%\begin{itemize}
%\item 
and for all $\bool{B}\in V_\delta$,
\[
V_\delta\models\phi_\Gamma(\bool{B},a_\Gamma)\quad\text{ if and only if } \quad V\models\phi_\Gamma(\bool{B},a_\Gamma).
\]
%\end{itemize}
\item \label{def:wellbeh4}
The $\Gamma$-rigid cbas are dense in $(\Gamma,\leq_\Gamma)$.
\end{enumerate}

$\Gamma$ is \emph{absolutely well-behaved} if $\Gamma^{V[G]}$ is well-behaved in $V[G]$ for any
$V$-generic filter $G$ for some $\bool{B}\in \Gamma$.

\end{definition}

\begin{remark}\label{sufficience-remark} In all classes $\Gamma$ we will consider, clause \ref{def:wellbeh6} in Definition \ref{def:wellbeh} will be proved by showing that $\Gamma$ can be defined both by a $\Sigma_2$ formula and by a $\Pi_2$ formula, possibly with parameters.   
\end{remark}

\begin{remark}
\emph{}

\begin{itemize}
\item
We will show that $\SSP$, proper, semiproper and many other well-known classes of forcings contained in
$\SSP$
%and in general any class of forcings defined by the preservation of a $\Pi_1$-property 
are absolutely well-behaved (in some cases assuming large cardinals in $V$).

The key point is that being well-behaved for all these $\Gamma$ is provable in $\ZFC$ (+ large cardinals).
The only property of well-behavedness not covered elsewhere in the literature for these classes of forcings
is the density of $\Gamma$-rigid forcings.

\item
The class $\Gamma$ given by CCC forcings is not well-behaved; for example it is not closed under lottery sums (easy), and it does not have $\Gamma$-rigid elements (less straightforward).
%(more delicate).
\end{itemize}
\end{remark}

The following is one of the main results of this paper.

\begin{theorem}\label{thm:mainth1}
Assume 
%\[
%\ap{V,\mathcal{V}}\models\MK+\Ord\text{ is Mahlo}+ 
%%\bp{\bool{B}: \bool{B}\text{ is $\sigma$-closed}}\subseteq
%\lambda_\Gamma\text{ is a regular uncountable cardinal}+\Gamma\text{ is absolutely well-behaved}.
%\]
$\ap{V,\mathcal{V}}$ satisfies 
\begin{itemize}
\item $\MK$, 
%\item $\Ord\text{ is Mahlo}$,
%\bp{\bool{B}: \bool{B}\text{ is $\sigma$-closed}}\subseteq
\item $\lambda_\Gamma\text{ is a regular uncountable cardinal}$, and 
\item $\Gamma\text{ is absolutely well-behaved}$.
\end{itemize}

Then%:
%\[
%(V,\mathcal{V},\in)\models\phi_\Gamma(\Gamma,a_\Gamma)
%\] 
%and
\begin{quote}
$V^\Gamma$ models that $\Ord$ is the successor of $\lambda_\Gamma$.
\end{quote}

Moreover, for all inaccessible $\delta$ such that
$V_\delta\prec V$:
\begin{enumerate}

\item \label{thm:mainth1-1}
$\Gamma_\delta\in\Gamma$ 
is $\Gamma$-rigid,  
and is such that for all $\bool{B}\in\Gamma_\delta$ there is
$\bool{C}\in\Gamma_\delta$ such that $\Gamma_\delta\restriction\bool{C}\leq_\Gamma\bool{B}$.

Hence the class of $\Gamma$-rigid forcings is pre-dense as witnessed by the forcings
$\Gamma_\delta$ as $\delta$ ranges on the inaccessible cardinals with $V_\delta\prec V$. 

\item \label{thm:mainth1-2}
$\Gamma_\delta$
preserves the regularity of $\delta$ and 
forces it to become the successor of $\lambda_\Gamma$.

\item \label{thm:mainth1-3}
If $\lambda_\Gamma=\omega_1$,
$\Gamma_\delta$ forces
$\BFA(\Gamma)$,  and if $\delta$ is supercompact, it forces also $\FA_{\omega_1}(\Gamma)$.

\item \label{thm:mainth1-4}
For all $G$ $V$-generic for $\Gamma$ with $\Gamma_\delta\in V[G]$, letting 
$G_\delta=G\cap V_\delta$ we have that
$H_{(\lambda_\Gamma)^+}^{V[G_\delta]}=V_\delta[G_\delta]\prec V[G]=
H_{(\lambda_\Gamma)^+}^{V[G]}$.
\end{enumerate}
\end{theorem}

\begin{corollary}
Assume 
%\[
%\ap{V,\mathcal{V}}\models\MK+\Ord\text{ is Mahlo}+ 
%%\bp{\bool{B}: \bool{B}\text{ is $\sigma$-closed}}\subseteq
%\lambda_\Gamma\text{ is a regular uncountable cardinal}+\Gamma\text{ is well-behaved}.
%\]
$\ap{V,\mathcal{V}}$ satisfies 
\begin{itemize}
\item $\MK$, 
\item $\Ord\text{ is Mahlo}$,
%\bp{\bool{B}: \bool{B}\text{ is $\sigma$-closed}}\subseteq
\item $\lambda_\Gamma\text{ is a regular uncountable cardinal}$, and 
\item $\Gamma\text{ is absolutely well-behaved}$.
\end{itemize}
Then
\[
D_\Gamma=\bp{\bool{B}\in\Gamma: \, H_{\dot{\lambda}_\Gamma^+}^\bool{B}\prec 
H_{\dot{\lambda}_\Gamma^+}^{\Gamma\restriction\bool{B}}}
\]
is dense as witnessed by $\Gamma_\delta\restriction\bool{C}$ as $\delta$ ranges among
the inaccessible cardinals $\gamma$ with
$$(V_\gamma,\in)\prec (V,\in)$$ and $\bool{C}$ among the elements of $\Gamma_\delta$.

In particular, for every $\bool{B}\in\Gamma$ there is some $\bool{C}\leq_\Gamma \bool{B}$ such that $\bool{C}$ forces  
$\BCFA(\Gamma)$.

%Consequently, 
%%$\Gamma$-absoluteness 
%$\BCFA(\Gamma)$ is a consistent axiom.
\end{corollary}

The next section is devoted to the proof of Theorem \ref{thm:mainth1}.
We will pay special attention to giving detailed proofs of \ref{thm:mainth1}(\ref{thm:mainth1-1}) and
 \ref{thm:mainth1}(\ref{thm:mainth1-2}). The reader familiar with these proofs will be able to fill
in the details needed to prove the remaining assertions of Theorem \ref{thm:mainth1}.
  
  The key result is \ref{thm:mainth1}(\ref{thm:mainth1-1}).

\begin{notation}
Given a category forcing $(\Gamma,\leq_\Gamma)$ with $\Gamma$ a definable class of complete boolean algebras and 
$\leq_\Gamma$ the order induced on $\Gamma$
by the $\Gamma$-correct homomorphisms between elements of $\Gamma$,
we denote the incompatibility relation with respect to $\leq_\Gamma$ by $\bot_\Gamma$, 
and the subclass of $\Gamma$
given by its
$\Gamma$-rigid elements by $\bool{Rig}^\Gamma$.
\end{notation}

%\end{document}
%\subsection{Why $\Gamma_\delta\in\Gamma$ if $\delta$ inaccessible and $V_\delta\prec V$.}

\section{Forcing with forcings: proofs} \label{sec:mainresultsM}

\subsection{Why $\Gamma_\delta\in\Gamma$?} \label{sec:proofGammadeltainGamma}

\begin{notation}
Given a $\Gamma$-rigid forcing $\bool{B}$:
\begin{itemize}
\item
$k_\bool{B}:b\mapsto\bool{B}\restriction b$ is the canonical embedding of $\bool{B}$ into $\Gamma\restriction\bool{B}$.

\item
if $\bool{C}\leq_\Gamma\bool{B}$, $k_{\bool{B}\bool{C}}$ denotes the unique 
$\Gamma$-correct homomorphism of
$\bool{B}$ into $\bool{C}$.
\end{itemize}
\end{notation}

\begin{theorem}\label{thm:univGamma}
Assume $\Gamma$ is absolutely well-behaved, and $\delta$ is inaccessible and such that $V_\delta\prec V$.
Then $\Gamma_\delta\in \Gamma$ preserves the regularity of $\delta$ making it the successor of $\lambda_\Gamma$.
\end{theorem}

Notice that Theorem \ref{thm:univGamma} proves Theorem \ref{thm:mainth1}(\ref{thm:mainth1-1}) (except for the assertion that
for all $\bool{B}\in\Gamma_\delta$ there is
$\bool{C}\in\Gamma_\delta$ such that $\Gamma_\delta\restriction\bool{C}\leq_\Gamma\bool{B}$) and 
Theorem \ref{thm:mainth1}(\ref{thm:mainth1-2}).

\begin{proof}
Let $P_\Gamma^\delta$ be the set of $\mathcal{F}=\bp{i_{\alpha,\beta}:\bool{B}_\alpha\to\bool{B}_\beta:\alpha\leq\beta<\eta}$ subsets of
$\to^\Gamma\cap V_\delta$  such that:
\begin{itemize}
\item $\mathcal{F}\in V_\delta$ is a partial play of $\mathcal{G}_\delta(\Gamma)$ according to $\Sigma_\delta(\Gamma)$;
\item
all the moves of $I$ in $\mathcal{F}$
are $\Gamma$-rigid forcings;
%\item $\bool{B}_\alpha\leq^*_\Gamma\bool{B}_\beta$ iff $\alpha\leq\beta<\eta$;
\item
$\mathcal{G}\leq \mathcal{F}$ if $\mathcal{G}$ is an end-extension of $\mathcal{F}$.
\end{itemize}

Since $\Gamma$ is iterable, it is immediate to check that
$P_\Gamma^\delta$ is ${<}\delta$-closed, hence in $\Gamma$, since $\delta\geq\lambda_\Gamma$ and $\Gamma$ is well-behaved
(by clause \ref{def:wellbeh}(\ref{def:wellbeh5})).

Moreover, let $\mathcal{G}=\bp{\bool{B}_\alpha:\alpha<\delta}$ be (the union of) a $V$-generic $G$ filter for $P_\Gamma^\delta$.
Then $\mathcal{G}$ is an iteration system in $V[G]$, and it is clear that $V[G]=V[\mathcal{G}]$.

%We claim that $\mathcal{G}$ is an iteration system of cbas in $\Gamma^{V[G]}$.
%Toward this aim it suffice to show that 
%\begin{equation}
%V_\delta[G]=V_\delta\\
%V_\delta[G]\prec V[G].
%\end{equation}

Since $P_\Gamma^\delta$ is ${<}\delta$-closed, $V_\delta^{V[G]}=V_\delta\prec V$. 
In particular $\delta$ is inaccessible in $V[G]$, hence
\[
\Gamma^{V[G]}_\delta=\Gamma^{V[G]}\cap V_\delta^{V[G]}=\Gamma\cap V_\delta
\]
where the equalities hold because of clause \ref{def:wellbeh}(\ref{def:wellbeh6}) in the definition of well-behaved class
applied in $V$ and in $V[G]$.

Now $\Gamma^{V[G]}$ is iterable in $V[\mathcal{G}]$ (by clause \ref{def:wellbeh}(\ref{def:wellbeh3})), hence it has the
Baumgartner property in $V[G]$. 
Observe that for all $\xi<\delta$ of cofinality $\lambda_\Gamma$, $\bool{B}_\xi$ is the direct limit of
$\mathcal{G}\restriction\xi$, and the set of such $\xi$ is stationary in $\delta$ in $V[G]$ since $P_\Gamma^\delta$ is ${<}\delta$-closed.
%the set of such $\xi$ is stationary in $\delta$ in $V[G]$ since $V_\delta^{V[G]}=V_\delta$.

By the Baumgartner property of $\Gamma^{V[G]}$ and Theorem \ref{iBaumgartner} applied in $V[G]$, letting
$\bool{B}_{\mathcal{G}}\in V[\mathcal{G}]$ be the direct limit of the iteration system $\mathcal{G}$,
\[
V[\mathcal{G}]\models\bool{B}_{\mathcal{G}}\text{ is in }\Gamma^{V[G]}\text{ and is }{<}\delta\text{-CC}.
\] 

Therefore
$V\models P_\Gamma^\delta\ast\bool{B}_{\dot{\mathcal{G}}}$ is in $\Gamma$ (by clause \ref{def:wellbeh}(\ref{def:wellbeh2})) and preserves the regularity of $\delta$, being a two-step iteration of a ${<}\delta$-closed
forcing with a ${<}\delta$-CC forcing.

%
%\item
%For all $D$ dense open subset of $\Gamma_\delta$, there is $\bool{B}_\eta\in \mathcal{G}$
%and $A\subseteq\bool{B}_\eta$ maximal antichain such that 
%$k_{\bool{B}_\eta\Gamma_\delta}[A]\subseteq D$.

\begin{claim}
Let $H$ be $V[\mathcal{G}]$-generic for $\bool{B}_{\mathcal{G}}$.
Then 
\[
\bp{\bool{B}_\alpha\restriction f(\alpha):\, 
f\in H,\, \alpha<\delta}
\] 
is $V$-generic for $\Gamma_\delta$.
\end{claim}

By the Claim we get that 
$\Gamma_\delta\sqsubseteq P^\delta_\Gamma\ast\dot{\bool{B}}$,
hence $\Gamma_\delta$ is in $\Gamma$ by clause \ref{def:wellbeh}(\ref{def:wellbeh2}).

Therefore it suffices to prove the Claim to conclude that $\Gamma_\delta$ is in $\Gamma$ and 
preserves the regularity of $\delta$.

\begin{proof}
We leave to the reader to check that
\[
\bp{\bool{B}_\alpha\restriction f(\alpha):\, 
f\in H,\, \alpha<\delta}
\] 
is a filter on $\Gamma_\delta$. We need to prove that this filter is $V$-generic.

The key to the proof is the following:
\begin{subclaim}
For any $D$ dense open subset of  $\Gamma_\delta$,
the set of partial plays 
$\mathcal{F}=\bp{\bool{B}_\eta:\eta<\alpha}\in P_\Gamma^\delta$ such that there is $\xi<\alpha$ 
and $A$ maximal antichain of $\bool{B}_\xi$ with
\[
\bp{\bool{B}_{\xi}\restriction a:\, a\in A}=k_{\bool{B}_{\xi}}[A]\subseteq D
\]
is dense open in $P_\Gamma^\delta$.
\end{subclaim}

Assume the subclaim holds. Then for $D$ dense open subset of $\Gamma_\delta$, there is 
$\mathcal{F}=\bp{\bool{B}_\eta:\eta<\alpha}\in\mathcal{G}$,
$\xi<\alpha$ and $A$ maximal antichain of $\bool{B}_\xi$ such that
\[
\bp{\bool{B}_{\xi}\restriction a:\, a\in A}\subseteq D.
\]
Now
\[
\bp{ f(\xi):\, f\in H}
\] 
is $V[G]$-generic for $\bool{B}_\xi$, and $A$ is still a maximal antichain of $\bool{B}_\xi$ in $V[G]$.
Therefore for some $a\in A$ and $f\in H$, $f(\xi)\leq a$. Hence
$\bool{B}_\xi\restriction a\in D\cap H$,
proving the claim.

We prove the subclaim:
\begin{proof}
Assume $\bp{\bool{B}_\xi:\,\xi<\alpha}\in P_\Gamma^\delta$ and $D$ is dense open in $\Gamma_\delta$. 

Then $\bp{\bool{B}_\xi:\xi<\alpha}$ is a play according to $\Sigma(\Gamma)$.
Notice that there is some freedom to decide what $\bool{B}_\xi$ is only for odd $\xi$ and for $\bool{B}_0$, because the even stages
are decided by the winning strategy $\Sigma(\Gamma)$ for player $II$.
W.l.o.g. (by prolonging $\bp{\bool{B}_\xi:\xi<\alpha}$ if necessary) we may assume that $\alpha$ is odd so that it is $I$'s turn to play.
This gives that $\alpha=\beta+1$.
Then (by Lemma \ref{lem:copyingdenseintoant}) %for all $\bool{C}\leq^*_\Gamma \bool{B}_\beta$, 
there is 
$\bool{B}_{\beta+1}\leq^*_\Gamma \bool{B}_\beta$ which is $\Gamma$-rigid and such that 
some $A\subseteq \bool{B}_{\beta+1}$ is a maximal antichain with $k_{\bool{B}_{\beta+1}}[A]\subseteq D$.
By definition of $P_\Gamma^\delta$,
$\bp{\bool{B}_\xi:\xi\leq\beta}\cup\bp{\bool{B}_{\beta+1}}\in P_\Gamma^\delta$ and 
\[
\bp{B_{\bool{B}_{\beta+1}}\restriction a:\, a\in A}=k_{\bool{B}_{\beta+1}}[A]\subseteq D.
\]
\end{proof}
The Claim is proved.
\end{proof}

To conclude the proof of the Theorem, we are left with proving the following.
\begin{claim}
$\Gamma_\delta$ makes $\delta$ the successor of $\lambda_\Gamma$.
\end{claim}
\begin{proof}
Since $\lambda_\Gamma$ is preserved by all forcings in $\Gamma$, we get that 
$\lambda_\Gamma$ is preserved by $\Gamma_\delta$.
We also know that $\delta$ is a regular cardinal of $V[G]$ whenever $G$ is $V$-generic for $\Gamma_\delta$.

We must show that $\delta$ is the successor of $\lambda_\Gamma$ in $V[G]$.

For any ordinal $\alpha\geq\lambda_\Gamma$ and $\bool{B}\in \Gamma$,
$V$ models that $\bool{B}\ast\dot{\Coll}(\lambda_\Gamma,\alpha)\in\Gamma$ 
(since $\Gamma$ contains all $\lambda_\Gamma$-closed forcings and is closed under two-step iterations);
we easily get (since $\delta$ is inaccessible) that
$\bool{B}\ast\dot{\Coll}(\lambda_\Gamma,\alpha)\in\Gamma_\delta$ for all $\lambda_\Gamma\leq\alpha<\delta$.

In particular, for any  $\lambda_\Gamma\leq\alpha<\delta$, 
the set $D_\alpha$ of $\bool{C}\in \Gamma_\delta$ which collapse $\alpha$ to have size $\lambda_\Gamma$
is dense open in $\Gamma_\delta$.

Since $V_\delta\prec V$, also the set $\bool{Rig}^\Gamma\cap D_\alpha$ is dense in $\Gamma_\delta$
for all $\alpha<\delta$.
By Remark \ref{rmk:gammarigembed} (and using $V_\delta\prec V$)
$b\mapsto \bool{B}\restriction b$ is a complete embedding\footnote{We do not as yet say that it is also a $\Gamma$-correct embedding; this is indeed the case but to infer it we need the Factor Lemma, whose proof is not as yet granted.} 
of $\bool{B}$ into 
$\Gamma_\delta\restriction \bool{B}$ for any $\bool{B}\in \bool{Rig}^\Gamma$.
In particular, if $G$ is $V$-generic for $\Gamma_\delta$, in $V[G]$ there is a $V$-generic filter $H$ for some 
$\bool{B}\in D_\alpha$ for any $\alpha<\delta$; any such generic filter $H$ 
adds a surjection of $\lambda_\Gamma$ onto
$\alpha$ existing in $V[G]$. We are done.
\end{proof}

The theorem is proved.
\end{proof}

%
%(i.e. $\bool{C}$ is what $\Sigma(\Gamma)$ requires $II$ to play against $\bp{\bool{B}_\xi:\xi<\alpha}$). 
%By density we can find 
% 
% Find a $\Gamma$-rigid $\bool{C}\in D$ and $A\subseteq

\subsection{$\Gamma$-freezeability versus $\Gamma$-rigidity} \label{sec:gammafreezgammarig}

It will be convenient in order to establish that a certain class is well-behaved to prove that it satisfies
a clause weaker than \ref{def:wellbeh}(\ref{def:wellbeh4}): the $\Gamma$-freezing property.

\begin{definition}
Let $\Gamma$ be a definable class of cbas closed under two-step iterations, and $\bool{B}\in \Gamma$.
A $\Gamma$-correct $k:\bool{B}\to\bool{C}$ is \emph{$\Gamma$-freezing} if
for all
$i_0,i_1:\bool{C}\to\bool{Q}$ in $\to^\Gamma$ we have that $i_0\circ k=i_1\circ k$ 
(i.e. if the map $b\mapsto\bool{C}\restriction k(b)$ is incompatibility preserving for $\leq_\Gamma$).
\end{definition}

We can give the following characterizations of $\Gamma$-freezeability, the proof of which is along the same lines of the proof of Lemma~\ref{lem:eqtrGamma} and is left to the reader.

\begin{lemma}\label{lem:eqfrGamma}
Let $k:\bool{B}\to\bool{Q}$ be a $\Gamma$-correct homomorphism.
The following are equivalent:
\begin{enumerate}
\item\label{lem:eqfrGamma1}
For all $b_0,b_1\in \bool{B}$ such that
$b_0\wedge_{\bool{B}}b_1=0_{\bool{B}}$ we have that
$\bool{Q}\restriction k(b_0)$ is incompatible with $\bool{Q}\restriction k(b_1)$ in 
$(\Gamma,\leq_\Gamma)$. 
\item\label{lem:eqfrGamma2}
For every $\bool{R}\leq_\Gamma\bool{Q}$ and every
$V$-generic filter $H$ for $\bool{R}$, there is just one  
$\Gamma$-correct $V$-generic filter $G\in V[H]$ for $\bool{B}$ such that $G=k^{-1}[K]$ for all $\Gamma$-correct $V$-generic 
filters $K\in V[H]$ for $\bool{Q}$.
\item\label{lem:eqfrGamma3}
For all $\bool{R}\leq_\Gamma\bool{Q}$ in $\Gamma$ 
and $i_0,i_1:\bool{Q}\to\bool{R}$ witnessing that $\bool{R}\leq_\Gamma\bool{Q}$ we have that
$i_0\circ k=i_1\circ k$.
\end{enumerate}
\end{lemma}

 \begin{theorem}
Assume $\Gamma$ is a definable class of forcings sastisfying clauses
\ref{def:wellbeh}(\ref{def:wellbeh1}), \ref{def:wellbeh}(\ref{def:wellbeh2}), \ref{def:wellbeh}(\ref{def:wellbeh3}),
\ref{def:wellbeh}(\ref{def:wellbeh5}), \ref{def:wellbeh}(\ref{def:wellbeh6}) of Def. \ref{def:wellbeh}.
Assume further that for all $\bool{B}\in\Gamma$ there is $i_{\bool{B}}:\bool{B}\to \bool{C}$ injective and
$\Gamma$-freezing 
$\bool{B}$.
Then the class of $\Gamma$-rigid partial orders is dense in $(\Gamma,\leq^*_\Gamma)$. 
\end{theorem}

\begin{proof}
%We first prove the density of $\Gamma$-rigid forcings in $(\Gamma,\leq_\Gamma)$.
Fix $\bool{B}_0\in \Gamma$ and $\Sigma_{\lambda_\Gamma}(\Gamma)$ be a winning strategy for 
player $II$ in $\mathcal{G}_{\lambda_\Gamma}(\Gamma)$.

Define 
\[
\FFF=\{k_{\alpha\beta}:\bool{B}_\alpha\to\bool{B}_\beta: \alpha\leq\beta<\lambda_\Gamma\}
%=\bp{(\bool{B}_\alpha,\dot{\bool{Q}}_\alpha):\alpha\leq\lambda_\Gamma}
\] 
by recursion on $\lambda_\Gamma$ as follows:
\begin{itemize}
%
%\item
%$\bool{B}_{\alpha+2n+1}=\bool{B}_{\alpha+2n}*\dot{\Coll}(\lambda_\Gamma,<\kappa_\alpha)$ 
%where $\kappa_{\alpha+2n}$ is regular larger than $|\bool{B}_\alpha|$ 
%(hence $\dot{\bool{Q}}_{\alpha+2n}=\dot{\Coll}(\omega_1,<\kappa_{\alpha+2n})$).
\item
at stage $0$, $II$ plays $k_0:\2\to\bool{B}_0$;
\item
at odd stages $\alpha$, $I$ plays
$k_{\alpha,\alpha+1}:\bool{B}_{\alpha}\to \bool{B}_{\alpha+1}$
$\Gamma$-correct and injective which freezes $\bool{B}_{\alpha}$;
%(hence -letting $\beta=\alpha+2n+1$-- $\dot{\bool{Q}}_{\beta}=\bool{B}_{\beta+1}/_{k_{\beta,\beta+1}[\dot{G}_{\bool{B}_\beta}]}$).
\item
at even stages $\alpha$, $II$ plays according to $\Sigma_{\lambda_\Gamma}$.
%$\bool{B}_\alpha$ is the inverse limit of $\bp{(\bool{B}_\beta,\dot{\bool{Q}}_\beta):\beta<\alpha}$.
\end{itemize}

By iterability of $\Gamma$, $\bool{B}_{\lambda_\Gamma}$ is the direct limit of $\FFF$ and belongs to $\Gamma$.
Clearly $k_{0\lambda_\Gamma}:\bool{B}_0\to\bool{B}_{\lambda_\Gamma}$ witnesses 
$\bool{B}_{\lambda_\Gamma}\leq^*_\Gamma \bool{B}_0$.
It suffices to prove the following:
\begin{claim} 
$\bool{B}_{\lambda_\Gamma}$ is $\Gamma$-rigid.
\end{claim}

\begin{proof} 
Assume $\bool{B}_{\lambda_\Gamma}\restriction f$ is compatible with $\bool{B}_{\lambda_\Gamma}\restriction g $ in 
$(\Gamma,\leq_\Gamma)$ for some threads $f,g$ incompatible in $\bool{B}_{\lambda_\Gamma}$.

 Let $\bool{R}\leq_\Gamma \bool{B}_{\lambda_\Gamma}\restriction f,\bool{B}_{\lambda_\Gamma}\restriction g$.

 Since $\bool{B}_{\lambda_\Gamma}$ is a direct limit, $f$, $g$ are threads with support bounded by some $\beta<\lambda_\Gamma$.
 Hence:
$f(\alpha)$ and $g(\alpha)$ are incompatible
\footnote{Here we use crucially use that $\bool{B}_{\lambda_\Gamma}$ is a direct limit! For inverse limits it is well possible that two incompatible threads $f$, $g$ are such that $f(\alpha)$ and $g(\alpha)$ are compatible in $\bool{B}_\alpha$ for all $\alpha$.}  in $\bool{B}_\alpha$ for all $\lambda_\Gamma\geq\alpha\geq\beta$.
 
Now $k_{\beta\beta+2}$ freezes $\bool{B}_\beta$, 
hence $\bool{B}_{\beta+2}\restriction k_{\beta\beta+2}(f(\beta))$
is incompatible with $\bool{B}_{\beta+2}\restriction k_{\beta\beta+2}(g(\beta))$.

But $k_{\beta\beta+2}(f(\beta))=f(\beta+2)$ and $k_{\beta\beta+2}(g(\beta))=g(\beta+2)$,
since both threads $f,g$ have support at most $\beta$. Hence 
\[
\bool{B}_{\beta+2}\restriction g(\beta+2)\bot_\Gamma\bool{B}_{\beta+2}\restriction f(\beta+2).
\]

We reached a contradiction:
\begin{itemize}
\item
On the one hand, for all $\alpha<\lambda_\Gamma$: 
 \[
 \bool{B}_\alpha \restriction f(\alpha)\geq_{\Gamma} \bool{B}_{\lambda_\Gamma}\restriction f\geq_\Gamma \bool{R}, 
  \]
  and
  \[
  \bool{B}_\alpha \restriction g(\alpha)\geq_{\Gamma} \bool{B}_{\lambda_\Gamma}\restriction g\geq_\Gamma \bool{R}.
\]

\item
On the other hand,
\[
\bool{B}_{\beta+2}\restriction g(\beta+2)\bot_\Gamma\bool{B}_{\beta+2}\restriction f(\beta+2).
\]
\end{itemize}
\end{proof}
The theorem is proved.
\end{proof}

Notice the following:

\begin{fact}\label{fac:keyrmkfreeze}
Assume $\Gamma\,\subseteq\,\Delta$ are definable classes of forcings. 
Then $\leq_\Gamma\subseteq\leq_\Delta$
and $\bot_\Delta\subseteq\bot_\Gamma$.
Hence, if $i:\bool{B}\to\bool{C}$ is $\Gamma$-correct and $\Delta$-freezes $\bool{B}$, we also have that
$i$ is $\Delta$-correct and $\Gamma$-freezes $\bool{B}$. 
\end{fact}
This fact will be repeatedly used to show that various classes of forcings $\Delta$ 
have the
$\Delta$-freezeability property providing for some $\Gamma\subseteq\Delta$ 
an $i:\bool{B}\to\bool{C}$ which is $\Gamma$-correct and $\Delta$-freezes $\bool{B}$. As we will see, all our freezeability results proceed  by proving the existence, given $\bool{B}\in\Gamma$, of a $\bool{B}$-name $\dot{\bool{Q}}$ for a forcing in $\Gamma$ such that $\bool{C}=\bool{B}\ast\dot{\bool{Q}}$ codes the generic filter $\dot{G_{\bool{B}}}$ for $\bool{B}$ as a subset $A_{\dot G_{\bool{B}}}$ of $\omega_1$ in some absolute manner, in the sense that in every outer model $M$ of $V^{\bool{C}}$ preserving stationary subsets of $\omega_1$, $A_{\dot G_{\bool{B}}}$ is the unique subset of $\pow{\omega_1}$ satisfying some given property.  It will thus follow that $\bool{C}$ $\SSP$-freezes $\bool{B}$, which will be an instance of the above since we will always have $\SSP\supseteq\Gamma$ for the $\Gamma$ of interest to us.

\subsection{Proof of the Factor Lemma for a well-behaved $\Gamma$} \label{subsec:prffactlem}

\begin{notation}
Given a well-behaved $\Gamma$,
for each $\bool{R}\in\bool{Rig}^\Gamma$ let
\[
k_{\bool{R}}:\bool{R}\to\Gamma\restriction\bool{R}
\]
be given by $r\mapsto\bool{R}\restriction r$.
Then $k_{\bool{R}}$ is an order and incompatibility preserving embedding of $\bool{R}$ 
in the class forcing
$\Gamma\restriction\bool{R}$ which maps maximal antichains to maximal antichains.
Moreover, for every $\bool{B}\geq_\Gamma\bool{C}$ with $\bool{B}\in \bool{Rig}^\Gamma$, let
\[
i_{\bool{B},\bool{C}}:\bool{B}\to\bool{C} 
\]
denote the unique $\Gamma$-correct homomorphism
from $\bool{B}$ into $\bool{C}$. 
\end{notation}

By the results of the previous sections, $\bool{Rig}^\Gamma$ is a dense subclass of $\Gamma$
and is a separative partial order. Hence, in order to simplify our calculations slightly, we focus on
$\bool{Rig}^\Gamma$ rather than on $\Gamma$ when analyzing this class forcing.

\begin{definition}\label{def:notcong}
Given $\bool{B}_0\in \Gamma$, fix $k_0:\bool{B}_0\to \bool{B}$
$\Gamma$-freezing $\bool{B}_0$ and such that $\bool{B}\in \bool{Rig}^\Gamma$.
Let $i_{\bool{C}}=i_{\bool{B},\bool{C}}\circ k_0$ and
\begin{align*}
k=k_{\bool{B}}\circ k_0:& \bool{B}_0\to \Gamma\restriction\bool{B}\\
& b\mapsto\bool{B}\restriction k_0(b)
\end{align*}
Given $G$ a $V$-generic filter for $\bool{B}_0$, define in $V[G]$ the class quotient forcing
\[
P_{\bool{B}_0}=((\bool{Rig}^\Gamma\restriction\bool{B})^V/_{k[G]},\leq_\Gamma/_{k[G]})
\] 
as follows:
\[
\bool{C}\in P_{\bool{B}_0}
\]
if and only if $\bool{C}\in(\bool{Rig}^\Gamma\restriction\bool{B})^V$ and,
letting $J$ be the dual ideal of $G$, we have that
$1_{\bool{C}}\not\in i_{\bool{C}}[J]$ (or, equivalently, if and only if $\coker(i_{\bool{C}})\in G$).

We let 
\[
\bool{C}\leq_{\Gamma/_{k[G]}}\bool{R}
\] 
if $\bool{C}\leq_\Gamma\bool{R}$ holds in $V$.
%$i\circ k_{\bool{C}}[]
%for all $\bool{S}_0\leq_\Gamma^V\bool{C}$ such that
%\[
%\bool{S}_0\in P_{\bool{B}}
%\]
%we have that some 
%\[
%\bool{S}_1\in P_{\bool{B}}
%\]
%is such that
%\[
%\bool{S}_1\leq_{P}^V\bool{R},\bool{S}_0.
%\]
\end{definition}
%We let $[\bool{C}]_{k_{\bool{B}}[G]}$ denote the equivalence class of $\bool{C}$ according to
%the partial order $\leq_{P}/_{k_{\bool{B}}[G]}$.
%In particular the poset given by the
%family of these equivalence classes with the induced quotient order 
%is the separative quotient
%of $(\Gamma\restriction\bool{B})^V/_{k_{\bool{B}}[G]}$.

\begin{theorem}\label{thm:quo-univ}
Suppose $\Gamma$ is absolutely well-behaved. 
Let $\bool{B}_0\in \Gamma$, and let $k_0:\bool{B}_0\to\bool{B}$ be a $\Gamma$-freezing
homomorphism for $\bool{B}_0$ with $\bool{B}\in\bool{Rig}^\Gamma$.
%Following the notation introduced above
Set $k=k_{\bool{B}}\circ k_0$ and $i_{\bool{C}}=i_{\bool{B},\bool{C}}\circ k_0$ for 
all $\bool{C}\leq_\Gamma \bool{B}$ in $\Gamma$.

\[
\begin{tikzpicture}[xscale=1.5,yscale=1.2]
\node (V) at (-0.5, 0.5) {$V$};
\node (B0) at (0, 0) {$\bool{B}_0$};
\node (B) at (1.5, 0) {$\bool{B}$};
\node (Gamma) at (3, 0) {$\Gamma \restriction B$};
\path (B0) edge [->]node [auto] {$\scriptstyle{k_0}$} (B);
\path (B) edge [->]node [auto] {$\scriptstyle{k_{\bool{B}}}$} (Gamma);
\path (B0) edge [->, bend left = 40]node [auto] {$\scriptstyle{k}$} (Gamma);
\end{tikzpicture}
\]

Let $G$ be $V$-generic for $\bool{B}_0$.
Then:
\begin{enumerate} 
\item \label{thm:equivGamma1}
The class forcing
\[
P_{\bool{B}_0}=((\bool{Rig}^\Gamma\restriction\bool{B})^V/_{k[G]},\leq_\Gamma/_{k[G]})
\] 
is in $V[G]$ forcing equivalent to 
the class forcing
\[
Q_{\bool{B}}=(\Gamma^{V[G]}\restriction(\bool{B}/_{k_0[G]}),\leq_{\Gamma^{V[G]}})
\]
via the map
\begin{align*}
i^*:&P_{\bool{B}} \to Q_{\bool{B}}\\
& \bool{C} \mapsto \bool{C}/_{i_{\bool{C}}[G]}.
\end{align*}

	\[
	\begin{tikzpicture}[xscale=1.5,yscale=1.2]
	\node (V[G]) at (-0.5, 0.5) {$V[G]$};
	\node (B0) at (0, 0) {$\bool{B}_0/_G$};
	\node (B) at (2, 0) {$\bool{B}/_{k[G]}$};
	\node (Gamma) at (4, 0) {$(\Gamma \restriction \bool{B})/_{k[G]}$};
	\node (2) at (0, -2) {$2$};
	\node (Gamma2) at (3, -2) {$\Gamma^{V[G]}\restriction(\bool{B}/_{k_0[G]})$};
	\node (x) at (4.5, -1) {$\bool{C}$};
	\node (y) at (3.5, -3) {$\bool{C} /_{i_\bool{C}[G]}$};
	\path (B0) edge [->]node [auto] {$\scriptstyle{k_0}$} (B);
	\path (B0) edge [->]node [auto] {$\cong$} (2);
	\path (B) edge [->]node [auto] {$\scriptstyle{k_{\bool{B}}}$} (Gamma);
	\path (B) edge [->]node [auto] {} (Gamma2);
	\path (Gamma) edge [->]node [auto,swap] {$\cong$} (Gamma2);
	\path (x) edge [|->]node [auto] {} (y);
	\end{tikzpicture}
	\]
\item \label{thm:equivGamma2}
Moreover %assume $T$ is strongly $\Gamma$-canonical.
let $\delta>|\bool{B}|$ be inaccessible and such that
$(V_\delta,\in)\prec (V,\in)$.
Then:
\begin{enumerate} 
\item \label{thm:equivGamma2.1}
$\Gamma^{V[G]}_\delta\restriction(\bool{B}/_{k_0[G]})$ is forcing equivalent in $V[G]$ 
to  $(\Gamma_\delta\restriction\bool{B})^V/_{k[G]}$ via the same map.
\item \label{thm:equivGamma2.2}
$V$ models that $k_\bool{B}:\bool{B}\to \Gamma_\delta\restriction\bool{B}$ is
$\Gamma$-correct.
\end{enumerate}
\end{enumerate}
\end{theorem}

Notice that this theorem proves the missing part of Theorem \ref{thm:mainth1}(\ref{thm:mainth1-1}), i.e. 
the assertion:
\begin{quote}
\emph{For all $\bool{B}\in\Gamma_\delta$ there is
$\bool{C}\in\Gamma_\delta$ such that $\Gamma_\delta\restriction\bool{C}\leq_\Gamma\bool{B}$}.
\end{quote}

In particular this theorem and Theorem \ref{thm:univGamma} give a completely self-contained and detailed proof of
Theorem \ref{thm:mainth1}(\ref{thm:mainth1-1}) and Theorem \ref{thm:mainth1}(\ref{thm:mainth1-2}).

%We leave to the reader to fill in the details for the proofs of Thm. \ref{thm:mainth1}(\ref{thm:mainth1-3}) and
%Thm. \ref{thm:mainth1}(\ref{thm:mainth1-4}).

\begin{proof} 
Part %\ref{thm:equivGamma2}.
\ref{thm:equivGamma2.1}
of the theorem follows immediately from its part \ref{thm:equivGamma1}
relativizing every assumption in part \ref{thm:equivGamma1} to $V_{\delta+1}$.
To prove part~\ref{thm:equivGamma2.2}, first observe that if $\bool{B}=\bool{B}_0$,
$k_0$ is necessarily the identity map, $G$ is $V$-generic for $\bool{B}$, and this  gives that
$(\bool{B}/_{k_0[G]})$ is the trivial complete boolean algebra $2=\bp{0,1}$, i.e.: 
%and~\ref{thm:equivGamma2.2}
%states that 
\[
\Gamma_\delta^{V[G]}\restriction(\bool{B}/_{k_0[G]})=\Gamma_\delta^{V[G]}.
\]
%whenever
%$\bool{Q}\leq_\Gamma\bool{B}$ is $\Gamma$-rigid in $V$ and $G$ is $V$-generic for 
%$\bool{B}$, we have that
%$\bool{Q}/_{i_{\bool{B},\bool{Q}}[G]}$ is $\Gamma^{V[G]}$-rigid in $V[G]$, by Fact~\ref{fac:torigquot}.
%%This gives that also in $V[G]$ we have that $(\bool{U}^\Gamma_\delta)^{V[G]}$ has a dense set of 
%%$\Gamma$-rigid elements. 
%We conclude that
Now 
\begin{itemize}
\item 
$\delta$ is inaccessible in $V[G]$;
\item
$V_\delta\prec V$ grants that $V_\delta[G]\prec V[G]$, since $G\in V_\delta[G]$;
\end{itemize}
hence the set of $\Gamma^{V[G]}$-rigid forcings is dense in 
$\Gamma^{V[G]}\cap V_\delta[G]$, since it is dense in $\Gamma^{V[G]}$, being $\Gamma^{V[G]}$ well-behaved in
$V[G]$.
By Theorem~\ref{thm:univGamma} applied in $V[G]$, $(\Gamma_\delta)^{V[G]}\in \Gamma^{V[G]}$.

By part~\ref{thm:equivGamma2.1} (applied in $V_{\delta+1}[G]$) for the case $\bool{B}_0=\bool{B}$ (so that $k=k_{\bool{B}}$),
we get that
$(\Gamma_\delta\restriction\bool{B})/_{k_{\bool{B}}[G]}\cong 
(\Gamma_\delta)^{V[G]}$ holds in $V[G]$
for all $G$ $V$-generic for $\bool{B}$. This concludes the proof of~\ref{thm:equivGamma2.2}
in case $\bool{B}=\bool{B}_0$. The desired conclusion~\ref{thm:equivGamma2.2}
for an arbitrary $\bool{B}_0\in\Gamma_\delta$ follows 
using the fact that the set of $\bool{B}\leq_\Gamma\bool{B}_0$ in $\bool{Rig}^\Gamma$ is dense in $\Gamma_\delta$
and applying~\ref{thm:equivGamma2.2} to all such $\bool{B}$.

We are left with proving part~\ref{thm:equivGamma1}:
%there is only one correct homomorphism
%\[
%i_{\bool{R}}:\bool{B}\to\bool{R}
%\]
Following the notation introduced in \ref{def:notcong}, 
we let $i_{\bool{R}}$ denote
the $\Gamma$-correct homomorphism $i_{\bool{B},\bool{R}}\circ k_0$
for any $\bool{R}\leq_\Gamma\bool{B}$, and we let
$k$ denote the map $k_{\bool{B}}\circ k_0:\bool{B}_0\to\Gamma\restriction\bool{B}$
given by $b\mapsto \bool{B}\restriction k_0(b)$.

Let $G$ be $V$-generic for $\bool{B}_0$ and $J$ denote its dual prime ideal.
We first observe that in $V[G]$,
\[
\downarrow k[J]=\{\bool{R}\in \Gamma^V: \exists q\in J\, \bool{R}\leq^V_\Gamma 
\bool{B}\restriction k_0(q)\}.
\]

We show that in $V[G]$ the map $i^*$ is total,
order and incompatibility preserving, and with a dense target.
%\[
%k^*=(i^*):\Gamma^{V[G]}\to (\Gamma^V\restriction\bool{B})/_{k[G]}
%\]
%given by pairs $(\bool{R}/_{i_{\bool{R}}[G]},\bool{R})$ with $\bool{R}\in \Gamma^V/_{k[G]}$
%is surjective on its range and domain and is such that:
%\begin{itemize}
%\item
%$\bool{R}/_{i_{\bool{R}}[G]}$ and $\bool{R}/_{i_{\bool{R}}[G]}$ are compatible in 
%$\Gamma^{V[G]}$ iff thery are compatible
%order and incompatibility preserving. 
This suffices to prove this part of the theorem.
%To simplify matters further we also study the map $k^*$ only on $\bool{Rig}^\Gamma^{V[G]}$.

\begin{description}
\item[\textbf{$i^*$ is total and with a dense target}]
%To this aim first observe that in $V[G]$
%\[
%\downarrow k[J]=\{\bool{R}\in (\bool{U}^{\Gamma})^V: \exists q\in J\, \bool{R}\leq^V_\Gamma 
%\bool{B}\restriction b\}.
%\]
%
By Theorem~\ref{qIso}, any $\bool{Q}\in Q_{\bool{B}}$ is isomorphic to
$\bool{C}/_{i_{\bool{C}}[G]}$ for some $\bool{C}\in (\Gamma\restriction\bool{B})^V$
such that $1_{\bool{C}}\notin \downarrow i_{\bool{C}}[J]$, since $\bool{Q}$ is a non-trivial
complete boolean algebra in $V[G]$. 
Let in $V$ $\bool{R}\in\bool{Rig}^\Gamma$ refine $\bool{C}$ in the $\leq^*_\Gamma$-order.

We claim that $\bool{R}/_{i_{\bool{R}[G]}}$ refines $\bool{Q}$ in $Q_{\bool{B}}$.

%First of all we show that $\bool{R}/_{i_{\bool{R}[G]}}$ is a non-trivial condition.
Assume towards a contradiction that $\bool{R}/_{i_{\bool{R}[G]}}\not\in Q_{\bool{B}}$. Then we would get that
$1_{\bool{R}}\in i_{\bool{R}}[J]$. Therefore for any $\Gamma$-correct \emph{injective} $u:\bool{C}\to\bool{R}$
witnessing that $\bool{R}\leq^*_\Gamma\bool{C}$,
we would have that $i_{\bool{C}}[J]=u^{-1}[i_{\bool{R}}[J]]$. This gives that $1_{\bool{C}}\in i_{\bool{C}}[J]$, and
contradicts our assumption that $1_{\bool{C}}\notin \downarrow i_{\bool{C}}[J]$.

Therefore $1_{\bool{R}}\notin i_{\bool{R}}[J]$, and 
\[
u/_J:\bool{C}/_{i_{\bool{C}}[G]}\to \bool{R}/_{i_{\bool{R}[G]}}
\]
witnesses that $i^*(\bool{R})$ refines $\bool{Q}$ in $Q_{\bool{B}}$. Hence $i^*$ has a dense image.

%Thus 
%$\bool{C}\in P_{\bool{B}_0}$ is such that $i^*(\bool{C})\cong \bool{Q}$ in $V[G]$.
%This shows that $i^*$ is surjective modulo isomorphism. 
Moreover for any $\bool{R}\in P_{\bool{B}_0}$, $1_{\bool{R}}\not\in i_{\bool{R}}[J]$, hence
$\bool{R}/_{i_{\bool{R}}[G]}$ is a non-trivial
complete boolean algebra in $\Gamma^{V[G]}$.
Thus $i^*$ is also well defined on all of $(\bool{Rig}^\Gamma\restriction\bool{B})^V/_{k[G]}$.
%Thus $i^*
%Assume $\bool{R}$ is a non-trivial element in the quotient forcing
%$\Gamma^V\restriction\bool{B}/_{k[G]}$.
%Then $\bool{R}$ is positive with respect to the filter (generated by) $k[G]$, thus
% $0_{\bool{R}}\not\in i_{\bool{R}}[G]$: else $1_{\bool{R}}\in i_{\bool{R}}[J]$, which gives that
%$\bool{R}\leq_\Gamma\bool{B}\restriction q$ for some $q\in J$. In particular
%$\bool{R}$ would be in $\downarrow k[J]$, contradicting our assumptions.
%%To this aim notice first of all that in $V[G]$ for all $q\in J$, the quotient boolean algebra
%%\[
%%(\RR\restriction i_{\RR}(q))/_{i_{\RR}[G]}
%%\] 
%%is trivial.
%Since 
% $0_{\bool{R}}\not\in i_{\bool{R}}[G]$, we can conclude that $\bool{R}/_{i_{\bool{R}}[J]}$ 
% is a non trivial complete boolean algebra
%in $\Gamma^{V[G]}$ and that 
%the pair $(\bool{R}/_{i_{\bool{R}}[G]},\bool{R})\in k^*$.

%by~\cite[Lemma 4.2, Theorem 4.10]{VIAAUDSTE13}.

%We are left to
%show that $i^*$ is order and incompatibility preserving.
%If this is the case $i^*$ is an embedding with a dense image which witnesses that
%the two forcing notions are equivalent.
%\begin{description}
\item[\textbf{$i^*$ is order and compatibility preserving}]
Let $i_{\bool{Q}_0\bool{Q}}:\bool{Q}_0\to \bool{Q}$
%\bool{R}_1/i_{\bool{R}_1}[G]$ is a 
be a $\Gamma$-correct complete
homomorphism in $V$ with $\bool{Q}_0,\bool{Q}\in P_{\bool{B}_0}$
witnessing that $\bool{Q}\leq_\Gamma/_{k[G]} \bool{Q}_0$. This occurs only if
$1_{\bool{Q}}\notin i_{\bool{Q}}[J]$.
By Lemma \ref{EmbTwoStepPI1pers},
$i_{\bool{Q}_0\bool{Q}}/_J:\bool{Q}_0/_{i_{\bool{Q}_0}[J]}\to\bool{Q}/_{i_{\bool{Q}}[J]}$ is $\Gamma^{V[G]}$-correct
and
witnesses that $\bool{Q_0}/_{i_{\bool{Q}_0}[J]}\geq_\Gamma \bool{Q}/_{i_{\bool{Q}}[J]}$ holds  in $V[G]$.
%Observe that if $j:\bool{Q}_0/_{i_{\bool{Q}_0}[G]}\cong Q_0\to Q$
%%\bool{R}_1/i_{\bool{R}_1}[G]$ is a 
%is a $\Gamma$-correct complete
%homomorphism in $V[G]$, we have (by Proposition~\ref{lem:PI1pers-2}(\ref{lem:firstfctlem-2-2}))
%%applied to 
%%$\bool{B},\bool{R}_0,Q$ in the place of $\bool{B},\bool{Q}_0,\bool{Q}$) 
%that $j=l_0/_G$ for some
% $\Gamma$-correct homomorphism $l_0:\bool{Q}_0\to \bool{C}$ in $V$ such that
% $l_0\circ i_{\bool{Q}_0}=i_\bool{C}$ 
% with range $\bool{C}$
% $Q\cong \bool{C}/_{k_{\bool{C}[G]}}$ in $V[G]$,
%$l$ in $V$, 
% and $0_{\bool{C}}\not \in i_{\bool{C}}[G]$
% (given that $i_{\bool{C}}$ is the unique eventually possible
% $\Gamma$-correct homomorphism $l:\bool{B}\to\bool{C} $ in $V$, since $\bool{B}$ is $\Gamma$-rigid in $V$).
%% We can also check that $j$ is $\Gamma$-correct in $V[G]$ iff $l$ is 
%% $\Gamma$-correct in $V$.
% In particular $l_0$ witnesses in $V$ that $\bool{Q}_0\geq^V_P\bool{C}$ and
% the fact that $0_{\bool{C}}\not \in i_{\bool{C}}[G]$ grants that 
% $\bool{C}$ is a non trivial condition in 
% $(\Gamma^V\restriction\bool{B})/_{k[G]}$ refining $\bool{Q}_0$.
 This shows that $i^*$ is order preserving and maps non-trivial conditions
 to non-trivial conditions. In particular we can also conclude that $i^*$
 maps compatible conditions to compatible conditions.
 
 \item[\textbf{$i^*$ preserves the incompatibility relation}]
 We prove this by contraposition. 
Assume $j_h:\bool{Q}_h/_{i_{\bool{Q}_h}[G]}\cong \bool{R}_h\to \bool{Q}$ for $h=0,1$
witness that $\bool{Q}_0/_{i_{\bool{Q}_0}[G]}$ and 
$\bool{Q}_1/_{i_{\bool{Q}_1}[G]}$ are compatible in $(\Gamma)^{V[G]}$.
We can assume that 
$\bool{Q}\cong \bool{C}/_{i_{\bool{C}}[G]}$.

%%\bool{R}_1/i_{\bool{R}_1}[G]$ is a 
%is a $\Gamma$-correct complete
%homomorphism in $V[G]$, 
By Proposition~\ref{lem:PI1pers-2}
applied for both $h=0,1$ to
$\bool{B},i_{\bool{Q_h}},j_h$
we have 
that $j_h=l_h/_G$ for some
$\Gamma$-correct homomorphism $l_h:\bool{Q}_h\to \bool{C}_h$ in $V$ such that:
\begin{itemize}
\item
$l_h\circ i_{\bool{Q}_h}=i_{\bool{C}_h}$ for both $h=0,1$,
\item
$\bool{C}_1/_{i_{\bool{C}_1[G]}}\cong \bool{Q}\cong \bool{C}_0/_{i_{\bool{C}_0[G]}}$ in $V[G]$,
\item
$0_{\bool{C}_h}\not \in i_{\bool{C}_h}[G]$ for both $h=0,1$.
\end{itemize}
By Proposition~\ref{prop:embfromembnames2}, we can find $s_j\notin i_{\bool{C}_j}[J]$ such that
$\bool{C}_1\restriction s_1$ and
$\bool{C}_0\restriction s_0$ are isomorphic.
Without loss of generality we can suppose that $\bool{C}_h\restriction s_h=\bool{C}\in\Gamma$.
This gives that (modulo  the refinement via $s_h$) $l_h\circ i_{\bool{Q}_h}=i_{\bool{C}}$ for both $h=0,1$, since 
both $l_h\circ i_{\bool{Q}_h}$ factor through $k_0$ which is $\Gamma$-freezing $\bool{B}_0$.

%% We can also check that $j$ is $\Gamma$-correct in $V[G]$ iff $l$ is 
%% $\Gamma$-correct in $V$.
 In particular each $l_h$ witnesses in $V$ that $\bool{Q}_h\geq_\Gamma\bool{C}$ and
 are both such that $1_{\bool{C}}\not \in i_{\bool{C}}[J]$.
 Find in $V$ $\bool{R}\leq^*_\Gamma\bool{C}$ with $\bool{R}\in\bool{Rig}^\Gamma$.
 Then $i_{\bool{R}}[J]=u\circ i_{\bool{C}}[J]$  for some (any) $\Gamma$-correct \emph{injective} $u:\bool{C}\to \bool{R}$.
 Hence   $1_{\bool{R}}\not \in i_{\bool{R}}[J]$, else $1_{\bool{C}}\in u^{-1}[i_{\bool{R}}[J]]=i_{\bool{C}}[J]$.
 
 This grants that 
 $\bool{R}$ is a non-trivial condition in 
 $(\bool{Rig}^\Gamma\restriction\bool{B})^V/_{k[G]}$ refining $\bool{Q}_h$ for both $h=0,1$.
 \end{description}

The proof of the theorem is completed.
\end{proof}

We also obtain the following completeness result as a corollary of Theorems \ref{thm:mainth1} and \ref{thm:quo-univ}. 

\begin{corollary}\label{completeness}
Assume 
$\ap{V,\mathcal{V}}$ satisfies 
\begin{itemize}
\item $\MK$, 
\item $\lambda_\Gamma\text{ is a regular uncountable cardinal}$, and 
\item $\Gamma\text{ is absolutely well-behaved}$.
\end{itemize}

Suppose $\bool{B}$, $\bool{C}\in\Gamma$ are such that $V^{\bool{B}}\models\BCFA(\Gamma)$ and $V^{\bool{C}}\models\BCFA(\Gamma)$. If $G$ and $H$ are generic filters over $V$ for, respectively, $\bool{B}$ and $\bool{C}$, then $H_{\lambda_\Gamma^+}^{V[G]}$ and $H_{\lambda_\Gamma^+}^{V[H]}$ have the same theory. 
\end{corollary}

\section{Absolutely well-behaved classes}\label{suitable classes}

We organize this part of the paper as follows:
\begin{itemize}
\item We start giving the necessary definitions in \ref{subsec:defdavid}.
\item We state our main results in \ref{david:mainresults}. Specifically we assert that there are uncountably many 
%$\omega_1$-suitable 
absolutely well-behaved definable classes of forcing notions with $\lambda_\Gamma=\omega_1$, 
%for the theory $\MK$ +`\emph{$\omega_1$ is the first uncountable cardinal}',  
whose bounded category forcing axioms yield pairwise incompatible theories for $H_{\omega_2}$ (this is incompatibility in first order logic).
\item In~\ref{david:proofs} we give the proofs, specifically:
\begin{itemize} 
\item
In \ref{4-freezing-posets}
we isolate four types of freezing posets which will be used to establish the freezeability property.
%for all these classes. 
\item
In \ref{iterability} we present the iteration lemmas 
that will be used to establish the 
%weak
 iterability property (all of which were already known). 
\item 
In \ref{omega1suitability} we give the proof that there are $\aleph_1$-many definable classes $\Gamma$ of forcing notions which are absolutely well-behaved with $\lambda_\Gamma=\omega_1$.
%$\omega_1$-suitable for the appropriate theories 
%$T\supseteq \MK$  + `\emph{$\omega_1$ is the first uncountable cardinal}'.
\item In \ref{incompatible-category-forcing-axioms} we prove that bounded category forcing axioms for the uncountably many absolutely well-behaved classes we produced in \ref{omega1suitability} yield pairwise 
incompatible theories for $H_{\omega_2}$.
\end{itemize}
\end{itemize}

\subsection{Forcing classes.} \label{subsec:defdavid}

%In this section our background theory $T$ is $\MK$ + `\emph{$\omega_1$ is the first uncountable cardinal}'.

We will now define the main classes of forcing notions considered in this paper. Most of these classes are 
completely standard, but we nevertheless include their definition here for the benefit of some readers. 
%From now on when we state that a certain property holds for a poset $P$ 
%we automatically infer that the property holds for its boolean completion $\RO(P)$.
%Conversely when we assert that a complete boolean algebra $\bool{B}$ has a property defined for posets.
%it means the the poset $\bool{B}^+=\bool{B}\setminus\bp{0_{\bool{B}}}$ with the order inherited from 
%$\bool{B}$ has this property.

\begin{definition} 
A poset has the \emph{countable chain condition} (is \emph{c.c.c.}, for short) if and only if it has no uncountable antichains. 
\end{definition}

Given an ordinal $\rho$, we will call a sequence $(X_i)_{i\leq\rho}$ a \emph{continuous chain} (or a \emph{continuous $\rho$-chain}, if we want to bring in the length) if 
\begin{itemize}
\item $(X_k)_{k\leq i}\in X_{i+1}$ whenever $i+1\leq \rho$ and 
\item $X_i=\bigcup_{k<i}X_k$ for every nonzero limit ordinal $i \leq\rho$.
\end{itemize}

An ordinal $\rho$ is said to be \emph{indecomposable} if $\rho=\o^\tau$ for some ordinal $\tau$.\footnote{Here, and elsewhere in the remainder of the paper, $\o^\tau$ denotes ordinal exponentiation.} 
Equivalently, $\rho$ is indecomposable if $\ot(\rho\setminus\eta)=\rho$ for every $\eta<\rho$.  $1$ is of course the 
first indecomposable ordinal. 

\begin{definition}
Given a countable indecomposable ordinal $\rho$, a poset $\mtcl P$ is \emph{$\rho$-proper} if and only if there is a cardinal 
$\t$ such that $\mtcl P\in H_\t$ and there is a club 
$D\sub [H_\t]^{\al_0}$ with the property that for every continuous chain $(N_i)_{i\leq \rho}$ of countable 
elementary submodels of $H_\t$ containing $\mtcl P$ and every $p\in N_0\cap \mtcl P$, there is an extension 
$q$ of $p$ such that $q$ is \emph{$(N_i, \mtcl P)$-generic} for all $i\leq\rho$, i.e., for every $i\leq\rho$ and every 
dense subset $D$ of $\mtcl P$, $D\in N_i$, $q\Vdash_{\mtcl P}D\cap \dot G\cap N_i\neq\emptyset$. 
\end{definition}

\begin{remark} $\mtcl P$ is $\rho$-proper if and only if for every cardinal $\t$ such that $\mtcl P\in H_\t$ there is such a club $D\sub [H_\t]^{\al_0}$ as in the above definition.\end{remark}

 $\rho$-$\PR$ denotes the class of $\rho$-proper posets. 
We write ${<}\o_1$-$\PR$ to denote the class of those posets that are in $\rho$-$\PR$ for every indecomposable 
$\rho<\o_1$. We say that $\mtcl P$ is \emph{proper} if it is $1$-proper, and denote $1\mbox{-}\PR$ also by $\PR$. 

The following is a simple but crucial observation:
\begin{fact}\label{fac:sigma2proper}
For any countable indecomposable ordinal $\rho$, 
%the theory 
%$\MK$+ `\emph{$\omega_1$ is the first uncountable cardinal}' + `\emph{$\rho$ is a countable indecomposable ordinal'}
%\begin{quote}
%$\MK+$\emph{($\omega_1$ is the first uncountable cardinal)}$+$\emph{($\rho$ is a countable indecomposable ordinal)}
%\end{quote}
% proves that
`\emph{$\RO(P)$ is $\rho$-proper}' is both a
%n absolutely 
$\Sigma_2$ property in parameters $\rho$ and $\omega_1$ and  a 
$\Pi_2$ property
in the same parameters, and the same 
%can be proved 
holds for `\emph{$\RO(P)$ is ${<}\omega_1$-proper}'. 
%with respect to
%$\MK$+`\emph{$\omega_1$ is the first uncountable cardinal}'.
\end{fact}
\begin{proof}
Let $\t$ be large enough such that  $\RO(P)\in H_\lambda$ for some $\lambda<\t$.
Then `\emph{$\RO(P)$ is $\rho$-proper}' holds in $V$ if and only if it holds in any (some) transitive set $X\supseteq H_\t$.
Hence,  $\RO(P)$ is $\rho$-proper if and only if there is some regular cardinal $\theta$ and some transitive $X\supseteq H_\t$ such that 
$(X,\in)\models\mbox{`$\RO(P)$ is $\rho$-proper'}$.

%\[
%\exists\t [\t\text{\emph{ is a regular cardinal}}\wedge \exists X (X\supseteq H_\t\wedge X\text{\emph{ is transitive }}
%\wedge(X,\in)\models\text{\emph{$\RO(P)$ is $\rho$--proper}})].
%\]
Now:
\begin{itemize}
\item
The formulae $(X,\in)\models\text{\emph{$\RO(P)$ is $\rho$-proper}}$ and \emph{$X$ is transitive} are 
$\Delta_1$ in the parameters $X$, $P$, $\rho$, $\omega_1$.
\item 
The formula $\t\text{\emph{ is a regular cardinal}}$ is $\Pi_1$ in parameter $\t$.
\item
The formula $X\supseteq H_\t$ is $\Pi_1$ in parameters $X,\t$ since it can be stated as 
\[
\forall w(|\text{trcl}(w)|<\t\rightarrow w\in X), 
\]
where $\text{trcl}(w)$ is the $\Delta_1$-definable operation assigining to the set $w$ its transitive closure.
\end{itemize}
It is now easy to check that \emph{$\RO(P)$ is $\rho$-proper} is 
%an absolutely 
a $\Sigma_2$ property %and a $\Pi_2$ property 
in parameters 
$\rho$ and $\omega_1$. We leave it to the reader to check that it is also $\Pi_2$ in the same parameters.
\end{proof}
Being $\rho$-proper, for a forcing $\mtcl P$, is equivalent to $\mtcl P$ preserving a certain combinatorial property: 
Given a set $X$, we say that $S\sub\, ^\rho([X]^{\al_0})$ is \emph{$\rho$-stationary} if for every club $D\sub[X]^{\al_0}$ 
there is a continuous $\rho$-chain $\s$ of members of $D$ such that $\s\in S$. 
 
Recalling the standard characterization of properness, the 
following is not difficult to see.
% over the theory 
%$\MK$ + `\emph{$\omega_1$ is the first uncountable cardinal}'+`\emph{$\rho$ is a countable indecomposable ordinal}'.
 
 \begin{fact}\label{char-proper} Given an indecomposable ordinal $\rho<\o_1$,  the following are equivalent for every 
 poset $\mtcl P$.
 
 \begin{enumerate}
 
 \item $\mtcl P$ is $\rho$-proper.
 \item For every set $X$,  $\mtcl P$ preserves $\rho$-stationary subsets of $^\rho([X]^{\al_0})$; i.e., if 
 $S\sub\,^\rho([X]^{\al_0})$ is $\rho$-stationary, then $\Vdash_{\mtcl P}S\mbox{ is $\rho$-stationary}$. 
\end{enumerate}
\end{fact}

Using the above fact we can prove:
\begin{fact}\label{fac:Birkhoffproper}
For every countable indecomposable ordinals $\rho$, 
%the theory $\MK$+`\emph{$\omega_1$ is the first uncountable cardinal}'+`\emph{$\rho$ is a countable indecomposable ordinal}'
%%\begin{quote}
%%$\MK+$\emph{($\omega_1$ is the first uncountable cardinal)}$+$\emph{($\rho$ is a countable indecomposable ordinal)}
%%\end{quote}
%proves that 
$\rho$-$\PR$ and ${<}\o_1$-$\PR$ are 
closed under isomorphisms, two-step iterations, lottery sums, restrictions and complete subalgebras,
%under preimages by complete injective homomorphisms, two-step iterations and products,
and contain
all countably closed forcings.
\end{fact}

\begin{definition}
A forcing notion $\mtcl P$ is \emph{$\rho$-semiproper} iff there is a cardinal $\t$ such that 
$\mtcl P\in H_\t$ for which there is a club $D\sub [H_\t]^{\al_0}$ 
with the property that for every continuous chain $(N_i)_{i\leq \rho}$ of countable elementary submodels of $H_\t$ 
containing $\mtcl P$ and every $p\in N_0\cap \mtcl P$, there is an extension $q$ of $p$ such that 
$q$ is \emph{$(N_i, \mtcl P)$-semi-generic} for all $i\leq\rho$. This means now that for every $i\leq\rho$ and every 
$\mtcl P$-name $\dot \a\in N_i$ for an ordinal in $\o_1^V$, $q\Vdash_{\mtcl P}\dot \a\in N_i$.
\end{definition}

\begin{remark} $\mtcl P$ is $\rho$-semiproper if and only if for every cardinal $\t$ such that $\mtcl P\in H_\t$ there is a club $D\sub [H_\t]^{\al_0}$ as in the above definition. \end{remark}

$\rho$-$\SP$ denotes the class of $\rho$-semiproper posets. Also, we write ${<}\o_1$-$\SP$ to denote 
 the class of those posets that are in $\rho$-$\SP$ for every indecomposable ordinal $\rho<\o_1$. We say that 
 $\mtcl P$ is \emph{semiproper} if it is $1$-semiproper, and denote $1\mbox{-}\SP$ also by $\SP$. 
 
 As before we have:
 \begin{fact}\label{fac:sigma2semiproper}
%$\MK$ + `\emph{$\omega_1$ is the first uncountable cardinal}'+`\emph{$\rho$ is a countable indecomposable ordinal}'
%proves that 
`\emph{$\RO(P)$ is $\rho$-semiproper}' is both
%an absolutely 
a $\Sigma_2$ property in parameters $\rho$ and $\omega_1$ and a 
$\Pi_2$ property
in the same parameters. The same holds for `\emph{$\RO(P)$ is ${<}\omega_1$-semiproper}'.
\end{fact}
 
 Let $X$ be a set such that $\o_1\sub X$. We say that $S\sub\,^\rho([X]^{\al_0})$ is \emph{$\rho$-semi-stationary} if for every club $D\sub [X]^{\al_0}$ there  are continuous $\rho$-chains $\s=(x_i\,:\,i\leq\rho)$ and $\s'=(x_i'\,:\,i\leq\rho)$ such that 
 
 \begin{itemize}
 
 \item $\s\in S$,  
 \item $\range(\s')\sub D$, and
 \item for each $i\leq\rho$, $x_i\sub x_i'$ and $x_i\cap\o_1=x_i'\cap\o_1$.
 \end{itemize}
 
 We have the following characterization of $\rho$-semiproperness (for any given indecomposable $\rho<\o_1$).
 
 \begin{fact}\label{char-semiproper} Given an indecomposable ordinal $\rho<\o_1$,  the following are equivalent for every poset $\mtcl P$.
 
 \begin{enumerate}
 
 \item $\mtcl P$ is $\rho$-semiproper.
 \item  $\mtcl P$ preserves $\rho$-semi-stationary subsets of $^\rho([X]^{\al_0})$ for every set $X$; i.e., if $S\sub\,^\rho([X]^{\al_0})$ is $\rho$-semi-stationary, then $S$ remains $\rho$-semi-stationary after forcing with $\mtcl P$. 
\end{enumerate}
\end{fact}
Again we get that:
\begin{fact}\label{fac:Birkhoffsemiproper}
For all countable indecomposable ordinals $\rho$,
%$\MK$+`\emph{$\omega_1$ is the first uncountable cardinal}'+ `\emph{$\rho$ is a countable indecomposable ordinal}'
%proves that 
$\rho$-$\SP$ and ${<}\o_1$-$\SP$ are closed under isomorphisms, two-step iterations, lottery sums, restrictions and complete subalgebras,
%closed under preimages by complete injective homomorphisms, two-step iterations and products, 
and contain
all countably closed forcings.
\end{fact}

\begin{definition}
Given a regular cardinal $\k\geq\o_1$, a poset $\mtcl P$ \emph{preserves stationary subsets of $\k$} if 
every stationary subset of $\k$ remains stationary after forcing with $\mtcl P$.  
\end{definition}
$\SSP$ denotes the class of 
partial orders preserving stationary subsets of $\o_1$. More generally, given a  cardinal $\l$, $\SSP(\l)$ denotes
the class of partial orders preserving stationary subsets of $\k$ for every uncountable regular cardinal $\k\leq\l$.
%\begin{fact}\label{fac:sigma2SSP}
%\emph{$\RO(P)$ is $\SSP(\l)$--proper} is an absolutely $\Sigma_2$-property in parameter $\l$ and is also a 
%$\Pi_2$-property
%in the same parameter. %The same can be proved for \emph{$\RO(P)$ is $<\omega_1$--proper}
%\end{fact} 

Recall that a Suslin tree  is an $\o_1$-tree $T$ (i.e., $T$ is a tree of height $\o_1$ all of whose levels are countable) without uncountable chains or antichains (a subset of $T$ is called an \emph{antichain} iff it consists of pairwise incomparable nodes). We will consider the above properties in conjunction with the preservation of some combination of the two following properties.

 \begin{definition}
 A poset $\mtcl P$ \emph{preserves Suslin trees} if $\Vdash_{\mtcl P}T\mbox{ is Suslin}$ for every Suslin tree $T$ in the ground model. %We use $\STP$ to denote the class of all posets preserving Suslin trees.
 \end{definition}
  
 \begin{definition}
 A poset $\mtcl P$ is \emph{$^\o\o$-bounding} iff every function $f:\o\into\o$ added by $\mtcl P$ is bounded by a function 
 $g:\o\into\o$ in the ground model; i.e., iff  for every $\mtcl P$-generic filter $G$ and every $f:\o\into \o$, $f\in V[G]$, there is some $g:\o\into g$, $g\in V$, such that $f(n)<g(n)$ for all $n$. 
 \end{definition}
$\STP$ denotes the class of all posets preserving Suslin trees and
$^\o\o\mbox{-bounding}$ the class of $^\o\o$-bounding posets. 

By the same arguments we gave for $\rho$-properness one gets:
\begin{fact}\label{fac:sigma2SPTomombound}
%$\MK$+ `\emph{$\omega_1$ is the first uncountable cardinal}'
%proves that 
`\emph{$\RO(P)$ preserves Suslin trees}' and `\emph{$\mtcl P$ is $^\o\o$-bounding}'
are 
%absolutely 
$\Sigma_2$ properties in parameters $\omega_1$ and $\omega^\omega$, and 
also $\Pi_2$ properties
in the same parameters. Moreover, 
%this theory  
%$\MK$+ `\emph{$\omega_1$ is the first uncountable cardinal}'
%proves that 
$\STP$  and
$^\o\o\mbox{-bounding}$ are both closed under isomorphisms, two-step iterations, lottery sums, restrictions and complete subalgebras,
%closed under preimages by complete injective homomorphisms, two-step iterations and products,
and contain all countably closed forcings.
%The same can be proved for \emph{$\RO(P)$ is $<\omega_1$--proper}
\end{fact}

In \cite[XI]{SHEPRO} Shelah isolates a property he calls $S$-condition, and for 
which he proves the following.
\begin{lemma}\label{S-cond-shelah}
Assume $\mtcl P$ is a forcing notion satisfying the $S$-condition. Then:
\begin{enumerate}
\it $\mtcl P$ preserves  stationary subsets of $\o_1$;\footnote{See \cite[XI--Thm. 3.6]{SHEPRO}. This theorem says that forcing notions with the $S$-condition do not collapse $\o_1$.  However, its proof actually establishes that such forcings in fact preserve stationary subsets of $\o_1$.}
\it if $\CH$ holds, then $\mtcl P$ adds no new reals.
\end{enumerate}
\end{lemma}

As shown in  \cite[XI--4]{SHEPRO}, among the forcing notions satisfying the $S$-condition are Namba forcing (and natural variations thereof), all  countably closed forcing notions, and the natural poset which, for a fixed stationary $S\sub \{\a<\o_2\,:\,\cf(\a)=\o\}$, adds an $\o_1$-club through $S$ with countable conditions. 

Given a tree $T$ and a node $\eta$ of $T$, let $\textsf{succ}_T(\eta)$ denote the set of immediate successors of $\eta$ in $T$. It will be convenient to define the following game $\mtcl P^{\mtcl P}_p$ (for a partial order $\mtcl P$ and a $\mtcl P$-condition $p$).

\begin{definition}
Given a partial order $\mtcl P$ such that $\av\mtcl P\av\geq\al_2$, $\mtcl G^{\mtcl P}$ is 
the following game of length $\o$ between players I and II, with player I playing at even stages and player II playing at 
odd stages.

\begin{enumerate}

\it At any given stage $n$ of the game, the corresponding player picks a pair $T^n$, $(p^n_\eta)_{\eta\in T^n}$, 
where $T^n$ is a tree consisting of finite sequences of ordinals in $\av\mtcl P\av$ without infinite branches and 
where $(p^n_\eta)_{\eta\in T^n}$ is a sequence of conditions in $\mtcl P$ extending $p$ such that $p^n_\n$ 
extends $p^n_\eta$ in $\mtcl P$ whenever $\n$ extends $\eta$ in $T^n$.  

\it If $n>0$, then 

\begin{enumerate}
 \it $T^n$ and $(p^n_\eta)_{\eta\in T^n}$ end-extend $T^{n-1}$ and $(p^{n-1}_\eta)_{\eta\in T^{n-1}}$, respectively, 
\it every terminal node in $T^{n-1}$ has a proper extension in $T^n$, and
\it every node in $T^n\setminus T^{n-1}$ extends a unique terminal node in $T^{n-1}$.
\end{enumerate} 

\it Player I starts by playing $T_0=\{\emptyset\}$ and $p^0_\emptyset \in \mtcl P$.

\it At any given even stage $n>0$ of the game, player I picks, for every terminal node $\eta$ of $T^{n-1}$, a finite sequence $\n_\eta$ of ordinals in $\av\mtcl P\av$ such that $\n_\eta$ extends $\eta$ properly.  He then builds $T^n$ as $$T^{n-1}\cup\{\n_\eta\restr k\,:\,k\leq \av\n_\eta\av,\,\eta\mbox{ a terminal node of }T^{n-1}\}.$$ Player I also has to choose of course  $(p^n_\eta)_{\eta\in T^n}$ in such a way that (1) and (2) are satisfied. 

\it At any given odd stage $n$ of the game, player II chooses, for every terminal node $\eta$ of $T^{n-1}$, a regular cardinal $\k^n_\eta\in [\al_2,\,\av\mtcl P\av]$, and builds $T^n$ from $T^{n-1}$ by adding to $T^{n-1}$ a next level where, for each terminal node $\eta$ of $T^{n-1}$, the set of immediate successors of $\eta$ in $T^n$ is $\{\eta^\smallfrown \la\a\ra\,:\,\a<\k^n_\eta\}$. Player II also has to choose of course  $(p^n_\eta)_{\eta\in T^n}$ in such a way that (1) and (2) are satisfied.

\end{enumerate}

After $\o$ moves, the players have naturally built a tree $T=\bigcup_n T^n$ of height $\o$ whose nodes are 
finite sequences of ordinals in $\av\mtcl P\av$, together with a sequence 
$(p_\eta)_{\eta\in T}=\bigcup (p^n_\eta)_{\eta\in T^n}$ of $\mtcl P$-conditions such that for all nodes $\eta$, 
$\nu$ in $T$, if $\nu$ extends $\eta$ in $T$, then $p_{\nu}$ extends $p_{\eta}$ in $\mtcl P$. Finally, player II wins the 
game iff for every subtree $T'$ of $T$ in $V$, 
if $\av\textsf{succ}_{T'}(\eta)\av=\av\textsf{succ}_T(\eta)\av$ for every $\eta\in T'$, 
then there is a condition in $\mtcl P$ forcing that there is an $\o$-branch $b$ through $T'$ such that 
$p_{b\restr n}\in\dot G$ for all $n<\o$. 
 
\end{definition}

The definition of the $S$-condition is the following.\footnote{Shelah's definition is more general, but 
the present form suffices for our purposes.}

\begin{definition}\label{s-cond-def}
A partial order $\mtcl P$ satisfies the $S$-condition if and only if $\av\mtcl P\av\geq\al_2$ and %for every $p\in\mtcl P$ 
player II has a winning strategy $\s$ in the game $\mtcl G^{\mtcl P}$%_p$ 
such that for every partial run of the game, 
the output of $\s$ at any given sequence $\eta\in\,^{{<}\o}\av\mtcl P\av$ depends only on $\eta$, 
$(p_{\eta\restr k})_{k\leq \av\eta\av}$ and $\{k<\av\eta\av\,:\,\av \textsf{succ}_T(\eta\restr k)\av>1\}$, 
where $T$ denotes the tree built by the players up to that point.
\end{definition}

$S\textsf{-cond}$ is the class of complete boolean algebras $\bool{B}$ satisfying the $S$-condition.

One has:
\begin{fact}\label{fac:sigma2Scond}
%$\MK$+`\emph{$\omega_1$ is the first uncountable cardinal}'
%proves that 
`\emph{$\RO(P)$ satisfies the $S$-condition}' is a
%n absolutely 
$\Sigma_2$ property in parameter $\omega_2$ 
and also a 
$\Pi_2$ property
in the same parameter. 

%$\MK$+ `\emph{$\omega_1$ is the first uncountable cardinal}'
%proves that 
$S\textsf{-cond}$ is closed under isomorphisms, two-step iterations, lottery sums, restrictions and complete subalgebras,
%closed under preimages by complete injective homomorphisms, two-step iterations and products, 
and contains
all countably closed forcings.
%The same can be proved for \emph{$\ROP}$ is $<\omega_1$--proper}
\end{fact}
\begin{proof}
As in the case of all other classes dealt with in this section, if $\RO(P)\in H_\theta$, then 
$H_\theta\models\mbox{`$\RO(P)$ satisfies the $S$-condition}$' if and only if 
$\RO(P)$ satisfies the $S$-condition.

%\[
%\text{$H_\theta\models$\emph{ $\RO(P)$ satisfies the $S$--condition} if and only if 
%$V\models$\emph{ $\RO(P)$ satisfies the $S$--condition}.}
%\]
%\begin{aMnote}{Problem!}
%Assume $P$ is a complete suborder of a poset with the
%$S$-condition, and $P$ forces that $\dot{Q}$ is a complete suborder of a poset with the
%$S$-condition. How do we prove that $P\ast\dot{Q}$ is a complete suborder of a poset with the
%$S$-condition?
The remaining properties of $S\textsf{-cond}$ other than the closure under complete subalgebras
are left to the reader.

We prove now that $S\textsf{-cond}$ is a class closed under complete subalgebras;
it will be convenient for this to resort to the algebraic properties of complete injective homomorphisms
with adjoints outlined in Theorem~\ref{eRetrProp} of the appendix.

Assume $\bool{B}$ is a complete subalgebra of some $\bool{C}$ satisfying the $S$-condition.
Let $\pi:\bool{C}\to\bool{B}$ be the adjoint map of the inclusion map of $\bool{B}$ into $\bool{C}$.
Let $\sigma$ be the winning strategy for player II in $\mtcl G^{\bool{C}^+}$. % for any $c\in\bool{C}$.
Define %for any $b\in\bool{B}$ 
$\sigma'$ to be the strategy for player II in 
 $\mtcl G^{\bool{B}^+}$ obtained by the following procedure:
 \begin{itemize}
% \item At stage $1$ player $II$ takes 
% \[
% \sigma(\emptyset)=\ap{T_0,\bp{c_\eta:\eta\in T_0}}
% \] 
% and 
% let 
% \[
% \sigma'(\emptyset)=\ap{T_0,\bp{\pi(c_\eta):\eta\in T_0}}.
% \]
 \item
 Player I and II build 
 partial plays $\ap{T_{n},\bp{b_\eta:\eta\in T_n}}$  in 
 $\mtcl G^{\bool{B}^+}$ and partial plays $\ap{T_n,\bp{c_\eta:\eta\in T_n}}$ in 
 $\mtcl G^{\bool{C}^+}$ according to these prescriptions:
 \begin{itemize}
 \item
 for all $\eta\in T_n$ and all $n$
 we have that $\pi(c_\eta)=b_\eta$;
 \item
 for all terminal nodes $\eta\in T_{2n}$ 
 with $\eta=\ap{\gamma_0,\dots,\gamma_m}$, we have that 
 \[
 c_\eta=b_\eta\wedge c_{\ap{\gamma_0,\dots,\gamma_{m-1}}};
 \]
  \item 
  $\ap{T_{2n+1},\bp{c_\eta:\eta\in T_{2n+1}}}=\sigma(\ap{T_{2n},\bp{c_\eta:\eta\in T_{2n}}})$;
 \end{itemize}
 \item Player II defines $\sigma'(\ap{T_{2n},\bp{b_\eta:\eta\in T_{2n}}})=\ap{T_{2n+1},\bp{b_\eta:\eta\in T_{2n+1}}}$.
 \end{itemize}
 
 Now assume $\ap{T,\bp{b_\eta:\eta\in T}}$ is built according to a play of $\mtcl G^{\bool{B}^+}$ in which II follows 
 $\sigma'$.
 Fix a subtree $T'\subseteq T$ as given by the winning condition for II in
 $\mtcl G^{\bool{B}^+}$.
 Given  a $V$-generic filter $H$ for $\bool{B}$, we must find some infinite branch $\eta$ of $T'$ in $V[H]$ such that
 $b_{\eta\restriction n}\in H$ for all $n$.
 To find this branch let 
 $\ap{T,\bp{c_\eta:\eta\in T}}$ be the tree built in tandem with $\ap{T,\bp{b_\eta:\eta\in T}}$
 according to the rules we used to define $\sigma'$. 
 Fix $G\supseteq H$ $V$-generic for $\bool{C}$.
 Since $\bool{C}$ satisfies the $S$-condition, we can find some infinite 
 branch $\eta$ of $T'$ in $V[G]$ such that $c_{\eta\restriction n}\in G$ for all $n$.
 This gives that $b_{\eta\restriction n}=\pi(c_{\eta\restriction n})\in H$ for all $n$.
 Hence, in $V[G]$ there is an infinite branch $\eta$ of $T'$ such that
 $b_{\eta\restriction n}\in H$ for all $n$. Therefore the tree $T^*=\bp{\eta\in T': b_\eta\in H}\in V[H]$ 
 is ill-founded in $V[G]$.
But then the same is true in $V[H]$ by absoluteness of ill-foundedness. Finally, any infinite branch of $T^*$ witnesses the winning condition for II 
 using
 $\sigma'$ relative to $T,T',H$.

 Since this can be done for all possible choices of $T\supseteq T'$ in $V$ with $T$ constructed using $\sigma'$, and 
 for all  $V$-generic filters $H$ for $\bool{B}$, we have that 
 $\bool{B}$ satisfies the $S$-condition.
\end{proof}

%The rest of this part is structured as follows. 
%In the next section, we define all relevant notions leading to the definition of $\l$--$\CFA(\Gamma)$, for a 
%class $\Gamma$ of complete Boolean algebras, and present, without proofs, the general theory of category 
%forcing axioms. We refer the reader to the forthcoming~\cite{VIAAUDSTEBOOK} for the proofs.  
%In Section \ref{4-freezing-posets} we give four instances of $\SSP$--freezing posets. In 
%Section \ref{incompatible-category-forcing-axioms} we state our main results,  viz.\ Theorems \ref{main-thm-suitable} 
%and  \ref{main-thm-incompatible}. Theorem \ref{main-thm-suitable} is proved in Subsection \ref{Proving Theorem 
%main-thm-suitable}, using the results from Section \ref{4-freezing-posets}, and 
%Theorem \ref{main-thm-incompatible} is proved in Subsection \ref{incompatibility}. 

\subsection{Incompatible bounded category forcing axioms}\label{david:mainresults}

In this section we  isolate $\al_1$-many classes $\Gamma$ of forcing notions with $\lambda_\Gamma=\omega_1$, all of which are absolutely well-behaved (in some cases modulo the existence of unboundedly many measurable cardinals), and which are pairwise incompatible.
%, provably in $\ZFC$ + $\LC$, where $\LC$ denotes `\emph{There is a proper class of supercompact cardinals}' + `\emph{There is a $2$-superhuge cardinal}'.\footnote{Sometimes, a weaker theory than this suffices for our incompatibility result.} 
Our main results are the following.

\begin{theorem}\label{main-thm-suitable} 
\begin{enumerate}
\item Each of the following classes $\Gamma$ is absolutely well-behaved and such that $\lambda_\Gamma=\omega_1$. 
%$\o_1$-suitable with respect
%to $\MK$+`\emph{$\omega_1$ is the first uncountable cardinal}'.

\begin{enumerate}
\item $\PR$
\item $\PR\cap \STP$
\item $\PR\cap\,^\o\o\textsf{-bounding}$
\item  $\PR\cap\STP\cap\,^\o\o\textsf{-bounding}$
%\end{enumerate}
\item $\rho\mbox{-}\PR$ for every countable indecomposable ordinal $\rho$ such that $1<\rho<\o_1$.
%is $\o_1$-suitable for the theory
%$\MK$+`\emph{$\omega_1$ is the first uncountable cardinal}' +
%`\emph{$\rho$ is a countable indecomposable ordinal}'
%for every countable indecomposable ordinal $\rho$ such that $1<\rho<\o_1$.

\item ${<}\o_1\mbox{-}\PR$ 
%is $\o_1$-suitable for 
%$\MK$+`\emph{$\omega_1$ is the first uncountable cardinal}'. 

\item  $S\textsf{-cond}$
% is $\o_1$-suitable for 
%$\MK$+`\emph{$\omega_1$ is the first uncountable cardinal}'.
\end{enumerate}

\item %Suppose there  is a proper class of measurable cardinals. If  $\rho<\omega_1$ is an indecomposable ordinal, then 
Suppose there  is a proper class of measurable cardinals. Then each of the following classes $\Gamma$ is absolutely well-behaved and such that $\lambda_\Gamma=\omega_1$. 
%Each of the following classes is $\o_1$-suitable for 
%$\MK$+`\emph{$\omega_1$ is the first uncountable cardinal}'+`\emph{$\rho$ is a countable indecomposable ordinal}'
%+ `\emph{there is a proper class of measurable cardinals}' for any countable indecomposable ordinal $\rho<\omega_1$:

\begin{enumerate}
\item $\rho\mbox{-}\SP$ for every countable indecomposable ordinal $\rho<\omega_1$.
\item $(\rho\mbox{-}\SP)\cap \STP$ for every countable indecomposable ordinal $\rho<\omega_1$.
\item $(\rho\mbox{-}\SP)\cap\,^\o\o\textsf{-bounding}$ for every countable indecomposable ordinal $\rho<\omega_1$.
\item $(\rho\mbox{-}\SP)\cap\STP\cap\,^\o\o\textsf{-bounding}$ for every countable indecomposable ordinal $\rho<\omega_1$.
%\end{enumerate}
%\item 
%Each of the following classes is $\o_1$-suitable for 
%$\MK$+ `\emph{$\omega_1$ is the first uncountable cardinal}'+ `\emph{there is a proper class of measurable cardinals}':
%%Suppose there  is a proper class of measurable cardinals. Then each of the following classes is $\o_1$--suitable.
%\begin{enumerate}
\item ${<}\o_1\mbox{-}\SP$ 
\item $({<}\o_1\mbox{-}\SP)\cap \STP$
\item $({<}\o_1\mbox{-}\SP)\cap\,^\o\o\textsf{-bounding}$
\item $({<}\o_1\mbox{-}\SP)\cap\STP\cap\,^\o\o\textsf{-bounding}$ 
\end{enumerate}

\end{enumerate}
\end{theorem}

%In Theorem \ref{main-thm-incompatible}, and elsewhere, we denote by $\LC$ the following second order assertion: There is a supercompact cardinal $\delta$ such that $V_\delta\prec V$.   

\begin{theorem}\label{main-thm-incompatible} 
Suppose there is a supercompact cardinal $\delta$ such that $V_\delta\prec V$.
Suppose $\Gamma$ and $\Gamma'$ are any two different classes of forcing notions mentioned in Theorem 
\ref{main-thm-suitable}. Then $\BCFA(\Gamma)$ implies $\lnot\BCFA(\Gamma')$.
\end{theorem}
 
 \begin{remark} For some choices of $\Gamma$ and $\Gamma'$, the incompatibility of $\BCFA(\Gamma)$ and $\BCFA(\Gamma')$ can be proved just assuming the existence of an inaccessible cardinal $\delta$ such that $V_\delta\prec V$, or just the existence of an inaccessible $\delta$ such that $V_\delta\prec V$ together with the existence of a proper class of Woodin cardinals. \end{remark}

As will be clear from the proofs, Theorems \ref{main-thm-suitable} and \ref{main-thm-incompatible} are just 
selected samples of a zoo of possibly incompatible instances of $\BCFA(\Gamma)$.
 In particular, it should be possible to combine (some of) the classes mentioned in Theorem  \ref{main-thm-suitable} with other classes of forcing notions, besides $\STP$ and $^\o\o\mbox{-bounding}$, so long as these classes have a suitable iteration theory and reasonable closure properties, are both 
 %absolutely 
 $\Sigma_2$ definable and 
 %absolutely
  $\Pi_2$ definable, and the resulting classes contain $\SSP$-freezing posets.

%The next section contains the lemmas which, together, prove Theorem \ref{main-thm-suitable}.
%Section \ref{incompatibility} contains the lemmas that yield the proof of Theorem \ref{main-thm-incompatible}.

We should point out that the following natural question---in the present context---remains open.

\begin{question} Is there, under any reasonable large cardinal, any indecomposable $\rho<\o_1$, $\rho>1$, for which any of the following classes is absolutely well-behaved?

\begin{enumerate}
\item $(\rho\mbox{-}\PR)\cap \STP$
\item $(\rho\mbox{-}\PR)\cap\,^\o\o\textsf{-bounding}$
\item $(\rho\mbox{-}\PR)\cap\STP\cap\,^\o\o\textsf{-bounding}$ 
\item $({<}\o_1\mbox{-}\PR)\cap \STP$
\item $({<}\o_1\mbox{-}\PR)\cap\,^\o\o\textsf{-bounding}$
\item $({<}\o_1\mbox{-}\PR)\cap\STP\cap\,^\o\o\textsf{-bounding}$ 
\end{enumerate}
\end{question}

The following question, of a more foundational import, addresses the possibility of there being absolutely well-behaved classes $\Gamma$ such that $\lambda_\Gamma>\omega_1$.

\begin{question} Are there, under some reasonable large cardinal assumption, any cardinal $\l\geq\o_2$ and any class 
$\Gamma$ of forcing notions such that $\Gamma$ is absolutely well-behaved and such that $\lambda_\Gamma=\lambda$? Are there, again  under some reasonable large cardinal 
assumption, any cardinal $\l\geq\o_2$ and any class $\Gamma$ of forcing notions with $\lambda_\Gamma=\l$ and such that $\Gamma$ is absolutely well-behaved and 
such that $\BCFA(\Gamma)$ is compatible with---or, even, implies---$\BCFA(\Gamma')$ for any 
absolutely well-behaved class $\Gamma'$ with $\lambda_{\Gamma'}=\omega_1$?
\end{question}

\subsection{Proof of theorems ~\ref{main-thm-suitable} and ~\ref{main-thm-incompatible}.}\label{david:proofs}

\subsubsection{Four freezing posets}\label{4-freezing-posets}

In this section we introduce four instances of $\SSP$-freezing posets. We feel free to confuse posets with complete
boolean algebras, as the context will dictate which is the correct intended meaning of the concept.
%These posets will be crucially used in 
%Section \ref{incompatible-category-forcing-axioms} when proving that the classes of forcing notions considered there are 
%$\omega_1$--suitable.
When proving $\SSP$-freezability, we will actually be showing the following sufficient condition (for $\lambda=\omega_1$).

\begin{lemma}\label{freezability-sufficient} Let $\lambda\geq\omega_1$ be a cardinal, $\bool{B}$ a forcing notion, and 
$\dot{\bool{C}}$ a $\bool{B}$-name for a forcing notion. 
Suppose that $p$ is a set, and that if $G$ is a $\bool{B}$-generic filter, then $\bool{C}=\dot{\bool{C}}_G$ forces that there is some $A_G\subseteq\lambda$ coding $G$ in an absolute way mod.\ $\SSP(\lambda)$, in the sense that there is some $\Sigma_1$ formula $\varphi(x, y, z)$ such that, if $H$ is a $\bool{B}\ast\dot{\bool{C}}$-generic filter over $V$ such that $H\cap \bool{B}=G$, then

\begin{enumerate}

\item $(H_{\lambda^+}; \in, \NS_\lambda)^{V[H]}\models\varphi(G, A_G, p)$, and

\item  in every outer model $M$ of $V[H]$ such that $\mtcl P(\lambda)^{V[H]}\cap (\NS_\lambda)^M=(\NS_\lambda)^{V[H]}$, if $$(H_{\lambda^+}; \in, \NS_\lambda)^{M}\models\varphi(G_0, A_{G_0}, p)$$ and $$(H_{\lambda^+}; \in, \NS_\lambda)^{M}\models\varphi(G_1, A_{G_1}, p),$$ then $G_0=G_1$.\end{enumerate} 

Then the natural inclusion $$i:\bool{B}\into \bool{B}\ast\dot{\bool{C}}$$ $\SSP(\lambda)$-freezes $\bool{B}$.

\end{lemma}

\begin{proof}
Suppose, towards a contradiction, that $b_0$, $b_1\in \bool{B}$ are incompatible, $\bool{D}$ is a complete boolean algebra, $k_0:(\bool{B}\ast\dot{\bool{C}})\restr b_0\into\bool{D}$, $k_1:(\bool{B}\ast\dot{\bool{C}})\restr b_1\into \bool{D}$ are complete homomorphisms, $K$ is $\bool{D}$-generic and, for each $\epsilon\in \{0, 1\}$, $H_\epsilon = k_\epsilon^{-1}(K)$ and every stationary subset of $\lambda$ in $V[H_\epsilon]$ remains stationary in $V[K]$. For each $\epsilon\in\{0, 1\}$, let $G_\epsilon$ be the filter on $\bool{B}$ generated by $H_\epsilon\cap (\bool{B}\restr b_\epsilon)$, and let $A_\epsilon\sub \lambda$ be such that $$(H_{\lambda^+}; \in, \NS_\lambda)^{V[H_\epsilon]}\models\varphi(G_\epsilon, A_\epsilon, p)$$ Since $$(H_{\lambda^+}; \in, \NS_\k)^{V[K]}\models\varphi(G_0, A_0, p)\wedge\varphi(G_1, A_1, p),$$ we have that $G_0=G_1$ by (2). But this is impossible since $b_0\in G_0$, $b_1\in G_1$, and since $b_0$ and $b_1$ are incompatible conditions in $\bool{B}$.  
\end{proof}

Our first freezing poset comes essentially from \cite{MOO06}.

Given a set $X$, the Ellentuck topology on $[X]^{\aleph_0}$ is the topology on $[X]^{\aleph_0}$ generated by the sets of the form $[s, Y]$, for $Y\in [X]^{\al_0}$ and $s\in [Y]^{{<}\omega}$, where $[s, Y]=\{Z\in [Y]^{\aleph_0}\,:\, s \subseteq Z\}$.

The following lemma, except for the conclusion that $\mtcl P$ preserves Suslin trees, is due to Moore \cite{MOO06}. The conclusion that $\mtcl P$ preserves Suslin trees is due to Miyamoto  \cite{Miyamoto-Yorioka}. 

\begin{lemma}\label{MRP}
Let  $X$ be a set, $\t$ a cardinal such that $X\in H_{\t}$,  and $\S$ a function with domain $[H_\t]^{\al_0}$ such that for every countable $M\elsub  [H_\t]^{\al_0}$, 

\begin{itemize} 
 
 \item $\S(M)\sub [X]^{\al_0}$ is open in the Ellentuck topology, and 

\item $\S(M)$ is $M$-stationary (meaning that for every function $F:[X]^{{<}\o}\into X$, if $F\in M$, then there is some $Z\in \S(M)\cap M$ such that $F``[Z]^{{<}\o}\sub Z$).\end{itemize}

Let $\mtcl P=\mtcl P_{X, \t, \S}$ be the set, ordered by reverse inclusion, of all countable  $\sub$-continuous $\in$-chains  $p=(M^p_i)_{i\leq\n}$ of countable elementary substructures of $H_\t$ such that for every limit ordinal $i\leq\n$ there is some $i_0<i$ with the property that $M^p_{k}\cap X\in \S(M^p_i)$ for all $k$ such that $i_0<k<i$. 

Then \begin{enumerate}

\item $\mtcl P$ is proper, preserves Suslin trees, and does not add new reals.

\item Whenever $G$ is $\mtcl P$-generic over $V$ and $M^G_i= M^p_i$ for $p\in G$ and $i\in \dom(p)$, 
$(M^G_i)_{i<\o_1}$ is in $V[G]$ the $\sub$-increasing enumeration of a club of $[H_\t^{V}]^{\al_0}$ 
and is such that:
\begin{quote} 
For every limit ordinal 
$i<\o_1$ there is some $i_0<i$ with the property that $M^G_{k}\cap X\in \S(M^G_i)$ for all $k$ such that $i_0<k<i$.   
\end{quote}
\end{enumerate}
\end{lemma}

\begin{remark} In most interesting cases, the forcing $\mtcl P_{X, \t, \S}$ in the above lemma is 
not $\o$-proper. \end{remark}

In \cite{MOO06}, Moore defines the Mapping Reflection Principle ($\MRP$) as the following statement: Given $X$, 
$\t$, and $\S$ as in the hypothesis of Lemma \ref{MRP}, there is a $\sub$-continuous $\in$-chain $(M_i)_{i<\o_1}$ 
of countable elementary substructures of $H_\t$ such that for every limit ordinal $i<\o_1$ there is some $i_0<i$ with the property 
that $M_{k}\cap X\in \S(M_i)$ for every $k$ such that $i_0<k<i$. 

It follows from Lemma \ref{MRP} that $\MRP$ is a consequence of $\textsc{PFA}$, and of the 
forcing axiom for the class of  forcing notions in $\PR\cap \STP$ not adding new reals. 

We will call a partial order $\mtcl R$ an \emph{$\MRP$-poset} if there are $X$, $\t$ and $\S$ as in the hypothesis of 
Lemma \ref{MRP} such that $\mtcl R=\mtcl P_{X, \t, \S}$.

\begin{proposition}\label{freeze1}
Given a forcing notion $\mtcl P$,  there is $\mtcl P$-name $\dot{\mtcl Q}$ for  a forcing notion such that 
\begin{enumerate} 
\it $\dot{\mtcl Q}$ is forced to be of the form $\Coll(\o_1, \mtcl P)\ast\dot{\mtcl R}$, where $\dot{\mtcl R}$ is a $\Coll(\o_1, \mtcl P)$-name for an $\MRP$-poset, and 
\it $\mtcl P\ast\dot{\mtcl Q}$ $\SSP$-freezes $\mtcl P$, as witnessed by the inclusion map.
  \end{enumerate}
\end{proposition}

\begin{proof}
By Lemma \ref{freezability-sufficient}, it suffices to prove that $\mtcl P$ forces that in $V^{\Coll(\o_1, \mtcl P)}$ there is an  $\MRP$-poset $\dot{\mtcl R}$ such that $\Coll(\o_1, \mtcl P)\ast\dot{\mtcl R}$ codes the generic filter for $\mtcl P$ in an absolute way mod.\ $\SSP$ in the sense of that lemma. For this, let us work in $V^{\mtcl P\ast \Coll(\o_1,\,\mtcl P)}$. Let $\dot B_G$ be a subset of $\o_1$ coding the generic filter for $\mtcl P$ in some canonical way, let $\vec C=(C_\delta \,:\, \delta \in \Lim(\omega_1))\in V$ be a ladder system on $\omega_1$ (i.e., every $C_\d$ is a cofinal subset of $\d$ of order type $\o$), and let $(S_\alpha)_{\alpha<\omega_1}\in V$ be a partition of $\omega_1$ into stationary sets. Given $X\sub Y$, countable sets of ordinals, such that  $Y\cap \omega_1$ and $\ot(Y)$ are both limit ordinals and such that $X$ is bounded in $\mbox{sup}(Y)$, let $c(X, Y)$ mean  $$\av C_{\ot(Y)} \cap \mbox{sup}(\pi_Y``X)\av < \av C_{Y\cap \omega_1}\cap X\cap \omega_1\av,$$ where $\pi_Y$ is the collapsing function of $Y$.

  Let  $\t$ be a large enough cardinal and let $\S$ be the function sending a countable $N\elsub H_\t$ to the set of  $Z\in [\o_2\cap N]^{\al_0}$ such that $c(X, \o_2\cap N)$ iff the unique $\a<\o_1$ such that $N\cap \o_1\in S_\a$ is in $B_{\dot G}$. Now,  $X=\o_2$, $\t$ and $\S$ satisfy the hypothesis of Lemma \ref{MRP} (s.\ \cite{MOO06}).\footnote{$\S$ is in essence the mapping used by Moore in \cite{MOO06} to prove that $\BPFA$ implies $2^{\al_1}=\al_2$.} Let $\mtcl R=\mtcl P_{\o_2, \t, \S}$. By Lemma \ref{MRP}, $\mtcl R$ adds  $(Z^{\dot G}_i)_{i<\o_1}$, a strictly $\sub$-increasing enumeration of a club of $[\o_2^V]^{\al_0}$, such that for every limit ordinal $i<\o_1$, if $Z^{\dot G}_i\cap\o_1\in S_\a$, then  there is a tail of $k< i$ such that  $c(Z^{\dot G}_{k}, Z^{\dot G}_i)$ if and only if $\a\in B_{\dot G}$.   
  
 Let $\k=\o_2$, let $H$ be $(\mtcl P\ast \Coll(\o_1, \mtcl P))\ast \dot{\mtcl R}$-generic, let $G=H\cap\mtcl P$, and let $A_G$ be a subset of $\omega_1$ which canonically codes $B_G$ and $(Z^G_i)_{i<\omega_1}$. If $M$ is any outer model such that every stationary subset of $\o_1$ in $V[H]$ remains stationary in $M$, then $B_G$ is the unique subset $B$ of $\o_1$ for which there is,  in $M$, a set $A\sub\o_1$ coding $B$ together with an $\sub$-increasing enumeration $(Z_i)_{i<\o_1}$ of  a club of $[\k]^{\al_0}$ with the property that  for every limit ordinal $i<\o_1$, if $Z_i\cap\o_1\in S_\a$, then  there is a tail of  $k< i$ such that  $c(Z_{k}, Z_i)$ if and only if $\a\in B$. Indeed, If $B'\in M$ were another such set, as witnessed by $A'\sub M$, $\a\in B\Delta B'$, and $(Z'_i)_{i<\o_1}\in M$ were an $\sub$-increasing enumeration of a club of $[\k]^{\al_0}$ with the property that for every limit ordinal $i<\o_1$, if $Z'_i\cap\o_1\in S_\a$, then there is a tail of  $k< i$ such that  $c(Z_{k}, Z_i)$ if and only if $\a\in B'$, then we would be able to find some $i$ such that $Z_i=Z'_i$, $Z_i\cap \o_1\in S_\a$, and such that $Z_{k}=Z'_k$ for  all $k$ in some cofinal subset $J$ of $i$. But then we would have that  $c(Z_{k}, Z_i)$, for all $k$ in some final segment of $J$, both holds and fails. 
 
 Finally, it is immediate to see that there is a $\Sigma_1$ formula $\varphi(x, y, z)$ such that $\varphi(G, A_G, p)$ expresses the above property of $G$ and $A_G$ over $(H_{\omega_2}; \in)^{M}$ for any $M$ as above, for $p=(\kappa, \vec C, (S_\a)_{\a<\o_1})$. 
\end{proof}

Using coding techniques from \cite{CAIVEL06}, one can prove the following stronger version of Lemma \ref{freeze1}. However, we do not have any use for this stronger form, so we will not give the proof here. 

\begin{lemma} 
Given a  partial order $\mtcl P$  there is $\mtcl P$-name $\dot{\mtcl Q}$ for  a  partial order with the following properties.
\begin{enumerate} 
\it $\dot{\mtcl Q}$ is forced to be of the form $\Coll(\o_1, \mtcl P)\ast\dot{\mtcl R}$ where $\dot{\mtcl R}$ is a 
$\Coll(\o_1, \mtcl P)$-name for a forcing of the form $\dot{\mtcl R}_0\ast\dot{\mtcl R}_1$, 
where $\dot{\mtcl R}_0$ has the countable chain condition and $\dot{\mtcl R_1}$ is forced to be an $\MRP$-poset. 
\it Suppose $b_0$, $b_1\in \RO(\mtcl P)$ are incompatible, $\bool{B}$ is a complete boolean algebra, and 
$k_\epsilon:\RO(\mtcl P\ast\dot{\mtcl Q})\restr b_\epsilon\into\bool{B}$ are complete homomorphisms for 
$\epsilon\in\{0, 1\}$. Then $\bool{B}$ collapses $\o_1$.
  \end{enumerate}

\end{lemma}

Our second freezing poset comes from \cite[Section 1]{Velickovic}, where the following is proved, 
using a result of Todor\v{c}evi\'{c} from \cite{Todor}.

\begin{lemma}\label{lemma-velickovic}
There is a sequence $((K^\x_0, K^\x_1)\,:\,\x<\o_1)$ of colourings of $[\k]^2$, for $\kappa=\cf(2^{\al_0})$, 
with the property that in any $\o_1$-preserving outer model in which $\av\k\av=\al_1$, if $B\sub\o_1$, then 
there is a c.c.c.\ partial order $\mtcl R$ forcing the existence of $\al_1$-many decompositions 
$\k=\bigcup_{n<\o}X^\x_n$, for $\x<\o_1$, such that
for all $\x <\o_1$:
\begin{itemize} 
\item
for some fixed $i_\x=0,1$, $X^\x_n$ is $K^\x_{i_\x}$-homogeneous for all $n<\o$;
\item
$\x\in B$ if and only if $i_\x=0$.
\end{itemize}
\end{lemma}

\begin{proposition}\label{freeze1.5}
Given a forcing notion $\mtcl P$,  there is $\mtcl P$-name $\dot{\mtcl Q}$ for  a forcing notion with the following properties.
\begin{enumerate} 
\it Letting $\mu = \av\mtcl P\av+\cf(2^{\al_0})$, $\dot{\mtcl Q}$ is forced to be a forcing of the form 
$\Coll(\o_1,\, \mu)\ast\dot{\mtcl R}$, where $\dot{\mtcl R}$ is a $\Coll(\o_1,\, \mu)$-name for a c.c.c.\  forcing.
\it $\mtcl P\ast\dot{\mtcl Q}$ $\SSP$-freezes $\mtcl P$, as witnessed by the inclusion map. In fact, if 
$b_0$, $b_1\in \RO(\mtcl P)$ are incompatible, $\bool{B}$ is a complete boolean algebra, and 
$k_\epsilon:\RO(\mtcl P\ast\dot{\mtcl Q})\restr b_\epsilon\into\bool{B}$ is a complete homomorphism for 
$\epsilon\in\{0, 1\}$, then $\bool{B}$ collapses $\o_1$.
  \end{enumerate}
\end{proposition}

\begin{proof}
 Let us work in $V^{\mtcl P\ast \Coll(\o_1,\,\mu)}$. Let $\dot B_G$ be a subset of $\o_1$ coding the generic filter for 
 $\mtcl P$ in some canonical way, let $\vec K=((K^\x_0, K^\x_1)\,:\,\x<\o_1)$ be a sequence of colourings of $[\k]^2$, for 
 $\kappa=\cf^V(2^{\al_0})$, as given by Lemma \ref{lemma-velickovic}, and let $\mtcl R$  be a c.c.c.\ partial order forcing 
 the existence of $\al_1$-many decompositions $\k=\bigcup_{n<\o}X^\x_n$ such that for all 
 $\x <\o_1$,
 \begin{itemize}
 \item there is $i_\x$ such that
 $[X^\x_n]^2\sub K^\x_{i_\x}$ for all $n<\o$
and 
\item $\x\in B$ if and only if for all $n<\o$, 
 $[X^\x_n]^2\sub K^\x_0$.
 \end{itemize}
 
Let $A_G$ be a subset of $\omega_1$ which canonically codes $B_G$ and $\bp{(X^\x_n)_{n<\o}:\x<\o_1}$. 
If $M$ is any outer model in which $\o_1^V$ has not been collapsed, then $B_G$ is the unique $B\sub\o_1$ 
for which there is,  in $M$, a set $A\sub\o_1$ coding $B$ together with $\al_1$-many decompositions  
$\bp{(X^\x_n)_{n<\o}:\x<\o_1}$ of $\k$ such that for all $n<\o$ and $\x <\o_1$, $\x\in B$ if and only if 
$[X^\x_n]^2\sub K^\x_0$. Indeed, if $B'\in M$ were another such set, as witnessed by $A'\sub M$, 
$\x\in B\Delta B'$, and $\al_1$-many decompositions  
$\bp{(Y^\x_n)_{n<\o}:\x<\o_1}\in M$ 
of $\k$,
%such that for all $\x <\o_1$, $\x\in B'$ if and only if 
%for all $n<\o$, $[Y^\x_n]^2\sub K^\x_0$, 
then there would be some $n$ and $m$ such that $X^\x_n\cap Y^\x_m$ 
has more than one element, and is in fact uncountable. But then, for every 
$s\in[X^\x_n\cap Y^\x_m]^2$, we would have that $s$ is both in $K^\x_0$ and $K^\x_1$, which is impossible.

 Finally, it is immediate to see that there is a $\Sigma_1$ formula $\varphi(x, y, z)$ such that $\varphi(G, A_G, p)$ expresses the above property of $G$ and $A_G$ over $(H_{\omega_2}; \in)^{M}$ for any $M$ as above, for $p=\vec K$. 
\end{proof}

The following principle, as well as Lemma \ref{psiAC-measurable}, are due to Woodin (\cite{WoodinBOOK}). 

\begin{definition} $\psi_{\AC}$ is the following statement: Suppose $S$ and $T$ are stationary and co-stationary subsets of $\o_1$. Then there are $\a<\o_2$ and a club $C$ of $[\a]^{\al_0}$ such that for every $X\in C$, $X\cap\o_1\in S$ if and only if $\ot(X)\in T$.
\end{definition}  

The $\AC$-subscript in the above definition hints at the fact that $\psi_{\AC}$ implies $L(\mtcl P(\o_1))\models\AC$ 
(which comes from an argument similar to the one in the proof of Lemma \ref{freeze2}).

Our third freezing poset is essentially the following forcing for adding a suitable instance of $\psi_{AC}$ by initial segments, using a measurable cardinal $\kappa$ (i.e., turning $\kappa$ into an ordinal $\alpha$ as required by the conclusion of $\psi_{AC}$).  

\begin{lemma}\label{psiAC-measurable}
Let $\kappa$ be a measurable cardinal and let $S$ and $T$ be stationary and co-stationary subsets of $\o_1$. 
Let $\mtcl Q=\mtcl Q_{\k, S, T}$ be the set, ordered by reverse inclusion, of all countable  $\sub$-continuous 
$\in$-chains  $p=(M^p_i)_{i\leq\n}$ of countable elementary substructures of $H_\k$ such that for 
every $i\leq\n$, $M^p_i\cap\o_1\in S$ if and only if $\ot(M^p_i\cap\k)\in T$. 
 
 \begin{enumerate}

\item $\mtcl Q$ is ${<}\omega_1$-semiproper, preserves Suslin trees, and does not add new reals.

\item if $G$ is $\mtcl Q$-generic over $V$ and $M^G_i= M^p_i$ whenever $p\in G$ and $i\in \dom(p)$, then 
$(M^G_i)_{i<\o_1}$ is the $\sub$-increasing enumeration of a club of $[H_\k^V]^{\al_0}$ such that for every limit ordinal $i<\o_1$, $M^G_i\cap\o_1\in S$ if and only if $\ot(M^G_i\cap\k)\in T$. 
\end{enumerate}
\end{lemma}

\begin{proof} The proofs of all assertions, except the fact that $\mtcl Q$ preserves Suslin trees, are standard. 
For the reader's convenience,  we sketch the proof that $\mtcl Q$ is ${<}\o_1$-semiproper,  though. We 
also prove that $\mtcl Q$ preserves Suslin trees. 

We get the ${<}\o_1$-semiproperness of $\mtcl Q$ as follows: the main point is that if $\mtcl U$ is a normal 
measure on $\kappa$, $N$ is an elementary submodel of some $H_\theta$ such that $\mtcl U\in N$ and 
$\av N\av<\k$, and $\eta\in \bigcap(\mtcl U\cap N)$, then $$N[\eta]:=\{f(\eta)\,:\, f\in N,\,f\mbox{ a function with 
domain }\kappa\}$$ is an elementary submodel of $H_\theta$ such that $N\cap\kappa$ is a proper initial segment 
of $N[\eta]\cap\kappa$ ($\eta\in N[\eta]$ is above every ordinal in $N\cap\kappa$, and any $\gamma\in N[\eta]\cap\eta$ 
is of the form $f(\eta)$, for some regressive function $f:\kappa\into\kappa$ in $N$ which, by normality of $\mtcl U$, 
is constant on some set in $\mtcl U\cap N$). If $N$ is countable, then by iterated applications of this construction, 
taking unions at nonzero limit stages, one  obtains a $\sub$-continuous and $\sub$-increasing sequence 
$(N_\nu)_{\nu<\omega_1}$ of elementary submodels of $H_\theta$ such that $N_0=N$ and such that 
$N_{\n'}\cap\kappa$ is a proper end-extension of $N_\n\cap\kappa$ for all $\n<\n'<\o_1$. 
Since $(\ot(N_\n\cap\kappa)\,:\,\n<\omega_1)$ is then a strictly increasing and continuous sequence 
of countable ordinals, we may find, by stationarity of $S$ and $\o_1\setminus S$, some $\n<\o_1$ 
such that $N\cap\omega_1\in S$ if and only if $\ot(N_\nu)\in T$. 

This observation yields the ${<}\omega_1$-semiproperness of $\mtcl Q$ since, given $\a<\o_1$ and an $\in$-chain $(N^\xi)_{\xi<\alpha}$ of countable elementary submodels of some $H_\theta$ such that $\mtcl U$, $S$, $T\in N^0$, one can run the above  construction for each $N^\xi$ by working inside $N^{\xi+1}$. 

The preservation of Suslin trees can be proved by the following version of the argument in \cite{Miyamoto-Yorioka} for showing that $\MRP$-forcings preserve them. Suppose $U$ is a Suslin tree, $\dot A$ is a $\mtcl Q$-name for a maximal antichain of $U$, and $N$ is a countable elementary submodel of some large enough $H_\t$ containing $U$, $\dot A$, and all other relevant objects. By moving to an $\o_1$-end-extension of $N$ if necessary as in the proof of ${<}\o_1$-semiproperness, we may assume that $N\cap\o_1\in S$ if and only if $\ot(N\cap\k)\in T$.  Let $(u_n)_{n<\o}$ enumerate all nodes in  $U$ of height $N\cap\o_1$. Given a condition $p\in\mtcl Q$ in $N$, we may build an $(N, \mtcl Q)$-generic sequence $(p_n)_{n<\o}$ of conditions in $N$ extending $p$ and such that for every $n$ there is some $v\in U$ below $u_n$ such that $p_{n+1}$ forces $v\in \dot A$. By the choice of $N$, we have in the end that $p^\ast=\bigcup_n p_n\cup\{(N\cap\o_1, N\cap \k)\}$ is a condition in $\mtcl Q$ extending $p$. But, by construction of $(p_n)_{n<\o}$, $p^\ast$ forces $\dot A$ to be contained in the countable set $U\cap N$: If $u\in \dot A\setminus N$, and $u_n$ is the unique node of height $N\cap\o_1$ such that $u_n\leq_U u$, then $u_n$ is forced to extend some node in $\dot A$ of height less than $N\cap\o_1$, which is a contradiction since $\dot A$ was supposed to be a name for an antichain. 

It remains to show how to find $p_{n+1}$ given $p_n$. Working in $N$, we first extend $p_n$ to some $p_n'$ in some suitable dense subset $D\in N$ of $\mtcl Q$. 
Since $U$ is a Suslin tree, we have that $u_n$ is totally $(U, N)$-generic, in the sense that for every antichain $B$ of $U$ in $N$, $u_n$ extends a unique node in $B$. Also, the set $E\in N$ of $u\in U$ for which there is some $v\in U$ below $u$ and some $q\in\mtcl Q$ extending $p_n'$ and forcing that $v\in\dot A$ is dense in $U$. It follows that we may find  some $u\in E\cap N$ below $u_n$, as witnessed by some $q\in \mtcl Q\cap N$ and some $v\in U\cap N$. But then we may let $p_{n+1}=q$. 
\end{proof}

Given a measurable cardinal $\kappa$, we will call a partial order $\mtcl R$ a \emph{$\psi_{\AC}^\kappa$-poset} if $\mtcl R$ is of the form $\mtcl Q_{\kappa, S, T}$ for stationary and co-stationary subsets $S$, $T$ of $\omega_1$.

\begin{proposition}\label{freeze2}

Given a forcing notion $\mtcl P$  and a measurable cardinal $\kappa$, there is $\mtcl P$-name $\dot{\mtcl Q}$ for  a forcing notion such that 
\begin{enumerate} 
\it $\dot{\mtcl Q}$ is forced to be of the form $\Coll(\o_1, \mtcl P)\ast\dot{\mtcl R}$, where $\dot{\mtcl R}$ is a $\Coll(\o_1, \mtcl P)$-name for a $\psi^\kappa_{AC}$-poset, and 
\it $\mtcl P\ast\dot{\mtcl Q}$ $\SSP$-freezes $\mtcl P$, as witnessed by the inclusion map.
  \end{enumerate}
\end{proposition}

\begin{proof} Let $(S_\a\,:\,\a<\o_1)$ be a partition of $\omega_1$ into stationary sets and let $T$ be a stationary 
and co-stationary subset of $\o_1$. Working in $V^{\mtcl P\ast\Coll(\o_1,\,\mtcl P)}$, let $B_{\dot G}\neq\emptyset$ 
be a subset of $\omega_1$ coding $\dot G$ in  a canonical way. We may assume that $B_{\dot G}\neq\omega_1$. 
Let $\dot{\mtcl R}$ be $\mtcl Q_{\kappa, S, T}$ for $S= \bigcup_{\alpha\in B_{\dot G}}S_\alpha$. 
By Lemma \ref{psiAC-measurable}, $\dot{\mtcl R}$ adds a club $C_{\dot G}$ of $[\kappa]^{\al_0}$ 
with the property that for each $X\in C_{\dot G}$, $X\cap \o_1\in \bigcup_{\alpha\in B_{\dot G}}S_\alpha$ if and only 
if $\ot(X)\in T$.  Now it is easy to see that $\mtcl P\ast\dot{\mtcl R}$ codes $\dot G$ in an absolute way in the sense 
of Lemma \ref{freezability-sufficient}. The main point is that if $H$ is a 
$(\mtcl P\ast\Coll(\o_1,\,\mtcl P))\ast\dot{\mtcl R}$-generic filter, $G=H\cap\mtcl P$, and $M$ is an outer model 
such that every stationary subset of $\o_1$ in $V[H]$ remains stationary in $M$, then in $M$ there is no $B'\sub\o_1$ 
such that $B'\neq B_G$ and such that there is a club  $C$ of $[\kappa]^{\al_0}$ with the property that for all 
$X\in C$, $X\cap\o_1\in\bigcup_{\alpha\in B'}S_\alpha$ if and only if $\ot(X)\in T$. Otherwise, if $\a\in B'\Delta B_G$, 
then there would be some $X\in C\cap C_G$ such that $X\cap \omega_1\in S_\alpha$. 
But then we would have that $\ot(X)$ is both in $T$ and in $\omega_1\setminus T$. 
\end{proof}

Let us move on now to our fourth freezing poset.   

Given cardinals $\m<\l$ with $\m$ regular, let $$S^\l_\m=\{\x<\l\,:\,\cf(\x)=\m\}$$ Let $\vec S = (S_\a)_{\a<\o_1}$ 
be a sequence of pairwise disjoint stationary subsets of $S^{\o_2}_\o$,  let $U\sub S^{\o_3}_\o$ be such that both 
$U$ and  $S^{\o_3}_\o\setminus U$ are stationary, and let $B\sub\o_1$. Then $\mtcl S_{\vec S, U, B}$ is 
the partial order, ordered by end-extension, consisting of all strictly $\sub$-increasing and $\sub$-continuous 
sequences $(Z_\n)_{\n\leq\n_0}$, for some $\n_0<\o_1$, such that for all $\n\leq\n_0$ and all $\a<\o_1$,  

\begin{itemize}

\it $Z_\n\in[\o_3]^{\al_0}$, and \it if $\mbox{sup}(Z_\n\cap\o_2)\in S_\a$, then $\mbox{sup}(Z_\n)\in U$ if and only if $\a\in B$. \end{itemize}

The proof of the following lemma appears in \cite{FOREMAGISHELAH} essentially.\footnote{See also the argument in the proof of Proposition \ref{freeze3} that $\dot{\mtcl Q}$ satisfies the $S$-condition.} 

\begin{lemma}\label{freezing-poset-3}
Let $\vec S= (S_\a)_{\a<\o_1}$ be a sequence of pairwise disjoint stationary subsets of $S^{\o_2}_\o$, let $U\sub S^{\o_3}_\o$ be such that both $U$ and  $S^{\o_3}_\o\setminus U$ are stationary, and let $B\sub\o_1$. Then $\mtcl S_{\vec S, U, B}$ preserves stationary subsets of $\o_1$, as well as the stationarity of all $S_\a$, and forces the existence of strictly $\sub$-increasing and $\sub$-continuous enumeration $(Z_\n\,:\,\n<\o_1)$ of a club of $[\o_3^V]^{\al_0}$ such that for all $\n$, $\a<\o_1$,  if $\mbox{sup}(Z_\n\cap\o_2^V)\in S_\a$, then $\mbox{sup}(Z_\n)\in U$ if and only if $\a\in B$.
\end{lemma}

\begin{proposition}\label{freeze3} For every partial order $\mtcl P$ and for all cardinals $\k^1>\k^0\geq\d\geq \av\mtcl P\av$, if $\Vdash_{\mtcl P\ast\Coll(\o_1,\,\d)}\k^0=\o_2$, $\Vdash_{\mtcl P\ast\Coll(\o_1,\,\d)}\k^1=\o_3$, $(S_\a)_{\a<\o_1}\in V$  is a  sequence of pairwise disjoint stationary subsets of $S^{\k^0}_\o$, and $U\sub S^{\k^1}_\o$ is a a stationary set in $V$ such that $S^{\k^1}_\o\setminus U$ is also stationary in $V$, then there is a $\mtcl P\ast \Coll(\o_1,\,\d)$-name $\dot B$  for a subset of $\o_1$ such that 

\begin{enumerate}
\item $\mtcl P$ forces $\Coll(\o_1,\, \d)\ast\dot{\mtcl S}_{\vec S, U, \dot B}$ to have the $S$-condition, and such that

\item $\mtcl P\ast (\Coll(\o_1,\, \d)\ast\dot{\mtcl S}_{\vec S, U, \dot B})$ $\SSP$-freezes $\mtcl P$, as witnessed by the inclusion map $$i:\mtcl P\into \mtcl P\ast (\Coll(\o_1,\,\d)\ast\dot{\mtcl S}_{\vec S, U, \dot B})$$
\end{enumerate}
\end{proposition}

\begin{proof}
Working in $V^{\mtcl P\ast \Coll(\o_1,\,\d)}$,  let $(S_\a)_{\a<\o_1}\in V$ and $U\in V$ be as stated, and let $B_{\dot G}$ be a subset of $\o_1$ coding the generic filter $G$ for $\mtcl P$ in a canonical way. Let $\dot{\mtcl R}$ be a $\mtcl P$-name for $\Coll(\o_1,\,\d)\ast \dot{\mtcl S}_{\vec S, U, B_{\dot G}}$.

\begin{claim}
$\mtcl P$ forces that $\dot{\mtcl R}$ has the $S$-condition. 
\end{claim}

\begin{proof} Since $\Coll(\o_1,\,\d)$ has the $S$-condition, it suffices to prove that $\mtcl P\ast\Coll(\o_1,\,\d)$ 
forces $\dot{\mtcl S}_{\vec S, U, \dot B}$ to have the $S$-condition. Let us work in $V^{\mtcl P\ast \Coll(\o_1,\,\d)}$. 
%Let $p\in\dot{\mtcl R}$. 
Let $\s$ be the following strategy for player II in $\mtcl G^{\dot{\mtcl R}}$: 
Whenever it is her turn to play, player II will alternate between the following courses of action (a), (b) 
(i.e., she will opt for (a) or (b) depending on the parity of the finite set 
$\{k<\av\eta\av\,:\,\av \textsf{succ}_T(\eta\restr k)\av>1\}$, with the notation used in Definition \ref{s-cond-def}).

\begin{itemize}

\item[(a)] Player II chooses $\k_\eta=\k^0$, $\textsf{succ}_T(\eta)=\{\eta^\smallfrown\la\a\ra\,:\,\a<\k^0\}$, and 
$(p_{\eta^\smallfrown\la\a\ra})_{\a<\k^0}$ where, for each $\a<\k^0$, $p_{\eta^\smallfrown\la\a\ra}$ is a condition extending 
$p_\eta$ and such that $\a\in \bigcup\range(p_{\eta^\smallfrown\la\a\ra})$.

\item[(b)]  Player II chooses $\k_\eta=\k^1$, $\textsf{succ}_T(\eta)=\{\eta^\smallfrown\la\a\ra\,:\,\a<\k^1\}$, and 
$(p_{\eta^\smallfrown\la\a\ra})_{\a<\k^1}$ where, for each $\a<\k^1$, $p_{\eta^\smallfrown\la\a\ra}$ is a condition 
extending $p_\eta$ and such that $\a\in \bigcup\range(p_{\eta^\smallfrown\la\a\ra})$.

\end{itemize}

Let now $T$ be the tree built along a run of $\mtcl G^{\dot{\mtcl R}}$ in which player II has played according to $\s$, 
let $T'$ be a subtree of $T$ such that $\av\textsf{succ}_{T'}(\eta)\av=\av\textsf{succ}_T(\eta)\av$ for every $\eta\in T'$, 
and let $N$ be a countable elementary substructure of some large enough $H_\t$ containing all relevant objects 
(which includes our run of $\mtcl G^{\dot{\mtcl R}}$ and $T'$), such that $\mbox{sup}(N\cap\k^0)\in S_0$, and such that 
$\mbox{sup}(N\cap \k^1)\in U$ if $0\in B_{\dot G}$ and $\mbox{sup}(N\cap \k^1)\notin U$ if $0\notin B_{\dot G}$. 

Such an $N$ can be easily found (s.\ \cite{FOREMAGISHELAH}):
Indeed, suppose, for concreteness, that $0\in B_{\dot G}$. Then,  letting $F:[H_\t]^{{<}\omega}\into H_\t$ 
be a function generating the club of countable elementary submodels of $H_\t$ containing all relevant objects, 
we  may find, using the stationarity of $U$, an ordinal $\a\in U$ such that the closure $X_0$ of $[\a]^{{<}\omega}$ 
under $F$ is such that $X_0\cap \k^1 = \alpha$. We may of course assume that $\a>\k^0$. Since $\cf(\a)=\omega$, 
we may pick a countable cofinal subset $Y$ of $\a$.  Using now the stationarity of $S_0$, we may find $\b\in S_0$ 
with the property that the closure $X_1$ of $[\b\cup Y]^{{<}\omega}$ under $F$ is such that $X_1\cap\kappa^0=\beta$. 
Since $\cf(\b)=\o$, we may now pick a countable subset $Z$ of $\b$. But then, letting $N$ be the closure of
 $Y\cup Z$ under $F$, we have that $\mbox{sup}(N\cap\k^0)=\b$ and $\mbox{sup}(N\cap\k^1)=\alpha$, and so 
 $N$ is as desired.

Letting $(p_\eta)_{\eta\in T'}$ be the tree of $\dot{\mtcl R}$-conditions corresponding to $T'$, it is now easy to find a 
cofinal branch $b$ through $T'$ such that for all $n<\o$, $b\restr n\in N$, and such that 
\[
\mbox{sup}(\bigcup_{n<\o}(\cup\range(p_{b\restr n})\cap\k^0))=  \mbox{sup}(N\cap\k^0)
\]
and 
\[
\mbox{sup}(\bigcup_{n<\o}(\cup\range(p_{b\restr n})))=  \mbox{sup}(N\cap\k^1).
\] 
Let 
\[
\n=\mbox{sup}\{\dom(p_{b\restr n})\,:\,n<\o\}\footnote{Incidentally, note that we cannot guarantee that $\n= N\cap\o_1$.}
\]
and 
\[
X=\bigcup_{n<\o}(\cup\range(p_{b\restr n})\cap\k^1).
\] 
It follows now that, letting $p_b=\bigcup_n p_{b\restr n}\cup\{\la\n, X\ra\}$, $p_b$ is a condition in 
$\dot{\mtcl R}$ forcing that $p_{b\restr n}$ is in the generic filter for all $n$.

\end{proof}

Going back to $V$, the proof that $\mtcl P\ast (\Coll(\o_1,\, \d)\ast\dot{\mtcl S}_{\vec S, U, \dot B})$ $
\SSP$-freezes $\mtcl P$ (as witnessed by the inclusion map) is very much like the proofs of Propositions \ref{freeze1} 
and \ref{freeze2}. Suppose $H$ is a generic filter for  $\mtcl P\ast (\Coll(\o_1,\, \d)\ast\dot{\mtcl S}_{\vec S, U, \dot B})$, 
$G=H\cap\mtcl P$, and $M$ is any outer model such that every stationary subset of $\o_1$ in $V[H]$ remains stationary 
in $M$. Suppose, towards a contradiction, that in $M$ there is some subset $B'\neq B_G$ of $\omega_1$ 
for which there is an $\sub$-increasing and $\sub$-continuous enumeration $(Z'_\n\,:\,\n<\o_1)$ of a club of 
$[\k^1]^{\al_0}$ such that for all $\n$, $\a<\o_1$,  if $\mbox{sup}(Z'_\n\cap\k^0)\in S_\a$, then $\mbox{sup}(Z'_\n)\in U$ 
if and only if $\a\in B'$. If $\a\in B'\Delta B_G$, there is some $\n$ such that $Z_\n=Z'_\n$ and 
$\mbox{sup}(Z_\n\cap \k^0)\in S_\a$. But then we have both $\mbox{sup}(Z_\n)\in U$ and $\mbox{sup}(Z_\n)\notin U$.

Finally, the existence of a $\Sigma_1$ definition---with $p=(\l, \vec S, U)$ as parameter---as required by 
Lemma \ref{freezability-sufficient} is easy. 
 \end{proof}

\subsubsection{Iteration Lemmas}\label{iterability}

%We start by proving that $\PR$, and each of its restrictions mentioned in Theorem \ref{main-thm-suitable}, is $\o_1$--suitable.

We need the following preservation lemmas, due to Shelah (\cite[III, resp. VI]{SHEPRO}, 
see also \cite{Goldstern}).

\begin{lemma}\label{presproper} Suppose $\rho<\o_1$ is an indecomposable ordinal and $$(\mtcl P_\a, \dot{\mtcl Q}_\b\,:\,\a\leq\g,\,\b<\g)$$ is a countable support iteration such that for all $\b<\g$, $$\Vdash_{\mtcl P_\b}\dot{\mtcl Q}_\b\in \rho\mbox{-}\PR$$ Then $\mtcl P_\g\in \rho\mbox{-}\PR$.

\end{lemma}

\begin{lemma}\label{presomegaomegabounding} Suppose $\rho<\o_1$ is an indecomposable ordinal and $$(\mtcl P_\a, \dot{\mtcl Q}_\b\,:\,\a\leq\g,\,\b<\g)$$ is a countable support iteration such that for all $\b<\g$, $$\Vdash_{\mtcl P_\b}\dot{\mtcl Q}_\b\in (\rho\mbox{-}\PR)\cap\,^\o\o\textsf{-bounding}$$ Then $\mtcl P_\g\in (\rho\mbox{-}\PR)\cap\,^\o\o\textsf{-bounding}$.
\end{lemma}

We will also use the following preservation result due to Miyamoto.

\begin{lemma}\label{pressuslintreepres} Suppose $\rho<\o_1$ is an indecomposable ordinal and $$(\mtcl P_\a, \dot{\mtcl Q}_\b\,:\,\a\leq\g,\,\b<\g)$$ is a countable support iteration such that for all $\b<\g$, $$\Vdash_{\mtcl P_\b}\dot{\mtcl Q}_\b\in (\rho\mbox{-}\PR)\cap\STP$$ Then $\mtcl P_\l\in (\rho\mbox{-}\PR)\cap\STP$.

\end{lemma}

Shelah defines a certain variant of the notion of countable support iteration, which he calls \emph{revised countable 
support (RCS) iteration}. Variants of the notion of RCS iteration have been proposed by Miyamoto and others 
(for example a detailed account of RCS-iterations in line with Donder and Fuchs' approach is given 
in~\cite{VIAAUDSTEBOOK}).\footnote{It is not 
%always 
clear whether these notions are equivalent 
in any reasonable sense.}
In the following, any mention of revised countable support iteration will refer to either Shelah's or Miyamoto's version. 

The first preservation result involving RCS iterations we will need is the following lemma, proved in \cite[XI]{SHEPRO}. 

\begin{lemma}\label{preservation-S-cond}
Suppose $\la\mtcl P_\a, \dot{\mtcl Q}_\b\,:\,\a\leq\g,\,\b<\g\ra$ is an RCS iteration such that the following holds for all $\b<\g$.

\begin{enumerate}

\item If $\b$ is even, $\Vdash_{\mtcl P_\b}\dot{\mtcl Q}_\b=\textsf{Coll}(2^{\av\mtcl P_\b\av},\,\o_1)$.
\item If $\b$ is odd, $\Vdash_{\mtcl P_\b}\dot{\mtcl Q}_\b\mbox{ has the $S$-condition}$.
\end{enumerate}

Then $\mtcl P_\g$ has the $S$-condition. \end{lemma}

The following is a well-known result of Shelah.

\begin{lemma}\label{pres-semiproper} Suppose  $\rho<\o_1$ is an indecomposable ordinal and $$(\mtcl P_\a, \dot{\mtcl Q}_\b\,:\,\a\leq\g,\,\b<\g)$$ is a revised countable support iteration such that for all $\b<\g$, $$\Vdash_{\mtcl P_\b}\dot{\mtcl Q}_\b\in \rho\mbox{-}\SP$$ Then $\mtcl P_\g\in \rho\mbox{-}\SP$.
\end{lemma}

We will also need the following lemmas due to Miyamoto \cite{Miyamoto,miyamoto-suslin-sp}.

\begin{lemma}\label{pres-semiproper-bounding}  Suppose $\CH$ holds, $\rho<\o_1$ is an indecomposable ordinal, and $$(\mtcl P_\a, \dot{\mtcl Q}_\b\,:\,\a\leq\g,\,\b<\g)$$ is a revised countable support iteration such that for all $\b<\g$, $$\Vdash_{\mtcl P_\b}\dot{\mtcl Q}_\b\in (\rho\mbox{-}\SP)\cap\,^\o\o\textsf{-bounding}$$ Then $\mtcl P_\g\in (\rho\mbox{-}\SP)\cap\,^\o\o\textsf{-bounding}$.
\end{lemma}

\begin{lemma}\label{miyamotosuslinsp}  Suppose $\rho<\o_1$ is an indecomposable ordinal and $$(\mtcl P_\a, \dot{\mtcl Q}_\b\,:\,\a\leq\g,\,\b<\g)$$ is a revised countable support  iteration such that for all $\b<\g$, $$\Vdash_{\mtcl P_\b}\dot{\mtcl Q}_\b\in(\rho\mbox{-}\SP)\cap\STP$$ Then $\mtcl P_\g\in (\rho\mbox{-}\SP)\cap\STP$.
\end{lemma}

\subsubsection{Absolutely well-behaved classes}\label{omega1suitability}

\begin{lemma}\label{pr-suitable} The following classes 
$\Gamma$ are absolutely well-behaved and such that $\lambda_\Gamma=\omega_1$.
%$\o_1$-suitable with respect to the theory
%$\MK$+`\emph{$\omega_1$ is the least uncountable cardinal}' + `\emph{$\rho$ is a countable indecomposable ordinal}' for every indecomposable ordinal $\rho<\o_1$.

\begin{enumerate}

\item $\rho\mbox{-}\PR$ for every countable indecomposable ordinal $\rho$.

\item  ${<}\o_1\mbox{-}\PR$
\item $\PR\cap \STP$ 
\item $\PR\cap\,^\o\o\textsf{-bounding}$ 
\item $\PR\cap\STP\cap\,^\o\o\textsf{-bounding}$ 
\end{enumerate}
\end{lemma}

\begin{proof}
Given any of these classes $\Gamma$, all conditions in the definition of absolutely well-behaved
%$\o_1$--suitable 
class---except for the fact that $\Gamma$ has the $\Gamma$-freezability property---are clearly satisfied for 
$\Gamma$. In particular, $\Gamma$ is defined by both a
%n absolutely
 $\Sigma_2$ property and a
 %n absolutely 
 $\Pi_2$ property by Fact~\ref{fac:sigma2proper},\footnote{By Remark \ref{sufficience-remark} this is enough to guarantee clause (5) in the definition of well-behaved class. And of course the same applies in the proofs of Lemmas \ref{scondlemma} and \ref{ssp-and-restrictions}.}
%and 
it is 
%closed under preimages by complete injective homomorphisms, two-step iterations and products, 
closed under isomorphisms, two-step iterations, lottery sums, restrictions and complete subalgebras,
and it contains all countably closed forcings 
by Fact~\ref{fac:Birkhoffproper}.
The 
%weak
 iterability property follows from Lemmas \ref{presproper}, \ref{presomegaomegabounding} 
and \ref{pressuslintreepres}: the winning strategy for player II is to play the countable support limit at all limit stages
(notice that at stages of cofinality $\omega_1$ this limit is the direct limit).  
 As to the freezability property, it turns out that $\Gamma$ has in fact the $\SSP$-freezability property. This 
 follows immediately from Proposition \ref{freeze1} together with Lemma \ref{MRP} for $\PR$, as well as for the classes 
 in (3), (4) and (5), and from Proposition \ref{freeze1.5} for $\rho\mbox{-}\PR$, for any given indecomposable 
 ordinal $\rho<\o_1$ such that $\rho>1$, and for ${<}\o_1\mbox{-}\PR$.
 Finally, it is clear that $\lambda_\Gamma=\omega_1$ holds for each of these classes $\Gamma$.
\end{proof}

We move on now to our first class not contained in $\PR$.

\begin{lemma}\label{scondlemma}  $S\textsf{-cond}$ is an absolutely well-behaved class $\Gamma$ with $\lambda_\Gamma=\omega_1$. 
%$\o_1$-suitable with respect to 
%$\MK$+ `\emph{$\omega_1$ is the least uncountable cardinal}'.
\end{lemma}

\begin{proof}
Let $\Gamma=S\textsf{-cond}$. Except for the freezability condition, all conditions in the definition of absolutely well-behaved
%$\o_1$-suitable 
class are clearly satisfied by 
$\Gamma$:
%\footnote{The closure of  $S\textsf{--cond}$ under complete subalgebras follows of course directly 
%from the definition of $S\textsf{--cond}$  as the class of complete Boolean algebras $\bool{B}$ such that $\bool{B}$ is 
%a complete suborder of some complete Boolean algebra satisfying the $S$--condition. On the other hand, it is not clear to 
%us whether a complete Boolean subalgebra of a complete Boolean algebra with the $S$--condition necessarily has the 
%$S$--condition.}. 
$\Gamma$ is defined by both a
%an absolutely 
$\Sigma_2$ property and a $\Pi_2$ property, 
it is closed under isomorphisms, two-step iterations, lottery sums, restrictions and complete subalgebras,
%closed under preimages by complete injective homomorphisms, two-step iterations and products, 
and it contains all countably closed forcings 
by Fact~\ref{fac:sigma2Scond}.
The iterability condition follows immediately from Lemma \ref{preservation-S-cond}: 
the winning strategy for player II is to play the revised countable support limit at all limit stages
(notice that at stages of cofinality $\omega_1$ this limit is the direct limit), and to play at all non-limit stages
$\alpha+2n$ the algebra $\bool{B}_{\alpha+2n}=\bool{B}_{\alpha+2n-1}\ast\dot{\bool{C}}$, where 
$\dot{\bool{C}}$ is a $\bool{B}_{\alpha+2n-1}$-name for the boolean completion of 
$\Coll(\omega_1,2^{|\bool{B}_{\alpha+2n-1}|})$. 
As to the freezability condition, 
we have that $S\textsf{-cond}$ has in fact,  by Proposition \ref{freeze3}, the $\SSP$-freezability condition---which implies the $\Gamma$-freezability condition by Lemma \ref{S-cond-shelah} (1). Finally, it is clear that 
$\lambda_\Gamma=\omega_1$. 
\end{proof}

\begin{lemma}\label{ssp-and-restrictions}
 Suppose there is a proper class of measurable cardinals. Given any indecomposable ordinal $\rho<\omega_1$, 
 each of the following classes $\Gamma$ is absolutely well-behaved and such that $\lambda_\Gamma=\omega_1$. 
 %$\o_1$-suitable with respect to
 %$\MK$+ `\emph{$\omega_1$ is the least uncountable cardinal}'+`\emph{$\rho$ is a countable indecomposable ordinal}'+ `\emph{there are class many measurable cardinals}'.

\begin{enumerate}
\item $\rho\mbox{-}\SP$
\item $(\rho\mbox{-}\SP)\cap \STP$
\item $(\rho\mbox{-}\SP)\cap\,^\o\o\textsf{-bounding}$
\item $(\rho\mbox{-}\SP)\cap\STP\cap\,^\o\o\textsf{-bounding}$ 
\end{enumerate}

Also, each of the following classes is $\o_1$-suitable with respect to the same theory.

\begin{enumerate}
\item ${<}\o_1\mbox{-}\SP$
\item $({<}\o_1\mbox{-}\SP)\cap \STP$
\item $({<}\o_1\mbox{-}\SP)\cap\,^\o\o\textsf{-bounding}$
\item $({<}\o_1\mbox{-}\SP)\cap\STP\cap\,^\o\o\textsf{-bounding}$ 
\end{enumerate}
\end{lemma}

\begin{proof}
Each of these $\Gamma$ is defined both by a 
%absolutely 
$\Sigma_2$ property and 
by a $\Pi_2$ property by Fact~\ref{fac:sigma2semiproper},
and 
is closed under preimages by complete injective homomorphisms, two-step iterations and products, and contains all countably closed forcings 
by Fact~\ref{fac:Birkhoffsemiproper}.
The iterability condition for each of these classes follows from (some combination of) 
Lemmas \ref{pres-semiproper}, \ref{pres-semiproper-bounding}, and \ref{miyamotosuslinsp}: 
the winning strategey for player II is to play the revised countable support limit at all limit stages
(notice that at stages of cofinality $\omega_1$ this limit is the direct limit).
%, and to play at all non-limit stages
%$\alpha+2n$ $\bool{B}_{\alpha+2n}=\bool{B}_{\alpha+2n-1}\ast\dot{\bool{C}}$ where 
%$\dot{\bool{C}}$ is a $\bool{B}_{\alpha+2n-1}$-name for the boolean completion of 
%$\Coll(\omega_1,|\bool{B}_{\alpha+2n-1}|)$. 
The freezability condition follows  from Lemma \ref{psiAC-measurable}, together with Proposition \ref{freeze2} 
(for the case $\rho=1$, one could as well invoke Lemma \ref{MRP} together with Proposition \ref{freeze1} instead). Finally, the fact that $\lambda_\Gamma=\omega_1$ is again immediate. 
\end{proof}

The standard proof, due to Shelah  (\cite{SPFA=MM}), that $\SPFA$ implies $\SSP=\SP$ actually shows the following.

\begin{proposition} $\textsf{FA}(({<}\o_1\mbox{-}\SP)\cap\STP\cap\,^\o\o\mbox{-bounding})$ implies 
$\SSP=\SP$. \end{proposition}

\begin{proof} This follows from the fact that  the natural semiproper forcing $\mtcl Q_{\mtcl P}$  
(s.\ \cite{SPFA=MM}) such that an application of $\FA_{\al_1}(\{\mtcl Q_{\mtcl P}\})$ 
yields the semiproperness of a given $\SSP$ $\mtcl P$ is in fact ${<}\o_1$-semiproper, does not add new reals, 
and preserves Suslin trees. The proof of the first two assertions is straightforward, and the preservation of 
Suslin trees follows by an argument as in the final part of the proof of Lemma \ref{psiAC-measurable}.
\end{proof} 
 
\begin{corollary} Suppose $\textsf{FA}(({<}\o_1\mbox{-}\SP)\cap\STP\cap\,^\o\o\mbox{-bounding})$ holds. 
Then:
\begin{enumerate}
\item $\BCFA(\SSP)$ holds iff $\BCFA(\SP)$ does.
\item $\BCFA(\SSP\cap \STP)$ holds iff $\BCFA(\SP\cap \STP)$ does.
\item $\BCFA(\SSP\cap\,^\o\o\textsf{-bounding})$ holds iff $\BCFA(\SP\cap\,^\o\o\textsf{-bounding})$ does.
\item $\BCFA(\SSP\cap\STP\cap\,^\o\o\textsf{-bounding})$ holds iff $\BCFA(\SP\cap\STP\cap\,^\o\o\textsf{-bounding})$ does.
\end{enumerate}
\end{corollary}

\subsubsection{Pairwise incompatibility of $\BCFA(\Gamma)$ for absolutely well-behaved $\Gamma$}
\label{incompatible-category-forcing-axioms}

Each one of the incompatibilities contained in Theorem \ref{main-thm-incompatible} follows from two or more of 
the lemmas in this subsection put together.

%Recall that $\bm{\delta}^1_2$ is the supremum of the the set of lengths of $\bm{\Delta}^1_2$-definable  
%pre-well-orderings on $\mtbb R$. 

%Also
Given an ordinal $\a<\o_2$,  a function $g:\o_1\into \o_1$ is a \emph{canonical function for $\a$} if 
there is a surjection $\p:\o_1\into\a$ and a club $C\sub\o_1$ such that for all $\nu\in C$, $g(\nu)=\ot(\pi``\n)$. 
Let \emph{Club Bounding} denote the following statement: For every function $f:\o_1\into \o_1$ there is 
some $\a<\o_2$ such that  $\{\n<\o_1\,:\, f(\n)<g(\n)\}$ contains a club  whenever $g$ is a canonical function for $\a$.

\begin{lemma}\label{consespfa+++} 
%($\ZFC+\LC$) 
Suppose $\delta$ is an inaccessible cardinal such that $V_\delta\prec V$. 
Let $\Gamma$ be any absolutely well-behaved class, defined with a parameter in $V_\delta$, and
%$\o_1$-suitable class
 such that $\lambda_\Gamma=\omega_1$. Suppose $\Gamma\sub \PR$ holds after forcing with $\Add(\omega_1, 1)$. If $\BCFA(\Gamma)$ holds, then Club Bounding fails.
\end{lemma}

\begin{proof} To see that Club Bounding fails in $V$, we first force with $\Add(\omega_1, 1)$. Let $G$ be the corresponding generic filter. 
%Let $f:\o_1\into \o_1$ be the generic function, 
We 
then force with $\Gamma_\delta^{V[G]}$.
By Theorem \ref{thm:mainth1}, it holds in $V[G]$ that $\Gamma_\delta^{V[G]}$ forces $\BCFA(\Gamma)$.  Also, it is immediate to check that $G$ adds a function $f:\o_1\into\o_1$ such that  
$$\{X\in [\a]^{\al_0}\,:\, X\cap\o_1\in\o_1,\,\ot(X)<f(X\cap\o_1)\}$$ is a stationary subset of $[\a]^{\al_0}$ for every ordinal 
$\a$. Since every proper forcing will preserve the stationarity of these sets and since $\Gamma_\delta^{V[G]}$ is proper in $V[G]$ by our assumption, it follows that Club Bounding fails in 
$(V[G])^{\Gamma_\delta}$. Since the failure of Club Bounding
is expressible 
over $H_{\o_2}$, and since we have that $H_{\omega_2}^{V^{\Gamma_\delta}}=H_{\omega_2}^{V^{\Add(\omega_1, 1)\ast\dot{\Gamma}_\delta}}\prec H_{\omega_2}^{V^\Gamma}$ and $H_{\o_2}^V\prec H_{\omega_2}^{V^\Gamma}$, it follows that Club Bounding fails in $V$. 
%We know that $\bool{U}^{\Gamma}_\delta$, for any $2$-superhuge cardinal $\d$, is in $\Gamma$, collapses $\o_2^{V}$ to $\al_1$ and, by Theorem \ref{consistency-of-CFA}, forces $\CFA(\Gamma)$ .
% Also, using
 % our background large cardinal assumption (in fact a  
%and since we have that $H_{\omega_2}^{V^{\Gamma_\delta}}\prec H_{\omega_2}^{V^\Gamma}$ and $H_{\o_2}^V\prec H_{\omega_2}^{V^\Gamma}$, it follows that $V\models\bm{\delta}^1_2<\o_2$. 
%now from the absoluteness theorem that $V\models\bm{\delta}^1_2<\o_2$. 
%$V^{\bool{U}^{\Gamma}_\delta}\models\bm{\delta}^1_2<\o_2$. Since `$\bm{\delta}^1_2<\o_2$'  is expressible 
%over $H_{\o_2}$, it follows now from the absoluteness theorem that $V\models\bm{\delta}^1_2<\o_2$. 
%To see that Club Bounding fails in $V$, we first add generically a function $f:\o_1\into \o_1$ by initial segments and 
%then force with the resulting version of $\Gamma_\delta$.  It is immediate to check that, after  adding $f$, 
%$$\{X\in [\a]^{\al_0}\,:\, X\cap\o_1\in\o_1,\,\ot(X)<f(X\cap\o_1)\}$$ is a stationary subset of $[\a]^{\al_0}$ for every ordinal 
%$\a$. Since every proper forcing will preserve the stationarity of these sets and since $\Gamma_\delta^{V[f]}$ is proper in $V[f]$, it follows that Club Bounding fails in 
%$(V[f])^{\Gamma_\delta}$. But then it also has to fail in $V$ by the same absoluteness argument as before.
\end{proof}

\begin{lemma} $\BFA_{\al_1}(({<}\o_1\mbox{-}\PR)\cap\,^\o\o\textsf{-bounding})$ implies that there are no Suslin trees. \end{lemma} 

\begin{proof}
Suppose, towards a contradiction, that $T$ is a Suslin tree and the forcing axiom $\BFA_{\al_1}(({<}\o_1\mbox{-}\PR)\cap\,^\o\o\textsf{-bounding})$ 
holds. Without loss of generality we may assume that $T$ is a normal Suslin tree. We have that $T$ is a c.c.c.\ forcing 
which is $^\o\o$-bounding as in fact it does not add new reals. But forcing with $T$ adds an $\o_1$-branch through 
$T$. Hence, by $\BFA_{\al_1}(T)$, $T$ has an $\o_1$-branch and so it is not Suslin, which is a contradiction. 
\end{proof}

Recall that $\mathfrak d$ is the minimal cardinality of  a family $\mtcl F\sub\,^\o\o$ with the property that for every 
$f:\o\into \o$ there is some $f\in\mtcl F$ such that $g(n)<f(n)$ for a tail of $n<\o$.

\begin{lemma}\label{consespfa+++omegaomega-bounding}  Suppose $\delta$ is an inaccessible cardinal such that $V_\delta\prec V$. 
Let $\Gamma$ be any absolutely well-behaved class, defined with a parameter in $V_\delta$, and
%$\o_1$-suitable class
 such that $\lambda_\Gamma=\omega_1$.
Suppose $\Gamma\sub\,^\o\o\mbox{-bounding}$. If $\BCFA(\Gamma)$ holds, 
then $\mathfrak d = \omega_1$.
\end{lemma}

\begin{proof}
We know that $\Gamma_\delta$ forces $\BCFA(\Gamma)$, and collapses $(2^{\al_0})^{V}$ to $\al_1$.  Since $\Gamma\sub\,^\o\o\mbox{-bounding}$, we thus have that 
$V^{\Gamma_\delta}\models \mathfrak d=\o_1$ . 
But `$\mathfrak d=\o_1$' is expressible over $H_{\o_2}$, and therefore $V\models\mathfrak d =\o_1$ by 
the same absoluteness argument as in the proof of Lemma \ref{consespfa+++}.
\end{proof}

\begin{lemma}\label{FAPRcapSTPvsd}
$\BFA_{\al_1}(({<}\o_1\mbox{-}\PR)\cap\STP)$ implies $\mathfrak d>\omega_1$.
\end{lemma}

\begin{proof} This is immediate since Cohen forcing, being countable, preserves Suslin trees.\end{proof}

%Recall that the Strong Reflection Principle ($\SRP$) is the following assertion: For every set $X$ such that 
%$\o_1\sub X$ and every $S\sub [X]^{\al_0}$ there is a strong reflecting sequence $(x_i)_{i<\o_1}$ for $S$, i.e., 
%$x_i\in [X]^{\al_0}$, $(x_i)_{i<\o_1}$ is strictly $\sub$-increasing and $\sub$-continuous, and for all $i$, $x_i\notin S$ if 
%and only if there is no $y\in S$ such that $x_i\sub y$ and $y\cap\o_1 =x_i\cap\o_1$. 

%\begin{lemma}\label{sspomegaomegaboundingvsdelta12} 
% $\FA_{\al_1}(\SP\cap \STP\cap\,^\o\o\mbox{-bounding})$ implies $\bm{\delta}^1_2=\o_2$. 
%\end{lemma} 

%\begin{proof} 
%We have that  $\FA_{\al_1}(\SP\cap \STP\cap\,^\o\o\mbox{-bounding})$ implies $\SRP$ since, given 
%$S\sub [X]^{\al_0}$, the standard forcing for adding a strong reflecting sequence for $S$ is semiproper, does not add reals, 
%and preserves Suslin trees, where the last fact follows from an argument as in the final part of the proof of 
%Lemma \ref{psiAC-measurable}. Also, $\SRP$ implies $\lnot\Box_\kappa$, for every cardinal $\k\geq\o_1$, 
%and hence implies that the universe is closed under sharps. Since it also implies the saturation of $\NS_{\o_1}$, 
%by a classical result of Woodin (\cite{WoodinBOOK}) it implies $\bm{\delta}^1_2=\o_2$. 
%\end{proof}

%\begin{question} Does  $\FA_{\al_1}(\o\mbox{-}\SP)$ imply $\bm{\delta}^1_2=\o_2$?
%\end{question}

\begin{lemma}\label{bfa-o1-sp-stp-oo-bd-Club-Bounding} Suppose $\BFA_{\al_1}(({<}\o_1\mbox{-}\SP)\cap \STP\cap\,^\o\o\mbox{-bounding})$ holds and 
there is a measurable cardinal. Then $\psi_{AC}$ holds, and therefore Club Bounding holds as well. 
\end{lemma}

\begin{proof} $\BFA_{\al_1}(({<}\o_1\mbox{-}\SP)\cap \STP\cap\,^\o\o\mbox{-bounding})$ implies $\psi_{AC}$ by Lemma \ref{psiAC-measurable}. To see that $\psi_{AC}$ implies Club Bounding, see for example Fact 3.1 in \cite{Aspero-Welch}.
%Given a function $f:\o_1\into\o_1$ and a measurable cardinal $\k$, let $\mtcl Q^\k_f$ be the set, 
%ordered by reverse inclusion, of all strictly $\sub$-increasing and $\sub$-continuous sequences $(x_i)_{i\leq\a}$, 
%for $\a<\o_1$, of countable subsets of $\k$ such that for all $i$, $x_i\cap\o_1\in \o_1$ and $\ot(x_i)>f(x_i\cap\o_1)$. 
%Then $\mtcl Q^\k_f$ is a ${<}\o_1$-semiproper forcing not adding reals, preserving Suslin trees, and adding a 
%canonical function for $\k$ dominating $f$ on a club. All this can be proved by a straightforward variation of 
%the proof of Lemma \ref{psiAC-measurable}.
\end{proof}

The following lemma is an immediate consequence of Lemma \ref{bfa-o1-sp-stp-oo-bd-Club-Bounding}.

\begin{lemma}\label{yielding-psiAC} Suppose $\delta$ is an inaccessible cardinal such that $V_\delta\prec V$. Suppose there is a proper class of measurable cardinals. Let $\Gamma$ be an absolutely well-behaved class with $\lambda_\Gamma=\omega_1$ defined from a parameter in $V_\delta$  and such that $({<}\o_1\mbox{-}\SP)\cap \STP\cap\,^\o\o\mbox{-bounding}\subseteq \Gamma$ holds in any generic extension by a member of $\Gamma$. If $\BCFA(\Gamma)$ holds, then $\psi_{AC}$, and therefore also Club Bounding, hold as well. 
\end{lemma}

\begin{proof} $\Gamma_\delta$ forces $\BFA_{\al_1}(\Gamma)$ and therefore also $\BFA_{\al_1}(({<}\o_1\mbox{-}\SP)\cap \STP\cap\,^\o\o\mbox{-bounding})$. Hence, $H_{\o_2}^{V^{\Gamma_\delta}}\models\psi_{AC}$ by Lemma \ref{bfa-o1-sp-stp-oo-bd-Club-Bounding}. But then $V\models \psi_{AC}$ by the usual absoluteness argument. \end{proof}

\begin{remark}\label{remark-on-cons-strength} The conclusion, in Lemma \ref{yielding-psiAC}, that Club Bounding holds is equiconsistent with the existence of an inaccessible limit of measurable cardinals \cite{Deiser-Donder}. In fact, if there is no inner model with an inaccessible limit of measurable cardinals, then Club Bounding fails.
\end{remark}

Lemma \ref{bfa-o1-pr-wdiamond} follows trivially from the fact that $\BFA_{\al_1}({<}\o_1\mbox{-}\PR)$ implies $\MA_{\o_1}$.
%proved in \cite{Velickovic} (s.\ Proposition \ref{freeze1.5}).

\begin{lemma}\label{bfa-o1-pr-wdiamond} $\BFA_{\al_1}({<}\o_1\mbox{-}\PR)$ implies $2^{\al_0}=2^{\al_1}$. 
\end{lemma}

%\begin{lemma} $\BFA_{\al_1}({<}\o_1\mbox{-}\PR)$ implies $2^{\al_0}=2^{\al_1}=\al_2$ 
%\end{lemma}

%$2^{\al_0}=2^{\al_1}$ follows of course already from $\MA_{\o_1}$.

A partial order $\mtbb P$ is said to have the \emph{$\sigma$-bounded chain condition} if 
$\mtbb P=\bigcup_{n<\o}\mtbb P_n$ and for each $n$ there is some $k_n<\o$ such that for every 
$X\in [\mtbb P_n]^{k_n}$ there are distinct $p$, $p'\in\mtbb P_n$ which are compatible in $\mtbb P$. Also, 
a partial order $\mtbb P$ is \emph{Knaster} if every uncountable subset of $\mtbb P$ contains an uncountable 
subset consisting of pairwise compatible conditions in $\mtbb P$.  

It is easy to see, and a well-known fact, that random forcing preserves Suslin trees. This follows from the fact that 
random forcing has the $\s$-bounded chain condition, that every forcing with the $\s$-bounded chain condition is 
Knaster, and that every Knaster forcing preserves Suslin trees.

Lemma \ref{random} follows from the above, together with the fact that random forcing is $^\o\o$-bounding and 
adds  a new real. 

\begin{lemma}\label{random} $\BFA_{\al_1}(({<}\o_1\mbox{-}\PR)\cap \STP\cap\,^\o\o\textsf{-bounding})$ implies $\lnot\CH$.
\end{lemma}

\begin{lemma}\label{scnd-conses} Suppose $\delta$ is an inaccessible cardinal such that $V_\delta\prec V$. 
Let $\Gamma$ be any absolutely well-behaved class, defined with a parameter in $V_\delta$, and
 such that $\lambda_\Gamma=\omega_1$. Suppose $\Gamma \sub S\textsf{-cond}$ holds after forcing with $\Add(\omega_1, 1)$. If $\BCFA(\Gamma)$ holds, then so does $\CH$.
\end{lemma}

\begin{proof} Let $G$ be $\Add(\omega_1, 1)$-generic, and note that $V[G]$ satisfies $\CH$. Then force $\BCFA(\Gamma)$ over $V[G]$ via 
$\Gamma_\delta$. 
Let $V_1$ be this extension. 
Since, in $V[G]$, $\Gamma_\delta$ has the $S$-condition, forcing with $\Gamma_\delta^{V[G]}$ over 
$V[G]$ did not add new reals thanks to Lemma \ref{S-cond-shelah} (2). In particular, 
$V_1\models\CH$. But then $\CH$ holds in $V$ by the usual absoluteness argument.
\end{proof}

It tuns out that $\BCFA(\Gamma)$, where $\Gamma$ is any absolutely well-behaved class such that $\lambda_\Gamma=\omega_1$ and $\Gamma\subseteq S\textsf{-cond}$ holds after adding a Cohen subset of $\omega_1$
%$\o_1$-suitable class 
actually implies $\diamondsuit$. The proof is essentially the same as above, 
using the following recent result due to Magidor, together with the fact that if $V\sub V_1\sub W$ are models 
with the same $\o_1$ and $\vec X\in V$ is a $\diamondsuit$-sequence in $W$, then $\vec X$ is also a 
$\diamondsuit$-sequence in $V_1$. 

\begin{theorem} (Magidor) Suppose $\diamondsuit$ holds. Then there is a $\diamondsuit$-sequence 
that remains a $\diamondsuit$-sequence after any forcing with the $S$-condition.
\end{theorem}

The following well-known fact can be proved by an argument as in the final part of the proof of 
Lemma \ref{psiAC-measurable}.

\begin{fact}\label{suslin-tree-pres-sigma-closed} If $\mtcl P$ is countably closed, then $\mtcl P$ preserves Suslin trees.
\end{fact}

The proof of the following lemma is like the proofs of Lemmas 
\ref{consespfa+++}, \ref{consespfa+++omegaomega-bounding}, and \ref{scnd-conses}, 
using the well-known fact that $\Add(\o_1, 1)$ adds a Suslin tree $T$. 

\begin{lemma}\label{consesmm+++pressuslintrees} Suppose $\delta$ is an inaccessible cardinal such that $V_\delta\prec V$. 
Let $\Gamma$ be any absolutely well-behaved class, defined with a parameter in $V_\delta$, and
 such that $\lambda_\Gamma=\omega_1$. Suppose $\Gamma \sub\STP$ holds after forcing with $\Add(\omega_1, 1)$.
If $\BCFA(\Gamma)$ holds, then 
there is a Suslin tree. 
\end{lemma}

It will be convenient to consider the following families of Club-Guessing principles on $\o_1$ (s.\ \cite{AFMS}).

\begin{definition}
Let $\tau<\o_1$ be a nonzero ordinal. 

\begin{enumerate}

\item $\tau$-$\TWCG$ denotes the following statement: There is a a sequence 
$$\vec C=(C_\d\,:\,\d=\o^\tau\cdot\eta\mbox{ for some nonzero }\eta<\o_1)$$ such that 
$\av\{C_\d\cap\g\,:\,\g<\o_1\}\av\leq\al_0$ for every $\d\in\dom(\vec C)$, and such that for every club $C\sub\o_1$ 
there is some $\d\in\dom(\vec C)$ with $\ot(C_\d\cap C)=\o^\tau$. 

\item $\tau$-$\TCG$ denotes the following statement: There is a a sequence 
$$\vec C=(C_\d\,:\,\d=\o^\tau\cdot\eta\mbox{ for some nonzero }\eta<\o_1)$$ 
such that $\av\{C_\d\cap\g\,:\,\g<\o_1\}\av\leq\al_0$ for every $\d\in\dom(\vec C)$, and such that for every club 
$C\sub\o_1$ there is some $\d\in\dom(\vec C)$ with $C_\d\sub C$. 
\end{enumerate}
\end{definition}

In the above definition, $\TWCG$ and $\TCG$ stand for \emph{thin weak club-guessing} 
and \emph{thin club-guessing}, respectively.

\begin{lemma} 
Let $\tau<\o_1$ be a nonzero ordinal. Then $$\BFA_{\al_1}((\o^\tau\mbox{-}\PR)\cap\STP\cap\,^\o\o\textsf{-bounding})$$ 
implies the failure of $\tau'$-$\TWCG$ for every $\tau'$ such that $\tau<\tau'<\o_1$. 
\end{lemma}

\begin{proof}
Let us consider the following natural forcing $\mtcl P_{\vec C}$ for killing an instance 
$$\vec C=(C_\d\,:\,\d=\o^\tau\cdot\eta\mbox{ for some nonzero }\eta<\o_1)$$ of $\tau'$-$\TWCG$: 
$\mtcl P_{\vec C}$ is the set, ordered by reverse end-extension, of countable closed subset $c$ of $\o_1$ 
such that $\ot(C_\d\cap c)<\o^{\tau'}$ for every $\d\in \dom(\vec C)$. It is proved in \cite{AFMS} that 
$\mtcl P_{\vec C}$ is $\o^\tau$-proper, does not add new reals, and adds a club $C\sub\o_1$ such that 
$\ot(C_\d\cap C)<\o^{\tau'}$ for every $\d\in\dom(\vec C)$. Hence, it only remains to prove that $\mtcl P_{\vec C}$ 
preserves Suslin trees. This can be shown by an argument similar to the main argument 
in the proof in \cite{Miyamoto-Yorioka} that $\MRP$-posets preserve Suslin tree. 
We present the argument here for the reader's convenience.

Suppose $U$ is a Suslin tree, $\dot A$ is a $\mtcl Q$-name for a maximal antichain of $U$, and $N$ is a 
countable elementary submodel of some large enough $H_\t$ containing $U$, $\dot A$, and all other relevant objects. 
Let $\d=N\cap\o_1$. As in  the last part of the proof of Lemma \ref{psiAC-measurable}, let $(u_n)_{n<\o}$ enumerate 
all nodes in  $U$ of height $\d$. Given a condition $c\in\mtcl P_{\vec C}\cap N$, we aim to build an 
$(N, \mtcl P_{\vec C})$-generic sequence $(c_n)_{n<\o}$ of conditions in $N$ extending $c$  such that for every $n$ 
there is some $v\in U$ below $u_n$ such that $c_{n+1}$ forces $v\in \dot A$. We will make sure that 
$C_\d\cap\bigcup_{n<\o}c_b\sub c$, which will guarantee that $c^\ast = \bigcup_{N<\o}c_n\cup\{\d\}\in\mtcl P_{\vec C}$. 
But this will be enough, as then $c^\ast$ will be an extension of $c$ in $\mtcl P_{\vec C}$ forcing $\dot A\sub U\cap N$. 

It thus remains to show how to find $c_{n+1}$ given $c_n$. Working in $N$, we may first fix some countable 
$M\preccurlyeq H_{\chi}$ (for some large enough $\chi$) containing $c_n$ and all other relevant objects 
(including some relevant dense set $D\sub\mtcl P_{\vec C}$ that we need to meet), and such that 
$[\eta,\,\d_M]\cap C_\d=\emptyset$, where $\d_M=M\cap\o_1$. In order to find $M$, we first consider a strictly 
$\sub$-increasing and continuous sequence $(M_\n)_{\n<\o_1}\in N$  of elementary submodels containing 
all relevant objects. Since $(M_\nu\cap\o_1)_{\n<\d}$ is a club of $\d$ of order type $\d$ and 
$\ot(C_\d)=\o^{\tau'}<\d$, we can then find some $\n<\d$ such that $M=M_\n$ is as desired. 
Now, working in $M$, we may, first,  extend $c_n$  to a condition $c_n'$ such that $\max(c_n')>\eta$ and 
$[\max(c_n),\,\eta]\cap c_n'=\emptyset$, and then extend $c_n'$ to a condition $c_n''$ in $D$.  
Let now $\bar u$ be the unique node in $U$  below $u_n$ of height $\d_M$. 
Since $U$ is a Suslin tree, we have that $u_n$ is totally $(U, M)$-generic. Also, the set $E\in N$ of $u\in U$ for which there 
is some $v\in U$ below $u$ and some $\bar c\in\mtcl P_{\vec C}$ extending $c_n''$ and forcing that $v\in\dot A$ is 
dense in $U$. It follows that we may find  some $u\in E\cap M$ below $u_n$, as witnessed by some 
$\bar c\in \mtcl P_{\vec C}\cap N$ and some $v\in U\cap M$. But then we may let $c_{n+1}=\bar c$. 
\end{proof}

\begin{lemma}
Suppose $\delta$ is an inaccessible cardinal such that $V_\delta\prec V$. 
Let $\Gamma$ be any absolutely well-behaved class, defined with a parameter in $V_\delta$, and
 such that $\lambda_\Gamma=\omega_1$. 

 Let $\tau<\o_1$ be a nonzero ordinal, and suppose $\Gamma\sub \o^\tau\mbox{-}\SP$ holds in any generic extension by countably closed forcing.  
If $\BCFA(\Gamma)$ holds, then so does $\tau$-$\TCG$.
\end{lemma}

\begin{proof}
By a result in \cite{Zapletal}, there is a countably closed forcing notion adding a $\tau$-$\TCG$-sequence. 
The rest of the argument is as in the proof of Lemma \ref{consespfa+++} (and subsequent lemmas), using 
the preservation of $\tau$-$\TCG$-sequences by any $\o^\tau$-semiproper forcing, which is a completely  standard fact.
\end{proof}

\subsection{Bounded category forcing axioms and stronger large cardinal assumptions}

As we know, the theory of $H_{\omega_2}$ given by the bounded category forcing axioms we have explored in this section is invariant, in the presence of reasonable large cardinals, relative to extensions via members in the corresponding class forcing $\CFA(\Gamma)$.
In this final subsection we show that the combinatorial theory of $H_{\omega_2}$ given by these axioms is nevertheless sensitive to additional background large cardinal assumptions.\footnote{See also Remark \ref{remark-on-cons-strength}.}  

Recall that $\bm{\delta}^1_2$ is the supremum of the the set of lengths of $\bm{\Delta}^1_2$-definable  
pre-well-orderings on $\mtbb R$. 

\begin{proposition}\label{consespfa+++extra}
Suppose $\delta$ is an inaccessible cardinal such that $V_\delta\prec V$. Suppose there is a proper class of Woodin cardinals.
Let $\Gamma$ be any absolutely well-behaved class, defined with a parameter in $V_\delta$, such that $\lambda_\Gamma=\omega_1$ and $\Gamma\sub \PR$. Suppose $\BCFA(\Gamma)$ holds. Then $\bm{\delta}^1_2<\o_2$.
 \end{proposition}

%\begin{enumerate}
%\item If 
%$\Gamma\sub \PR$ holds after forcing with $\Add(\omega_1, 1)$, then Club Bounding fails.
%\item  If there is a proper class of Woodin cardinals, then $\bm{\delta}^1_2<\o_2$.
%\end{enumerate}
%\end{lemma}

\begin{proof}
By a result of Neeman and Zapletal \cite{Neeman-Zapletal}, the existence of  a proper class of Woodin cardinals yields that if $\mtcl P$ 
is a proper poset and $G$ is $\mtcl P$-generic over $V$, then the identity on $L(\mtbb R)^{V}$ is an elementary 
embedding between $L(\mtbb R)^{V}$ and  $L(\mtbb R)^{V[G]}$. Since $\Gamma_\delta$ is proper, it follows that  
$V^{\Gamma_\delta}\models\bm{\delta}^1_2<\o_2$. Since `$\bm{\delta}^1_2<\o_2$'  is expressible 
over $H_{\o_2}$, we then have that $V \models\bm{\delta}^1_2<\o_2$ by the usual absoluteness argument. 
\end{proof}

Recall that the Strong Reflection Principle ($\SRP$) is the following assertion: For every set $X$ such that 
$\o_1\sub X$ and every $S\sub [X]^{\al_0}$ there is a strong reflecting sequence $(x_i)_{i<\o_1}$ for $S$, i.e., 
$x_i\in [X]^{\al_0}$, $(x_i)_{i<\o_1}$ is strictly $\sub$-increasing and $\sub$-continuous, and for all $i$, $x_i\notin S$ if 
and only if there is no $y\in S$ such that $x_i\sub y$ and $y\cap\o_1 =x_i\cap\o_1$. 

\begin{lemma}\label{sspomegaomegaboundingvsdelta12} 
 $\FA_{\al_1}(\SP\cap \STP\cap\,^\o\o\mbox{-bounding})$ implies $\bm{\delta}^1_2=\o_2$. 
\end{lemma} 

\begin{proof} 
We have that  $\FA_{\al_1}(\SP\cap \STP\cap\,^\o\o\mbox{-bounding})$ implies $\SRP$ since, given 
$S\sub [X]^{\al_0}$, the standard forcing for adding a strong reflecting sequence for $S$ is semiproper, does not add reals, 
and preserves Suslin trees, where the last fact follows from an argument as in the final part of the proof of 
Lemma \ref{psiAC-measurable}. Also, $\SRP$ implies $\lnot\Box_\kappa$, for every cardinal $\k\geq\o_1$, 
and hence implies that the universe is closed under sharps. Since it also implies the saturation of $\NS_{\o_1}$, 
by a classical result of Woodin (\cite{WoodinBOOK}) it implies $\bm{\delta}^1_2=\o_2$. 
\end{proof}

\begin{corollary}
Suppose $\delta$ is a supercompact cardinal such that $V_\delta\prec V$. 
Let $\Gamma$ be any absolutely well-behaved class, defined with a parameter in $V_\delta$, such that $\lambda_\Gamma=\omega_1$ and such that $\SP\cap \STP\cap\,^\o\o\mbox{-bounding}\sub\Gamma$ forces in every generic extension via any member from $\Gamma$. Suppose $\BCFA(\Gamma)$ holds. Then $\bm{\delta}^1_2=\o_2$.
\end{corollary}

\begin{proof}
We know that $\Gamma_\delta$ forces $\FA_{\al_1}(\Gamma)$ and therefore, by our assumption, it also forces $\FA_{\al_1}(\SP\cap \STP\cap\,^\o\o\mbox{-bounding})$. But then $V\models \bm{\delta}^1_2=\o_2$ by Lemma \ref{sspomegaomegaboundingvsdelta12} together with the usual absoluteness argument. 
\end{proof}

\begin{question} Does  $\FA_{\al_1}(\o\mbox{-}\SP)$ imply $\bm{\delta}^1_2=\o_2$?
\end{question}

\section{Appendix}\label{appendix}

We collect here a few results (without proofs) 
translating the approach to forcing and iterations via posets to that
done via complete boolean algebras; for details see the forthcoming~\cite{VIAAUDSTEBOOK}
or \cite{VIAAUDSTE13} (the latter is available on ArXiv).

Given a boolean algebra $\bool{B}$ and a prefilter $G$  on $\bool{B}$
(i.e., a family such that $\bigwedge F>0_\bool{B}$ for all finite $F\subseteq G$), we 
denote by $\bool{B}/_G$ the quotient algebra obtained using the ideal 
$J=\bp{b:b\wedge c=0_{\bool{B}}\text{ for all $c\in G$}}$. $\bool{B}^+$ denotes the positive elements
of $\bool{B}$ and $\dot{G}_{\bool{B}}=\bp{\ap{\check{b},b}:b\in\bool{B}}$ is the canonical $\bool{B}$--name for the $V$--generic filter.

\begin{theorem}
Assume $(P,\leq)$ is a partial order. Let $\RO(P)$ denote the family of regular open sets
for the topology $\tau_P$ whose open sets are the downward closed subsets of $P$. The function $i_p:P\to \RO(P)$ given by
$i_P:p\mapsto \Reg{\downarrow p}$ (where $\downarrow p=\bp{q\in P:q\leq p}$ and $\Reg{A}$ is the interior of the closure of $A$ for the topology generated by the sets $\downarrow p$) is
an order and incompatibility preserving map of $(P,\leq)$ into $\RO(P)^+$ with image dense in $\RO(P)^+$,
and is such that
$p\Vdash_P\phi(\tau_1,\ldots, \tau_n)$ if and only if $i_P(p)\leq\Qp{\phi(\tau_1, \ldots, \tau_n)}_{\RO(P)}$.
\end{theorem}

\subsection{Two-step iterations}

Given $\dot{C}$, a $\bool{B}$-name for a forcing notion,
by $\bool{B}\ast\dot{C}$ we intend the two-step iteration as defined in 
\cite[Section VIII.5]{KUNEN} (or, equivalently---when $\dot{\bool{C}}$ is 
a $\bool{B}$-name for a cba---according to \cite[Lemma 16.3]{JECH}, which is actually more in line with the approach pursued in \cite{VIAAUDSTEBOOK}
or \cite{VIAAUDSTE13}). To simplify notation we also feel free to confuse (in some occasions) a partial order with its boolean completion.

\begin{theorem}\label{qIso}
	If $i: \bool{B} \rightarrow \bool{C}$ is an injective complete homomorphism of complete boolean algebras, then 
	$$\bool{B} \ast (\bool{C}/_{i[\dot{G}_{\bool{B}}]}) \cong \bool{C}.$$
	
	Conversely, if $\dot{\bool{Q}}\in V^{\bool{B}}$ is a $\bool{B}$-name for a 
	complete boolean algebra and $G$ is $V$-generic for $\bool{B}$, then
	\[
	(\bool{B}\ast\dot{\bool{Q}})/_{i_{\bool{B}\ast\dot{\bool{C}}}[G]}\cong\dot{\bool{Q}}_G.
	\]
\end{theorem}

\begin{proposition}\label{prop:embfromembnames}
	Let $\dot{\bool{C}}_0$, $\dot{\bool{C}}_1$ be $\bool{B}$-names for complete boolean algebras, 
	and let $\dot{k}$ be a $\bool{B}$-name for a complete homomorphism from 
	$\dot{\bool{C}}_0$ to $\dot{\bool{C}}_1$. Then 
	there is a complete homomorphism 
	$i : ~ \bool{B} \ast \dot{\bool{C}}_0 \to \bool{B} \ast \dot{\bool{C}}_1$ such that
	\[
	\Qp{\dot{k}=i/_{\dot{G}_{\bool{B}}}}=\1_{\bool{B}}.
	\]
	Moreover if $\dot{k}$ is a $\bool{B}$ name for an injective homomorphism, $i$ is injective.
	
	\smallskip
	
	In the following diagrams we assume
	%picture assumes 
	$G$ is $V$-generic for $\bool{B}$.
		\[
	\begin{tikzpicture}[xscale=1.5,yscale=-1.5]
	\node (VG) at (-6.8, 0.5) {$V[G]:$};

	\node (Q0a) at (-6, 0) {$(\dot{\bool{C}}_0)_G$};
	\node (Q1a) at (-6, 1) {$(\dot{\bool{C}}_1)_G$};
	\path (Q0a) edge [->]node [auto] {$\scriptstyle{(\dot{k})_G}$} (Q1a);

	\node (V) at (-3.5, 0.5) {$V:$};
	\node (B) at (-3, 0) {$\bool{B}$};
	\node (Q0) at (-2, 0) {$\bool{B}*\dot{\bool{C}}_0$};
	\node (Q1) at (-2, 1) {$\bool{B}*\dot{\bool{C}}_1$};
	\path (Q0) edge [->]node [auto] {$\scriptstyle{i}$} (Q1);
	\path (B) edge [->]node [auto] {$\scriptstyle{i_0}$} (Q0);
	\path (B) edge [->]node [auto,swap] {$\scriptstyle{i_1}$} (Q1);
	
	\node (VG1) at (-0.3, 0.5) {$V[G]:$};
	
   \node (Q02) at (1, 0) {$(\bool{B} * \dot{\bool{C}_0})/_{i_0[{G}_{\bool{B}}]}$};
    \node (Q12) at (1, 1) {$(\bool{B} * \dot{\bool{C}_1})/_{i_1[{G}_{\bool{B}}]}$};
    \node (Q03) at (3, 0) {$(\dot{\bool{C}}_0)_G$};
	\node (Q13) at (3, 1) {$(\dot{\bool{C}}_1)_G$};
	\path (Q03) edge [->]node [auto,swap] {$\scriptstyle{(\dot{k})_G}$} (Q13);

	\path (Q02) edge [->]node [auto] {$i/_{G_{\bool{B}}}$} (Q12);
	\path (Q02) edge [->]node [auto] (a) [sloped] {$\cong$} (Q03);
	\path (Q12) edge [->]node [auto,swap] (b) [sloped] {$\cong$} (Q13);
	%\path (a) edge[draw=none] node  [] {$\cong$} (b);
	\end{tikzpicture}
	\]
%	
%	
%	\[
%	\begin{tikzpicture}[xscale=1.5,yscale=-1.5]
%	
%	\node (Q0a) at (-2, 0) {$\dot{\bool{C}}_0$};
%	\node (Q1a) at (-2, 1) {$\dot{\bool{C}}_1$};
%	\path (Q0a) edge [->]node [auto] {$\scriptstyle{\dot{k}}$} (Q1a);
%	
%	
%	\node (B) at (0, 0) {$\bool{B}$};
%	\node (Q0) at (1, 0) {$(\dot{\bool{C}}_0)_{G_{\bool{B}}}$};
%	\node (Q1) at (1, 1) {$(\dot{\bool{C}}_1)_{G_{\bool{B}}}$};
%   \node (Q02) at (3, 0) {$(\bool{B} * \dot{\bool{C}_0})/_{i[{G}_{\bool{B}}]}$};
%    \node (Q12) at (3, 1) {$(\bool{B} * \dot{\bool{C}_1})/_{i[{G}_{\bool{B}}]}$};
%	\path (B) edge [->]node [auto] {$\scriptstyle{i_0}$} (Q0);
%	\path (B) edge [->]node [auto,swap] {$\scriptstyle{i_1}$} (Q1);
%	\path (Q0) edge [->]node [auto] {$\scriptstyle{\dot{k}_{G_B}}$} (Q1);
%	\path (Q02) edge [->]node [auto,swap] {$i/_{G_{\bool{B}}}$} (Q12);
%	\path (Q0) edge[draw=none] node (a) [sloped] {$\cong$} (Q02);
%	\path (Q1) edge[draw=none] node (b) [sloped] {$\cong$} (Q12);
%	\path (a) edge[draw=none] node  [] {$\cong$} (b);
%	\end{tikzpicture}
%	\]
\end{proposition}

	\begin{proposition}\label{prop:embfromembnames2}
	Let $G$ be $V$-generic for $\bool{B}$, $I$ be its dual ideal, and 
	$i_j:\bool{B}\to \bool{C}_j$ be complete homomorphisms for $j=0,1$.
	
	Assume $\bool{C}_0/_{i_0[G]}$ and  $\bool{C}_1/_{i_1[G]}$ are isomorphic complete boolean algebras in $V[G]$.
	
	\[
	\begin{tikzpicture}[xscale=1.5,yscale=-1.2]
	\node (VG) at (0.5, -0.5) {$V[G]:$};
	\node (V) at (-2.5, -0.5) {$V:$};
	\node (B) at (-2, 0) {$\bool{B}$};
	\node (Q0) at (-1, 0) {$\bool{C}_0$};
	\node (Q1) at (-1, 1) {$\bool{C}_1$};
	\path (B) edge [->]node [auto] {$\scriptstyle{i_0}$} (Q0);
	\path (B) edge [->]node [auto,swap] {$\scriptstyle{i_1}$} (Q1);
	\node (Q02) at (1.5, 0) {$\bool{C}_0/_{i_0[G]}$};
	\node (Q12) at (1.5, 1) {$\bool{C}_1/_{i_1[G]}$};
	\path (Q02) edge [->]node [auto] {$k\, \cong$} (Q12);
	\end{tikzpicture}
	\]
	
	Then 
	$\bool{C}_0\restriction i_0(b)$ and  $\bool{C}_1\restriction i_1(b)$ are isomorphic in $V$ for some $b\in G$
	and $k\cong l/_G$.
	
	\[
	\begin{tikzpicture}[xscale=1.5,yscale=-1.2]
\node (VG) at (-5, -0.5) {$V[G]:$};
	\node (Q02) at (-2.5, 0) {$\bool{C}_0 \upharpoonright i_0(b)/_{i_0[G]}$};
	\node (Q12) at (-2.5, 1.5) {$\bool{C}_1 \upharpoonright i_1(b)/_{i_1[G]}$};
	\node (Q0) at (-4, 0) {$\bool{C}_0/_{i_0[G]}$};
	\node (Q1) at (-4, 1.5) {$\bool{C}_1/_{i_1[G]}$};
	\path (Q02) edge [->]node[auto]  {$l/_G\, \cong$} (Q12);
	\path (Q0) edge[draw=none] node  [sloped] {$=$} (Q02);
	\path (Q1) edge[draw=none] node  [sloped] {$=$} (Q12);
	\path (Q0) edge [->]node[auto,swap] {$k\, \cong$} (Q1);

	\node (V) at (-0.5, -0.5) {$V:$};
	\node (B) at (0, 0) {$\bool{B}$};
	\node (Q0) at (1, 0) {$\bool{C}_0$};
	\node (Q1) at (1, 1.5) {$\bool{C}_1$};
	\path (B) edge [->]node [auto] {$\scriptstyle{i_0}$} (Q0);
	\path (B) edge [->]node [auto,swap] {$\scriptstyle{i_1}$} (Q1);
	\node (Q02) at (2.5, 0) {$\bool{C}_0 \upharpoonright i_0(b)$};
	\node (Q12) at (2.5, 1.5) {$\bool{C}_1 \upharpoonright i_1(b)$};
	\path (Q02) edge [->]node[auto]  {$l\, \cong$} (Q12);
	\path (Q0) edge[->] node  [auto] {$\scriptstyle{\sf{rest}}$} (Q02);
	\path (Q1) edge[->] node  [auto] {$\scriptstyle{\sf{rest}}$} (Q12);
	\end{tikzpicture}
	\]

\end{proposition}

	\begin{lemma}\label{EmbTwoStepPI1pers}
	Assume $\Gamma$ is a definable class of forcings.
	Let $\bool{B}$, $\bool{C}_0$, $\bool{C}_1$ be complete boolean algebras, 
	and let $G$ be any $V$-generic filter for $\bool{B}$. 
	Let $i_0$, $i_1$, $j$ be $\Gamma$-correct complete homomorphisms in $V$
	forming a commutative diagram of injective complete homomorphisms of complete boolean algebras
	as in the following picture:
	\[
	\begin{tikzpicture}[xscale=1.5,yscale=-1.2]
		\node (B) at (0, 0) {$\bool{B}$};
		\node (Q0) at (1, 0) {$\bool{C}_0$};
		\node (Q1) at (1, 1) {$\bool{C}_1$};
		\path (B) edge [->]node [auto] {$\scriptstyle{i_0}$} (Q0);
		\path (B) edge [->]node [auto,swap] {$\scriptstyle{i_1}$} (Q1);
		\path (Q0) edge [->]node [auto] {$\scriptstyle{j}$} (Q1);
	\end{tikzpicture}
	\]
%	Moreover assume that $\bool{C}_0/i_0[G]$ is in $\Gamma$ and
%	\[
%	\Qp{\bool{C}_1/j[\dot{G}_{\bool{C}_0}]\mbox{ is $\Gamma$-correct}}=\1_{\bool{C}_0}.
%	\]
	Then in $V[G]$, $j/_G: \bool{C}_0/_G \to \bool{C}_1/_G$ is still a  $\Gamma^{V[G]}$-correct complete
	homomorphism. 
\end{lemma}

	\begin{proposition}\label{lem:PI1pers-2}
	Assume $\Gamma$ is a definable class of forcings. 	Let
	$G$ be $V$-generic for some complete boolean algebra $\bool{B}$.
	Assume $k:\bool{B}\to\bool{R}$ is a $\Gamma$-correct homomorphism in $V$ and
	$h:\bool{R}/_{k[G]}\to \bool{Q}$ is a $\Gamma$-correct homomorphism in $V[G]$.
%	for $j=0,1$ as in the following pictures:
	\[
			\begin{tikzpicture}
						\node (B) at (0, 0) {$\bool{B}$};
						\node (Q0) at (2, 0) {$\bool{R}$};
						%\node (Q1) at (0, -2) {$\bool{Q}_1$};
						\node (V) at (-1, 0) {$V:$};
						\path (B) edge [->]node [auto] {$\scriptstyle{k}$} (Q0);
						%\path (B) edge [->]node [auto,swap] {$\scriptstyle{k_1}$} (Q1);
		
						\node (Q0G) at (8, 0) {$\bool{R}/_{k[G]}$};
						%\node (Q1G) at (6,-2) {$\bool{Q}_1/_{k_1[G]}$};
						\node (QG) at (8, -2) {$\bool{Q}$};
						\node (V[G]) at (5, 0) {$V[G]:$};
						\path (Q0G) edge [->]node [auto] {$\scriptstyle{h}$} (QG);
						%\path (Q1G) edge [->]node [auto,swap] {$\scriptstyle{h_1}$} (QG);
			\end{tikzpicture}
	\]
	
	Then there are in $V$:
	\begin{itemize} 
	\item
	$\bool{C}\in \Gamma$, 
	\item
	a $\Gamma$-correct homomorphism 
	$l:\bool{B}\to\bool{C}$,
	\item
	a $\Gamma$-correct homomorphism
	$\bar{h}:\bool{R}\to\bool{C}$,
%	\item
%	an isomorphism $i:\bool{C}\to\bool{C}$
	\end{itemize}
	such that:
	\begin{itemize}
	\item
	$\bool{Q}$ is isomorphic to $\bool{C}/_{l[G]}$  in $V[G]$,
	\item
	$\bar{h}/_G\cong h$ (modulo the isomorphism of $\bool{Q}$ with $\bool{C}/_{l[G]}$) holds in $V[G]$,
	\item
	$\bar{h}\circ k=l$ holds in $V$,
	\item
	$0_{\bool{C}}\not\in  l[G]$.
	\end{itemize}
%	so that the above diagrams are completed as follows:
%	\begin{aMnote}{Manca una freccia che colleghi $\bool{Q}$ con $\bool{C}/l\qp{G}$ ed il simbolo di isomorfismo.}
	\[
				\begin{tikzpicture}
							\node (B) at (0, 0) {$\bool{B}$};
							\node (Q0) at (3, 0) {$\bool{R}$};
							%\node (Q1) at (0, -3) {$\bool{Q}_1$};
							\node (C) at (3,-3) {$\bool{C}$};
							\node (V) at (-1, 0) {$V:$};
							\path (B) edge [->]node [auto] {$\scriptstyle{k}$} (Q0);
							%\path (B) edge [->]node [auto,swap] {$\scriptstyle{k_1}$} (Q1);
							\path (Q0) edge [->, dashed]node [auto] {$\scriptstyle{\bar{h}}$} (C);
							%\path (Q1) edge [->, dashed]node [auto,swap] {$\scriptstyle{l_1}$} (C);
							\path (B) edge [->, dashed ]node [auto,swap] {$\scriptstyle{l}$} (C);
				
							\node (BG) at (7, 0) {$\2\cong \bool{B}/G$};
							\node (Q0G) at (10, 0) {$\bool{R}/_{k[G]}$};
							%\node (Q1G) at (6, -3) {$\bool{Q}_1/_{k_1[G]}$};
							\node (CG) at (10, -3) {$\bool{C}/l[G]\cong \bool{Q}$};
							%\node (QG) at (7.5, -1.5) {$\bool{Q}$};
							\node (V[G]) at (5, 0) {$V[G]:$};
							\path (Q0G) edge [->]node [auto] {$\scriptstyle{\bar{h}/_G\cong h}$} (CG);
							%\path (Q1G) edge [->, dashed]node [auto,swap] {$\scriptstyle{l_1/_G}$} (CG);
							\path (BG) edge [->, dashed]node [auto,swap] {$\scriptstyle{k/_G}\cong \id$} (Q0G);
							\path (BG) edge [->, dashed]node [auto,swap] {$\scriptstyle{l/_G}\cong \id$} (CG);
							%\path (CG) edge [->, dashed]node [auto,swap] {i} (QG);
							%\path (Q0G) edge [->]node [auto,swap] {$\scriptstyle{h_0}$} (QG);
							%\path (Q1G) edge [->]node [auto] {$\scriptstyle{h_1}$} (QG);
							
				\end{tikzpicture}
	 \]
%	 \end{aMnote}
%	\end{enumerate}
	\end{proposition}

\subsection{Limit length iterations}

\begin{theorem}  \label{eRetrProp}
	Assume $i: \bool{B} \to \bool{C}$ is a complete injective homomorphism of complete boolean algebra.
	Then the following holds:
	\begin{enumerate}
	\item
	$i$ has an adjoint $\pi_i:\bool{C} \to \bool{B}$ defined by $\pi_i(c)=\bigwedge\bp{b\in\bool{B}:i(b)\geq c}$
	\item
	For any $b \in \bool{B}$ and $c,d \in \bool{C}$, we have that:
	\begin{enumerate}[(a)]
		\item $\pi_i$ is order preserving; %and is defined by $c\mapsto\bigwedge \bp{b \in \bool{B}: ~ i(b) \geq c}$;
		\item $(\pi_i\circ i)(b)= b$, hence $\pi_i$ is surjective;
		\item \label{eRPComp} 
		$(i\circ\pi_i)(c)\geq c$; in particular, 
		$\pi_i$ maps $\bool{C} ^+$ to $\bool{B}^+$;
		\item \label{eRPJoins} $\pi_i$ preserves joins, i.e., $\pi_i (\bigvee X)= \bigvee \pi_i[X]$ for all 
		$X \subseteq \bool{C}$ for which the supremum $\bigvee X$ exists in $\bool{C}$;
		\item $i(b)= \bigvee \{e : \pi_i (e) \leq b \}$;
		\item \label{eRPHomo} $\pi_i (c \wedge i(b)) = \pi_i(c)\wedge b = \bigvee \{ \pi_i(e): e \leq c, \pi_i(e) \leq b\}$;
		%hence $(i,\pi_i)$ forms an adjoint pair;
		\item \label{eRPMeets} $\pi_i$ preserves neither meets nor complements whenever 
		$i$ is not surjective, but $\pi_i(d \wedge c) \leq \pi_i(d) \wedge \pi_i(c)$ and 
		$\pi_i(\neg c) \geq \neg \pi_i(c)$.\qedhere
	\end{enumerate}
	\end{enumerate}
	%Hence $(i,\pi_i)$ is a residuated pair.
\end{theorem}

\begin{definition} 
 $\FFF=\{i_{\alpha \beta}:\bool{B}_\alpha\to \bool{B}_\beta\,\mid\,\alpha \leq \beta < \lambda\}$ is a \emph{complete iteration system} of complete boolean algebras iff for all $\alpha \leq \beta \leq \gamma < \lambda$:
	\begin{enumerate}
		\item $\bool{B}_\alpha$ is a complete boolean algebra and $i_{\alpha \alpha}$ is the identity on it;
		\item $i_{\alpha \beta}$ is a complete injective homomorphism and 
		$\pi_{\alpha\beta}:\bool{B}_\beta\to\bool{B}_\alpha$, given by 
		$c\mapsto\bigwedge\bp{b\in\bool{B}:i_{\alpha\beta}(b)\geq c}$, is its associated 
		adjoint;
		\item $i_{\beta \gamma}\circ i_{\alpha \beta}=i_{\alpha \gamma}$.
 \end{enumerate}

Let $\FFF$ be a complete iteration system of length $\lambda$. Then:

\begin{itemize}
\item
 The \emph{inverse limit} of the iteration is
	\[
	\varprojlim(\FFF) = \bp{ f \in \prod_{\alpha < \lambda} \bool{B}_\alpha : ~ \forall \alpha \forall \beta > \alpha\, ~ (\pi_{\alpha \beta}\circ f)(\beta) = f(\alpha) }
	\]
	and its elements are called \emph{threads}. 
\item
	The \emph{direct limit} is
	\[
	\varinjlim(\FFF) = \bp{ f \in \varprojlim(\FFF) : ~ \exists \alpha \forall \beta > \alpha ~ f(\beta) = i_{\alpha\beta}(f(\alpha)) }
	\]
	and its elements are called \emph{constant threads}. 
	The support of a constant thread, $\supp(f)$, is 
	the least $\alpha$ such that $(i_{\alpha\beta}\circ f)(\alpha)=f(\beta)$ for all $\beta\geq\alpha$. 
\item
	The \emph{revised countable support limit} 
	is\footnote{This definition of revised contable support limit is due to Donder and Fuchs; it is not clear
	whether it is in any sense equivalent to Miyamoto or Shelah's definition of rcs-limit. It is nonetheless effective
	to prove the main results about semiproper iterations: specifically, the 
	forthcoming~\cite{VIAAUDSTEBOOK} or \cite{VIAAUDSTE13} (which is available on ArXiv) give complete
	proofs of the preservation of semiproperness through RCS--iterations. However, in this paper
	we also want to preserve properties stronger than semiproperness. The proofs of 
	such preservation results are not given in
	in~\cite{VIAAUDSTEBOOK} or in~\cite{VIAAUDSTE13}; hence we refer the reader to Miyamoto's or Shelah's
	iteration theorems for the preservation through RCS--limits (according to their definitions) of these stronger properties.}
	\[
	\rcslim(\FFF) = \bp{ f \in \varprojlim(\FFF) : ~ f \in \varinjlim(\FFF) \vee \exists \alpha ~ f(\alpha) \Vdash_{\bool{B}_\alpha} \cof(\check{\lambda}) = \check{\omega} }.
	\]
\end{itemize}	
\end{definition}

\begin{theorem}
Assume $\bp{P_\alpha,\dot{Q}_\beta:\alpha\leq\lambda,\beta<\lambda}$ is an iteration of posets.
Let $i_{\alpha\beta}:\RO(P_\alpha)\to\RO(P_\beta)$ be the complete homomorphism induced
by the natural inclusion of $P_\alpha$ into $P_\beta$. Then:
\begin{itemize}
\item 
$\FFF=\bp{i_{\alpha,\beta}:\alpha\leq\beta<\lambda}$ is an iteration system of  complete injective homomorphisms
of complete boolean algebras.
\item
If $\lambda=\omega$,  $\invlim(\FFF)$ is isomorphic to the boolean completion of the full limit of
$\bp{P_\alpha,\dot{Q}_\beta:\alpha\leq\lambda,\beta<\lambda}$.
\item For any regular $\lambda$,
$\dirlim(\FFF)$ is isomorphic to the boolean completion of the direct limit of
$\bp{P_\alpha,\dot{Q}_\beta:\alpha\leq\lambda,\beta<\lambda}$.
\end{itemize}
\end{theorem}

\begin{theorem}[Baumgartner] \label{iBaumgartner}
	Let $\lambda$ be a regular cardinal and
	$\FFF = \{i_{\alpha \beta} : \alpha \leq \beta < \lambda\}$ be an iteration system such that 
	$\bool{B}_\alpha$ is ${<}\lambda$-cc for all 
	$\alpha$ and $S = \bp{\alpha: ~ \bool{B}_\alpha \cong \RO(\varinjlim(\FFF \rest{\alpha}))}$ 
	is stationary. Then $\varinjlim(\FFF)$ is ${<}\lambda$-cc.
\end{theorem}

These results suffice to prove all the needed equivalences (used in this paper)
between results on forcing and iterations proved in the language of 
partial orders and their corresponding formulation in terms of complete boolean algebras.

	\bibliographystyle{amsplain}
	\bibliography{Biblio}

\providecommand{\bysame}{\leavevmode\hbox to3em{\hrulefill}\thinspace}
\providecommand{\MR}{\relax\ifhmode\unskip\space\fi MR }
\providecommand{\MRhref}[2]{%
  \href{http://www.ams.org/mathscinet-getitem?mr=#1}{#2}
}
\providecommand{\href}[2]{#2}
\begin{thebibliography}{10}

\bibitem{AFMS}
D.~Asper\'{o}, S.D. Friedman, Mota M.A., and M.~Sabok, \emph{Bounded forcing
  axioms and baumgartner's conjecture}, Annals of Pure and Applied Logic
  \textbf{164} (2013), no.~12, 1178--1186.

\bibitem{ASPSCH(*)}
D.~Asper\'o and R.~Schindler, \emph{$\mathsf{MM}^{++}$ implies $(*)$},
  https://arxiv.org/abs/1906.10213, 2019.

\bibitem{Aspero-Welch}
D.~Asper\'{o} and P.D. Welch, \emph{Bounded martin's maximum, weak erd\"{o}s
  cardinals, and {$\psi_{AC}$}}, The Journal of Symbolic Logic \textbf{67}
  (2002), no.~3, 1141--1152.

\bibitem{VIAAUDSTEBOOK}
Giorgio Audrito, Raph\"ael Carroy, Silvia Steila, and Matteo Viale,
  \emph{Iterated forcing, category forcings, generic ultrapowers, generic
  absoluteness}, Book in preparation, 2017.

\bibitem{AUDVIA}
Giorgio Audrito and Matteo Viale, \emph{Absoluteness via resurrection}, J.
  Math. Log. \textbf{17} (2017), no.~2, 1750005, 36. \MR{3730561}

\bibitem{VIAAUD}
\bysame, \emph{Absoluteness via resurrection}, J. Math. Log. \textbf{17}
  (2017), no.~2, 1750005, 36. \MR{3730561}

\bibitem{CAIVEL06}
Andr\'es~Eduardo Caicedo and Boban Veli\v{c}kovi\'c, \emph{The bounded proper
  forcing axiom and well orderings of the reals}, Math. Res. Lett. \textbf{13}
  (2006), no.~2-3, 393--408. \MR{2231126}

\bibitem{Deiser-Donder}
O.~Deiser and Donder D., \emph{Canonical functions, non-regular ultrafilters
  and ulam’s problem on {$\omega_1$}}, The Journal of Symbolic Logic
  \textbf{68} (2003), no.~3, 713--739.

\bibitem{FOREMAGISHELAH}
M.~Foreman, M.~Magidor, and S.~Shelah, \emph{Martin's maximum, saturated
  ideals, and nonregular ultrafilters. {I}}, Ann. of Math. (2) \textbf{127}
  (1988), no.~1, 1--47. \MR{924672 (89f:03043)}

\bibitem{Goldstern}
M.~Goldstern, \emph{Tools for your forcing constructions}, Israel Mathematical
  Conference Proceedings, vol. 6, 1992, 307--362, 1992.

\bibitem{HAMJOH}
Joel~David Hamkins and Thomas~A. Johnstone, \emph{Resurrection axioms and
  uplifting cardinals}, Arch. Math. Logic \textbf{53} (2014), no.~3-4,
  463--485. \MR{3194674}

\bibitem{JECH}
T.~Jech, \emph{Set theory}, Springer Monographs in Mathematics, Springer,
  Ber\-lin, 2003, The third millennium edition, revised and expanded.
  \MR{1940513}

\bibitem{KUNEN}
K.~Kunen, \emph{Set theory}, Studies in Logic and the Foundations of
  Mathematics, vol. 102, North-Holland, Amsterdam, 1980, An introduction to
  independence proofs. \MR{597342}

\bibitem{LARSONBOOK}
Paul~B. Larson, \emph{The stationary tower}, University Lecture Series,
  vol.~32, American Mathematical Society, Providence, RI, 2004, Notes on a
  course by W. Hugh Woodin. \MR{2069032}

\bibitem{Miyamoto-Yorioka}
T.~Miyamoto and T.~Yorioka, \emph{Some results in the extension with a coherent
  suslin tree, part ii}, Kyoto University research information repository,
  1851, pp. 49--61, 2013.

\bibitem{Miyamoto}
Tadatoshi Miyamoto, \emph{Iterated forcing with {$^\omega\omega$}-bounding and
  semiproper preorders}, S\=urikaisekikenky\=usho K\=oky\=uroku (2001),
  no.~1202, 83--99, Axiomatic set theory (Japanese) (Kyoto, 2000). \MR{1855553}

\bibitem{miyamoto-suslin-sp}
\bysame, \emph{On iterating semiproper preorders}, J. Symbolic Logic
  \textbf{67} (2002), no.~4, 1431--1468. \MR{1955246}

\bibitem{MOO06}
J.~T. Moore, \emph{Set mapping reflection}, J. Math. Log. \textbf{5} (2005),
  no.~1, 87--97. \MR{2151584 (2006c:03076)}

\bibitem{Neeman-Zapletal}
Itay Neeman and Jind\v{r}ich Zapletal, \emph{Proper forcing and {$L(\Bbb R)$}},
  J. Symbolic Logic \textbf{66} (2001), no.~2, 801--810. \MR{1833479}

\bibitem{SPFA=MM}
S.~Shelah, \emph{Semiproper forcing axiom implies martin maximum but not
  $\mathsf{PFA}^+$}, The Journal of Symbolic Logic \textbf{52} (1987), no.~2,
  360--367.

\bibitem{SHEPRO}
Saharon Shelah, \emph{Proper and improper forcing}, second ed., Perspectives in
  Mathematical Logic, Springer-Verlag, Berlin, 1998. \MR{1623206 (98m:03002)}

\bibitem{Todor}
S.~Todor\v{c}evi\'{c}, \emph{Remarks on chain conditions in products},
  Compositio Mathematica \textbf{55} (1985), no.~3, 295--302.

\bibitem{Velickovic}
B.~Veli\v{c}kovi\'{c}, \emph{Forcing axioms and stationary sets}, Advances in
  Mathematics \textbf{94} (1992), no.~2, 256--284.

\bibitem{VIAVENMODCOMP}
G.~Venturi and M.~Viale, \emph{The model companions of set theory},
  https://arxiv.org/abs/1909.13372, 2019.

\bibitem{VIATAMSTII}
M.~Viale, \emph{Model companionship versus generic absoluteness $ii$},  (2020),
  arXiv:2003.07120v1.

\bibitem{VIAMM+++}
Matteo Viale, \emph{Category forcings, {$MM^{+++}$}, and generic absoluteness
  for the theory of strong forcing axioms}, J. Amer. Math. Soc. \textbf{29}
  (2016), no.~3, 675--728. \MR{3486170}

\bibitem{VIAMMREV}
\bysame, \emph{Martin's maximum revisited}, Arch. Math. Logic \textbf{55}
  (2016), no.~1-2, 295--317. \MR{3453587}

\bibitem{VIAUSAX}
\bysame, \emph{Useful axioms}, To appear in a special issue of ``IfCoLog
  Journal of Logics and their Applications'' dedicated to the winners of Kurt
  G\"odel fellowships, 2016.

\bibitem{VIAAUDSTE13}
Matteo Viale, Giorgio Audrito, and Silvia Steila, \emph{A boolean algebraic
  approach to semiproper iterations}, arXiv:1402.1714, 2014.

\bibitem{WoodinBOOK}
W.~Hugh Woodin, \emph{The axiom of determinacy, forcing axioms, and the
  nonstationary ideal}, revised ed., De Gruyter Series in Logic and its
  Applications, vol.~1, Walter de Gruyter GmbH \& Co. KG, Berlin, 2010.
  \MR{2723878}

\bibitem{Zapletal}
Jind\v{r}ich Zapletal, \emph{Transfinite open--point games}, Topology and its
  applications \textbf{111} (2001), no.~3, 289--297.

\end{thebibliography}
\end{document}